\documentclass[aos]{imsart}
\usepackage{amsmath}
\usepackage{amsthm}
\usepackage{amssymb}
\usepackage{dsfont}
\usepackage{footmisc}
\usepackage[authoryear,round]{natbib}
\usepackage[colorlinks=true,linkcolor=blue,citecolor=blue]{hyperref}
\usepackage{hypernat}
\usepackage{verbatim}
\usepackage{textcomp}
\usepackage{enumerate}
\usepackage{graphicx}
\usepackage{bbm}
\usepackage{array}
\usepackage{multirow}
\usepackage{tikz}
\usetikzlibrary{calc}

\newcolumntype{H}{>{\setbox0=\hbox\bgroup}c<{\egroup}@{}}

\newcommand{\R}{{\mathbb R}}

\newcommand{\E}{{\mathbb E}}
\renewcommand{\P}{{\mathcal P}}

\newcommand{\N}{{\mathbb N}}
\newcommand{\Z}{{\mathcal Z}}

\newcommand{\eps}{\varepsilon}

\newcommand{\X}{\mathcal X}

\DeclareMathOperator*{\argmax}{arg\,max}

\usepackage{xr}

\DeclareMathOperator{\Var}{Var}

\DeclareMathOperator{\trace}{tr}

\DeclareMathOperator{\sign}{sign}

\newtheorem{theorem}{Theorem}[section]

\newtheorem{lemma}[theorem]{Lemma}

\newtheorem{remark}[theorem]{Remark}
\newtheorem*{remark*}{Remark}
\newtheorem{example}[theorem]{Example}

\newtheorem*{condition*}{Condition}
\newtheorem*{definition*}{Definition}
\newtheorem{definition}{Definition}

\numberwithin{equation}{section}

\newcounter{rcnt}[section]

\newcommand{\rem}[1]{}

\newcounter{desccount}

\newcommand{\descref}[1]{\hyperref[#1]{#1}}

\begin{document}

\sloppy

\begin{frontmatter}

\title{Efficiency in local differential privacy}
\runauthor{Steinberger, L.}

\runtitle{Efficiency in local differential privacy}

\begin{aug}
%%%%%%%%%%%%%%%%%%%%%%%%%%%%%%%%%%%%%%%%%%%%%%%
%% Only one address is permitted per author. %%
%% Only division, organization and e-mail is %%
%% included in the address.                  %%
%% Additional information can be included in %%
\author[A]{\fnms{Lukas}~\snm{Steinberger}\ead[label=e1]{lukas.steinberger@univie.ac.at}}
%%%%%%%%%%%%%%%%%%%%%%%%%%%%%%%%%%%%%%%%%%%%%%
%% Addresses                                %%
%%%%%%%%%%%%%%%%%%%%%%%%%%%%%%%%%%%%%%%%%%%%%%
\address[A]{University of Vienna, Department of Statistics and OR\printead[presep={,\ }]{e1}}
\end{aug}

\begin{abstract}
We develop a theory of asymptotic efficiency in regular parametric models when data confidentiality is ensured by local differential privacy (LDP). Even though efficient parameter estimation is a classical and well-studied problem in mathematical statistics, it leads to several non-trivial obstacles that need to be tackled when dealing with the LDP case. Starting from a regular parametric model $\P=(P_\theta)_{\theta\in\Theta}$, $\Theta\subseteq\R^p$, for the iid unobserved sensitive data $X_1,\dots, X_n$, we establish local asymptotic mixed normality (along subsequences) of the model $$Q^{(n)}\P=(Q^{(n)}P_\theta^n)_{\theta\in\Theta}$$ generating the sanitized observations $Z_1,\dots, Z_n$, where $Q^{(n)}$ is an arbitrary sequence of sequentially interactive privacy mechanisms. This result readily implies convolution and local asymptotic minimax theorems. In case $p=1$, the optimal asymptotic variance is found to be the inverse of the supremal Fisher-Information $\sup_{Q\in\mathcal Q_\alpha} I_\theta(Q\P)\in\R$, where the supremum runs over all $\alpha$-differentially private (marginal) Markov kernels. We present an algorithm for finding a (nearly) optimal privacy mechanism $\hat{Q}$ and an estimator $\hat{\theta}_n(Z_1,\dots, Z_n)$ based on the corresponding sanitized data that achieves this asymptotically optimal variance.
\end{abstract}



\end{frontmatter}

\section{Introduction}

Consider iid data $X_1,\dots, X_n$ from a probability distribution $P_\theta$ belonging to a regular parametric model $\P = (P_\theta)_{\theta\in\Theta}$, with parameter space $\Theta\subseteq\R^p$, sample space $(\X,\mathcal F)$ and with Fisher-Information $I_\theta(\P)\in\R^{p\times p}$. If the data owners are to be protected by (non-interactive) \emph{local} differential privacy, the analyst who wants to estimate the unknown parameter $\theta$ does not get to see these original data but only observes sanitized versions $Z_1,\dots, Z_n$ such that individual $i$ was able to generate $Z_i$ using its sensitive data $X_i$ and a (publicly known) Markov kernel $Q$ that determines the conditional distribution $Z_i|X_i=x \thicksim Q(dz|x)$ and obeys the $\alpha$-differential privacy constraint
\begin{equation}\label{eq:QNI}
Q(A|x) \le e^\alpha\cdot Q(A|x'), 
\end{equation}
for all $A$ in a sigma algebra $\mathcal G$ on some sample space $\mathcal Z$ and for all $x,x'$ in the original sample space $\X$. Thus, the observed data $Z_1,\dots, Z_n$ are iid from $QP_\theta$, where $[QP_\theta](A) = \int_\X Q(A|x) P_\theta(dx)$. We will see below that the resulting statistical model $Q\P = (QP_\theta)_{\theta\in\Theta}$ is again regular with Fisher-Information denoted by $I_\theta(Q\P)$. Hence, under standard regularity conditions, the MLE in $Q\P$ is asymptotically distributed as $N(0,I_\theta(Q\P)^{-1})$ and it is well known to be efficient, i.e., this asymptotic covariance matrix is the smallest (in the partial Loewner ordering) that any (regular) estimator in the model $Q\P$ can achieve. However, we are actually free to choose the privacy mechanism $Q$, as long as it satisfies \eqref{eq:QNI}. It is therefore natural to ask, at least in the case $p=1$, what is the mechanism $Q$ that leads to the smallest asymptotic variance, or, in other words, to look for solutions of the optimization problem
\begin{equation}\label{eq:Qopt}
\sup_{Q\in\mathcal Q_\alpha} I_\theta(Q\P),
\end{equation}
where $\mathcal Q_\alpha := \mathcal Q_\alpha(\X) := \bigcup_{(\mathcal Z,\mathcal G)}\Big\{Q:\mathcal G\times\X\to[0,1] \Big| Q$ is a Markov kernel satisfying \eqref{eq:QNI}$\Big\}$ is the set of all $\alpha$-private marginal data release mechanisms. The problem \eqref{eq:Qopt} is challenging, because the constraint set is an infinite dimensional set of Markov kernels with even unspecified domain $\mathcal Z$ and the supremum may not be attained. Moreover, we are trying to \emph{maximize} a convex function over a convex domain and thus we expect to see many local maxima.
Furthermore, a (near) maximizer will typically depend on the unknown parameter $\theta$ and the resulting privacy mechanism, $Q_\theta$ say, is also unknown and can not be shared with the data owners to generate sanitized data $Z_i$, $i=1,\dots, n$.

In this paper we develop the theory of statistical efficiency for local differential privacy, extending the classical approach due to \citet{LeCam60}, \citet{Hajek70} and \citet{Jeganathan80}. We study regularity properties of locally private models, prove a general local asymptotic mixed normality result and derive a private convolution theorem. Finally, we prove exact attainability of the asymptotic lower bound. In particular, we develop an algorithm based on discretization of the model $\P$ and a linear program proposed by \citet{Kairouz16} to find a privacy mechanism (Markov kernel) $\hat{Q}_\theta$ such that $I_\theta(\hat{Q}_{\theta}\P)\approx \sup_{Q\in\mathcal Q_\alpha} I_\theta(Q\P)$. We employ a two-step (sequentially interactive) procedure that first consistently estimates $\theta$ in an $\alpha$-private way with an estimator $\tilde{\theta}_{n_1}$ based on a subsample $Z_1, \dots, Z_{n_1}$ generated by some preliminary privacy mechanism $Q_0$. The individuals in the second part of the sample then use $\hat{Q}_{\tilde{\theta}_{n_1}}$ to generate $Z_{n_1+1},\dots, Z_n$ in an $\alpha$-private way. The MLE $\hat{\theta}_{n_2}$ in the model $(\hat{Q}_{\tilde{\theta}_{n_1}}\P)^{n_2}$ is finally shown to be regular with asymptotic distribution $N(0,[\sup_{Q\in\mathcal Q_\alpha}I_\theta(Q\P)]^{-1})$ (see Figure~\ref{fig:Qn} for a graphical representation of this procedure). Notice that since this procedure is sequentially interactive rather than simply non-interactive (that is, the local randomization mechanism $\hat{Q}_{\tilde{\theta}_{n_1}}$ used by each individual in the second group depends on the outcome of previous randomizations from the first group), it is no longer clear whether the reciprocal of \eqref{eq:Qopt} is indeed the smallest possible asymptotic variance among all mechanisms within this larger class of procedures. Thus, to establish asymptotic efficiency, we prove local asymptotic mixed normality (along subsequences) for every model of the form $Q^{(n)}\P^n = (Q^{(n)}P_\theta^n)_{\theta\in\Theta}$, $\Theta\subseteq\R^p$, where $\P$ is regular (differentiable in quadratic mean) and $Q^{(n)}$ is an arbitrary sequence of sequentially-interactive $\alpha$-differential privacy mechanisms. This result implies convolution and local asymptotic minimax theorems which show that the reciprocal of the expression in \eqref{eq:Qopt} is, indeed, the minimal asymptotic variance. Since our two-step procedure belongs to this class of sequentially-interactive mechanisms and we show that it is regular with the minimal asymptotic variance, it is, indeed, asymptotically efficient.

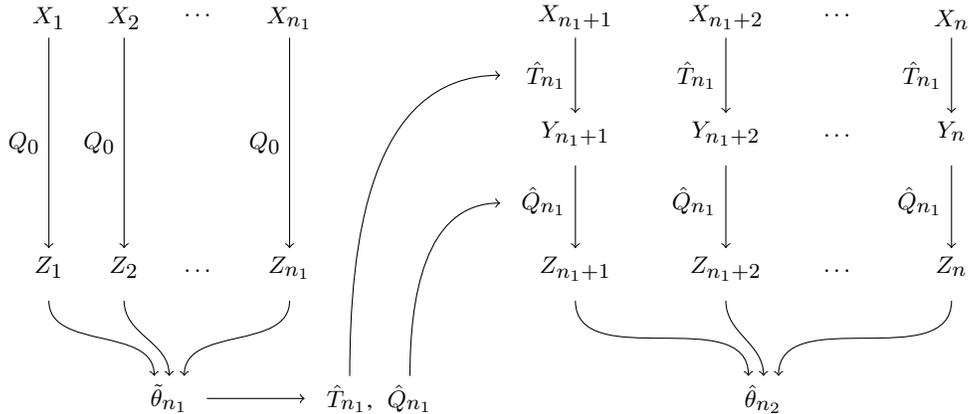
\begin{figure}
\caption{Graphical representation of the efficient two-step $\alpha$-sequentially interactive privacy mechanism and estimation procedure.}
\centering
\begin{tikzpicture}
%===== group 1 ========
\coordinate[label=above:$X_1$] (X1) at (0,0);
\coordinate[label=below:$Z_1$] (Z1) at (0,-2.8);
\draw[->] (X1) -- (Z1);
\coordinate[label=left:$Q_0$] (C1) at ($(X1)!0.5!(Z1)$);
\draw (C1);

\coordinate[label=above:$X_2$] (X2) at (1,0);
\coordinate[label=below:$Z_2$] (Z2) at (1,-2.8);
\draw[->] (X2) -- (Z2);
\coordinate[label=left:$Q_0$] (C2) at ($(X2)!0.5!(Z2)$);
\draw (C2);

\draw (2,0.3) node[]{$\dots$};
\draw (2,-3.1) node[]{$\dots$};

\coordinate[label=above:$X_{n_1}$] (Xn1) at (3.2,0);
\coordinate[label=below:$Z_{n_1}$] (Zn1) at (3.2,-2.8);
\draw[->] (Xn1) -- (Zn1);
\coordinate[label=left:$Q_0$] (Cn1) at ($(Xn1)!0.5!(Zn1)$);
\draw (Cn1);

\coordinate[label=below:$\tilde{\theta}_{n_1}$] (thetan1) at (1.6,-4.5);

\draw [->, out=270, in=90] (Z1) +(0,-0.7) to (1.4,-4.5);
\draw [->, out=270, in=90] (Z2) +(0,-0.7) to (thetan1);
\draw [->, out=270, in=90] (Zn1) +(0,-0.7) to (1.8,-4.5);

\coordinate[label=below:$\hat{Q}_{n_1}$] (Qhat) at (4.8,-4.5);
\coordinate[label=below:$\hat{T}_{{n_1}}\text{,}$] (That) at (4,-4.5);

\draw[->] (thetan1) +(0.5,-0.3) -- (3.4,-4.8);

%======= group 2 =======
\coordinate[label=above:$X_{n_1+1}$] (Xn11) at (7,0);
\coordinate[label=below:$Y_{n_1+1}$] (Y1) at (7,-1);
\coordinate[label=below:$Z_{n_1+1}$] (Zn11) at (7,-2.8);
\draw[->] (Xn11) -- (Y1);
\coordinate[label=left:$\hat{T}_{{n_1}}$] (D1) at ($(Xn11)!0.5!(Y1)$);
\draw (D1);

\coordinate[label=above:$X_{n_1+2}$] (Xn12) at (9,0);
\coordinate[label=below:$Y_{n_1+2}$] (Y2) at (9,-1);
\coordinate[label=below:$Z_{n_1+2}$] (Zn12) at (9,-2.8);
\draw[->] (Xn12) -- (Y2);
\coordinate[label=left:$\hat{T}_{{n_1}}$] (D2) at ($(Xn12)!0.5!(Y2)$);
\draw (D2);

\draw (10.5,0.3) node[]{$\dots$};
\draw (10.5,-1.3) node[]{$\dots$};
\draw (10.5,-3.1) node[]{$\dots$};

\coordinate[label=above:$X_{n}$] (Xn) at (12,0);
\coordinate[label=below:$Y_{n}$] (Yn) at (12,-1);
\coordinate[label=below:$Z_{n}$] (Zn) at (12,-2.8);
\draw[->] (Xn) -- (Yn);
\coordinate[label=left:$\hat{T}_{{n_1}}$] (Dn) at ($(Xn)!0.5!(Yn)$);
\draw (Dn);

\draw[->] (Y1) +(0,-0.7) -- (Zn11);
\coordinate[label=left:$\hat{Q}_{n_1}$] (E1) at ($(Y1)!0.65!(Zn11)$);
\draw (E1);
\draw[->] (Y2) +(0,-0.7) -- (Zn12);
\coordinate[label=left:$\hat{Q}_{n_1}$] (E2) at ($(Y2)!0.65!(Zn12)$);
\draw (E2);
\draw[->] (Yn) +(0,-0.7) -- (Zn);
\coordinate[label=left:$\hat{Q}_{n_1}$] (E3) at ($(Yn)!0.65!(Zn)$);
\draw (E3);

\coordinate[label=below:$\hat{\theta}_{n_2}$] (thetan) at (9.5,-4.5);

\draw [->, out=270, in=90] (Zn11) +(0,-0.7) to (9.3,-4.5);
\draw [->, out=270, in=90] (Zn12) +(0,-0.7) to (thetan);
\draw [->, out=270, in=90] (Zn) +(0,-0.7) to (9.7,-4.5);

\draw [->, out=90, in=180] (Qhat) to (6,-2.2);
\draw [->, out=90, in=180] (That) to (6,-0.5);

\end{tikzpicture}
\label{fig:Qn}
\end{figure}

The paper is organized as follows. We finish this section with a discussion of related literature and we introduce the necessary notation for our exposition. In Section~\ref{sec:Examples} we consider two very basic statistical models, the Bernoulli$(\theta)$ and the Binomial$(2,\theta)$ model, to illustrate some of the challenges a theory of statistical efficiency is facing under local differential privacy. In Section~\ref{sec:LB} we study how differentiability in quadratic mean (DQM) of statistical models for original sensitive data $X_i$ is inherited to the model describing the sanitized data $Z_i$ in the non-interactive case. We then establish local asymptotic mixed normality (LAMN) along subsequences for every model $Q^{(n)}\P^n$, where $\P$ is DQM and $Q^{(n)}$ is any sequence of sequentially interactive privacy mechanisms. Given this result, we can connect to the classical literature on LAMN models and immediately obtain asymptotic convolution and local asymptotic minimax theorems. In Section~\ref{sec:MLE} we show that the lower bounds of the convolution and local minimax theorems can actually be attained by a two-step MLE approach. The crucial step towards this goal is to approximately solve \eqref{eq:Qopt} using discretization. We also show that an appropriate discretization generally exists. Finally, in Section~\ref{sec:appl} we consider the special cases of the Gaussian location and the Gaussian scale model where we can provide simple explicit discretizations. In these examples we also demonstrate that our numerical algorithm, which is based on the linear program of \citet{Kairouz16}, is computationally feasible. Most technical results and proofs are collected in the supplementary material.

\subsection{Related literature}

Differential privacy and its \emph{local} version, where there is no need for a trusted third party, has originated some 20 years ago in the computer science literature \citep[cf.][]{Dinur03, Dwork04, Dwork06, Dwork08a, Evfim03} as a means for privacy protection that is insusceptible to any type of adversary. It took a few years until mathematical statisticians started to investigate the new notion, which is an inherently statistical one, and groundbreaking first contributions are \citet{Wasserman10} and \citet{Duchi17}. Recently, however, the statistical literature on differential privacy has exploded. Consider, for example, \citet{Rohde20, Butucea19, Butucea22, Duchi23, Lalanne22, Acharya22, Cai21, Dunsche22}, to mention only a few of the most recent contributions. Nevertheless some of the most basic statistical questions are still open under differential privacy. One such problem is that of asymptotically efficient parameter estimation. In the central paradigm of differential privacy, efficient estimation has been addressed early on by works of \citet{Smith08, Smith11}. However, in the statistically more challenging local paradigm, where there is no trusted data curator, no analog to the classical theory of statistical efficiency has been available so far. An interesting contribution towards such a direction is \citet{Barnes20}, who investigate and bound Fisher-Information when original data are perturbed using a differentially private randomization mechanism. But those authors have not attempted to find an optimal randomization mechanism or a subsequent estimation procedure. The work probably closest to our own approach is \citet{Duchi23}. They investigate the local minimax risk when data are subject to a local differential privacy constraint. They develop a new information quantity they call the $L^1$-information and use it to characterize the local minimax risk in the regular parametric case. However, while they are more ambitious than our present investigations in that they aim also at potentially non-parametric models, they can only characterize the difficulty of a regular parametric estimation problem up to certain numerical constants and thus all their attainability results are also only up to numerical factors. Hence, the question of existence and construction of asymptotically efficient (asymptotic minimum variance) locally private estimation procedures is still open. In this paper we begin to develop the statistical theory of asymptotic efficiency for local differential privacy by starting with simple regular parametric statistical models. This objective already poses several non-trivial challenges that need to be carefully addressed. For traditional results on statistical efficiency we mainly rely on the well established theory as presented in \citet{vanderVaart07} and \citet{Hopfner14}. 

%Finally, we point out that the problem of maximizing Fisher-Information such as in \eqref{eq:Qopt} also prominently occurs in the literature on optimal experimental designs \citep[cf.][]{Pukelsheim06, Atkinson92}.

%
%Only very recently have statisticians started to gain interest in local differential privacy \citep[see, e.g.,][]{Duchi13a, Duchi13b, Duchi14, Wasserman10, Smith08, Smith11, Ye17, Kroll19}. So far, however, the only rigorous results on statistical estimation under \emph{local} differential privacy deal with optimal rates of convergence of estimators whose worst case risk is characterized only up to constants. 
%In this paper, we are the first to provide asymptotic efficiency results analogous to the classic asymptotic theory of efficiency due to \citet{Hajek70} and \citet{LeCam60}, but for local differential privacy. In contrast to the central paradigm, where often differentially private estimators can easily be constructed that have the same asymptotic variance as their non-private analoges \citep{Smith08, Smith11}, we shall see that a theory of statistical efficiency under local differential privacy is much more challenging. 
%

%========================================================================================

\subsection{Preliminaries and notation}
\label{sec:Prelim}

Let $(\Z_1,\mathcal G_1), \dots, (\Z_n,\mathcal G_n)$, $(\Z,\mathcal G)$, $(\Z',\mathcal G')$ and $(\X,\mathcal F)$ be measurable spaces and write $\mathcal Z^{(n)} := \mathcal Z_1\times\dots\times\mathcal Z_n$ and $\mathcal G^{(n)} := \mathcal G_1\otimes\dots\otimes\mathcal G_n$. Moreover, we use the conventions that $\mathcal Z^{(0)}= \varnothing = \mathcal G^{(0)}$ and $A\times \varnothing = A$ for any set $A$. Let $\mathfrak P(\Z) = \mathfrak P(\Z, \mathcal G)$ be the set of all probability measures on $(\mathcal Z, \mathcal G)$ and $\mathfrak P(\X\to\Z)$ be the set of all Markov kernels $Q:\mathcal G\times\X\to[0,1]$ from $(\X,\mathcal F)$ to $(\Z, \mathcal G)$. Note that we can always think of a Markov kernel $Q\in\mathfrak P(\X\to\Z)$ as a linear mapping from $\mathfrak P(\X)$ to $\mathfrak P(\Z)$ by defining $P\mapsto QP$ as $[QP](dz) := \int_\X Q(dz|x)\,P(dx)$. In the same sense we can also consider compositions $[QR](dz|x) = \int_{\Z'} Q(dz|z')\,R(dz'|x)$ of Markov kernels $R\in\mathfrak P(\X\to\Z')$ and $Q\in\mathfrak P(\Z'\to\Z)$. For $P\in\mathfrak P(\X)$, we easily see that $Q[RP] = [QR]P$ and thus it is allowed to just write $QRP$. For a statistical model $\P\subseteq \mathfrak P(\X)$ we also write $Q\P = \{QP:P\in\P\}$ for the model obtained by applying $Q$ to all original probability measures in $\P$. Notice that any measurable function $T:\X\to\Z$ can also be represented as a (degenerate) Markov kernel $R_T\in \mathfrak P(\X\to\Z)$ by defining $R_T(dz|x)$ to be the Dirac measure $\delta_{T(x)}$ at $T(x)$. Hence, we sometimes write $T\P := R_T\P$ which is easily seen to coincide with the push-forward model $\{PT^{-1}: P\in\P\}$, since $[R_TP](A) = \int_\X \delta_{T(x)}(A)\,P(dx) = \int_\X \mathds 1_{T^{-1}(A)}(x)\,P(dx) = PT^{-1}(A)$. Similarly, we can also make sense of the composition of a Markov kernel $Q\in\mathfrak P(\Z'\to\Z)$ and the measurable function $T:\X\to\mathcal Z'$, by setting $[QT](dz|x) := [QR_T](dz|x) = Q(dz|T(x))$. Throughout, we write $J_p$ for the $p\times p$ identity matrix and $\|\cdot\|_r$ to denote the $\ell_r$ as well as the $L_r(\mu)$ norm whenever the meaning is clear from the context. We also use simply $\|\cdot\|$ to denote the Euclidean norm of vectors and the spectral norm for matrices and we write $\preccurlyeq$ to denote the partial Loewner ordering of matrices. If $\mathcal P = (P_\theta)_{\theta\in\Theta}$ is a parametric model with $\Theta\subseteq\R^p$ that is dominated by $\mu$, we write $p_\theta = \frac{dP_\theta}{d\mu}$ and $\E_\theta$ for a $\mu$-density and for the expectation operator with respect to $P_\theta$, respectively. If $V$ is a random $p$-vector we also write $\Var[V]$ to denote the corresponding $p\times p$ covariance matrix. Finally, $\stackrel{R_{n}}{\rightsquigarrow}$ denotes weak convergence with respect to the sequence of probability measures $(R_n)_{n\in\N}$.

\subsubsection{Sequentially interactive and non-interactive differential privacy}
\label{sec:DefPriv}

For $x,x'\in\X^n$, define $d_0(x,x') := |\{i:x_i\ne x_i'\}|$ to be the Hamming distance, i.e., the number of differing components in $x$ and $x'$. Recall that for $\alpha\in(0,\infty)$, a Markov kernel $Q \in\mathfrak P(\X^n\to\Z)$ is called \emph{$\alpha$-differentially private}, if, for all $A\in\mathcal G$ and for all $x,x'\in\X^n$ with $d_0(x,x')=1$, we have
\begin{equation}\label{eq:alphaPriv}
Q(A|x) \;\le\; e^\alpha\,Q(A|x').
\end{equation}
We sometimes also refer to such a Markov kernel as a \emph{privacy mechanism} or a \emph{channel}.
Note that this definition implies that the probability measures $Q(dz|x)$, for different $x\in\X^n$, are mutually absolutely continuous (equivalent). In particular, if $\P\subseteq\mathfrak P(\X^n)$ is a statistical model on $\X^n$, then all the measures $QP$, for $P\in\P$, are equivalent and they are also equivalent to $Q(dz|x)$, for any $x\in\X^n$. We now turn to \emph{local differential privacy}, that is, the situation where there is no trusted data curator available and each individual is able to generate a sanitized version $Z_i$ of their sensitive data $X_i$ on their \emph{local} machine without ever sharing $X_i$ with anybody else.

First, we define the set of marginal $\alpha$-differentially private channels by
\begin{equation}\label{eq:Margset}
\mathcal Q_\alpha(\X\to\mathcal Z) := \left\{ Q\in\mathfrak P(\X\to\Z) : Q(A|x)\le e^\alpha Q(A|x'), \;\forall A\in\mathcal G, \forall x,x'\in\X \right\}.
\end{equation}
When searching for an optimal privacy mechanism, we want to impose no a priori restrictions on the space $(\mathcal Z,\mathcal G)$ from which the sanitized observations are drawn. Thus, we also consider the union over all possible measurable output spaces
\begin{equation}
\mathcal Q_\alpha := \mathcal Q_\alpha(\X) := \bigcup_{(\mathcal Z,\mathcal G)} \mathcal Q_\alpha(\X\to\mathcal Z).\footnote{Notice that strictly speaking $\mathcal Q_\alpha(\X)$ is set theoretically not well defined, as it is at least as large as the set of all sets. However, this is not an issue. We simply need to define expressions like $\sup_{Q\in\mathcal Q_\alpha(\X)} I_\theta(Q\P)$ by $\sup\{a\in\R: \exists (\mathcal Z,\mathcal G) \text{ measurable}: \exists Q\in\mathcal Q_\alpha(\X\to\mathcal Z) : a = I_\theta(Q\P)\}$}
\end{equation}
The elements of $\mathcal Q_\alpha$ are the basic building blocks of locally private channels, because they can be executed by a single individual independently of anybody else, using only their own information $X_i=x$ and producing a sanitized version $Z_i\thicksim Q(dz|x)$.
We also introduce two specific classes of $\alpha$-locally differentially private mechanisms. A Markov kernel $Q^{(n)}\in \mathfrak P(\X^n \to \Z^{(n)})$ is said to be \emph{sequentially interactive} if the following condition is satisfied: For every $i=1,\dots, n$, there exists a Markov kernel $Q_i\in\mathfrak P(\X\times\mathcal Z^{(i-1)}\to\mathcal Z_i)$, such that for all $x_1,\dots, x_n\in\X$,
\begin{align}\label{eq:Seq}
Q^{(n)}\left( dz_{1:n}\Big|x_1,\dots, x_n\right)
= Q_{z_{1:n-1}}(dz_n|x_n) 
Q_{z_{1:n-2}}(dz_{n-1} |x_{n-1})\cdots Q_\varnothing(dz_1|x_1),
\end{align}
where $Q_{z_{1:i-1}}(dz_i|x_i) := Q_i(dz_i|x_i,z_{1:i-1})$, $z_{1:0} = \varnothing$ and $z_{1:i} = (z_1,\dots, z_i)^T\in\mathcal Z^{(i)}$. 
Here, the idea is that individual $i$ can only use $X_i$ and previous $Z_j$, $j<i$, in its local privacy mechanism, thus leading to the sequential structure in the above definition. If, in addition, for all $i\in[n]$ and for all $z_{1:i-1}\in\mathcal Z^{(i-1)}$ we have $Q_{z_{1:i-1}}\in\mathcal Q_\alpha(\X\to\mathcal Z_i)$, then, by the usual approximation of integrands by simple functions, it is easy to see that $Q$ in \eqref{eq:Seq} is $\alpha$-differentially private (i.e., satisfies \eqref{eq:alphaPriv}), in which case we call it $\alpha$-\emph{sequentially interactive}. This notion coincides with the definition of sequentially interactive channels in \citet{Duchi17} and \citet{Rohde20}. 

An important subclass of sequentially interactive mechanisms are the so called \emph{non-interactive} mechanisms $Q\in\mathfrak P(\X^n\to\mathcal Z^{(n)})$ that are of product form
\begin{equation}\label{eq:non-Inter}
Q\left( dz_{1:n}\Big|x_1,\dots, x_n\right) = \bigotimes_{i=1}^n Q_i(dz_i|x_i), \quad\quad \forall x_1,\dots, x_n\in\X,
\end{equation}
for some marginal kernels $Q_i\in\mathfrak P(\X\to\Z_i)$.
Clearly, a non-interactive mechanism $Q$ satisfies \eqref{eq:alphaPriv} if, and only if, for all  $i=1,\dots, n$, $Q_i\in\mathcal Q_\alpha$. In that case it is also called \emph{$\alpha$-non-interactive}. 

%====================================================================================

\section{Two simple illustrative examples}
\label{sec:Examples}

We begin our exposition by considering two very basic statistical models and investigate the possibility of efficient parameter estimation in these special cases in order to highlight some of the obstacles a general theory of efficiency has to overcome.

\subsection{Bernoulli data}
\label{sec:Bernoulli}

Perhaps one of the most basic models of statistics is the simple Bernoulli experiment with unknown success probability $\theta\in\Theta=(0,1)$ and (marginal) sample space $\X=\{0,1\}$, that is, $p_\theta(x) = \theta^x(1-\theta)^{1-x}$, $x\in\X$. The model $\P=(p_\theta)_{\theta\in\Theta}$ is regular in the sense of differentiability in quadratic mean (see Definition~\ref{def:DQM} below) with score $s_\theta$ and has Fisher-Information $I_\theta(\P) = [\theta(1-\theta)]^{-1}$. For arbitrary $Q\in\mathcal Q_\alpha(\X)$, the model $Q\P$ of the sanitized data $Z$ is also regular with Fisher-Information denoted by $I_\theta(Q\P)$ (cf. Lemma~\ref{lemma:DQM} below). 
We now show that the $\alpha$-private channel $Q_0\in\mathfrak P(\{0,1\}\to\{0,1\})$ defined by its matrix representation
\begin{align*}
Q_0 \triangleq \begin{pmatrix}
e^\alpha &1\\
1 &e^\alpha
\end{pmatrix}\frac{1}{e^\alpha+1},
\end{align*}
that is, $Q_0(0|0) = \frac{e^\alpha}{e^\alpha+1} = Q_0(1|1)$, maximizes the Fisher-Information $Q\mapsto I_\theta(Q\P)$ over all of $\mathcal Q_\alpha$ and for every $\theta\in(0,1)$. The channel $Q_0$ actually coincides with Warner's original randomized response mechanism \citep{Warner65}. The following argument is inspired by \citet[][Theorem~18]{Kairouz16}.

Fix a mechanism $Q\in \mathcal Q_\alpha(\X)$. For $\nu(dz):= Q(dz|0)$, choose $\nu$-densities $q(z|0)$ and $q(z|1)$ of $Q(dz|0)$ and $Q(dz|1)$, respectively. In particular, $q(z|0)=1$. Consider the formal matrix inverse of the randomized response mechanism $Q_0$, that is, 
\begin{align*}
Q_0^{-1} \triangleq \begin{pmatrix}
e^\alpha &-1\\
-1 &e^\alpha
\end{pmatrix}\frac{1}{e^{\alpha}-1}.
\end{align*}
Even though $Q_0^{-1}$ does not define a Markov kernel (as it contains negative values), we can still study the composition $\bar{Q} := QQ_0^{-1}$, which we define by its $\nu$-densities $\bar{q}(z|x) := \sum_{y\in\{0,1\}}q(z|y)Q_0^{-1}(y|x) = \frac{1}{e^{\alpha}-1}\left[ q(z|x) e^\alpha - q(z|1-x)\right]$. Thus, $\bar{Q}(A|x) =  \frac{1}{e^{\alpha}-1}\left[ Q(A|x) e^\alpha - Q(A|1-x)\right] \ge 0$, $\bar{Q}(\Z|x) = 1$, $A\mapsto \bar{Q}(A|x)$ is $\sigma$-additive and $x\mapsto \bar{Q}(A|x)$ is measurbale, that is, $\bar{Q}$ is a Markov kernel. Lemma~\ref{lemma:DQM} yields that both $Q_0\mathcal P$ and $Q\mathcal P$ are regular with scores $r_\theta$ and $t_\theta$, say. By construction $\bar{Q}Q_0 = Q$, which means that $(Z,X)\thicksim Q(dz|x)P_\theta(dx)$ is the $(Z,X)$-marginal distribution of $(Z,Y,X)\thicksim \bar{Q}(dz|y)Q_0(dy|x)P_\theta(dx)$. Moreover, $X$ and $Z$ are conditionally independent given $Y$. Thus, in view of Lemma~\ref{lemma:DQM},
\begin{align*}
I_\theta(Q\mathcal P) &= \Var(t_\theta(Z)) = \Var(\E[s_\theta(X)|Z]) = \Var(\E[ \E[s_\theta(X)|Y,Z]|Z])\\
&= \Var(\E[ \E[s_\theta(X)|Y]|Z]) = \Var(\E[ r_\theta(Y)|Z]) \\
&= \Var(r_\theta(Y)) - \E[\Var(r_\theta(Y)|Z)]
\le \Var(r_\theta(Y)) = I_\theta(Q_0\mathcal P).
\end{align*}
Since $Q$ was arbitrary, optimality of $Q_0$ is established.

It is easy to see that 
$$
I_\theta(Q_0\P) = \left[ \frac{e^\alpha}{(e^\alpha-1)^2} + \theta(1-\theta)\right]^{-1}.
$$
Moreover, if $Z_1, \dots, Z_n$ are iid from $Q_0P_\theta$ (note that such a sample can be generated in an $\alpha$-non-interactive way by letting all $n$ individuals use the same $Q_0$ independently to generate $Z_i\thicksim Q_0(dz|x=X_i)$), then the estimator
$$
\hat{\theta}_{\text{priv}} := \frac{e^\alpha+1}{e^\alpha-1}\left(\bar{Z}_n - \frac{1}{e^\alpha+1}\right),
$$
is UMVUE among all non-interactive $\alpha$-private procedures, reaching the (minimal) Cram\'{e}r-Rao lower bound of $I_\theta(Q_0\P)^{-1}$, which is the smallest lower bound we can hope for in view of our choice $Q_0$. Thus, in the simple Bernoulli model, efficient estimation is possible even in finite samples.

\begin{remark}[On the $L^1$-information of \citet{Duchi23}]\normalfont
In a setup similar to ours -- yet a more ambitious semi-parametric one -- \citet{Duchi23} proposed the so-called $L^1$-information as an alternative to Fisher-Information as a complexity measure in local differential privacy. It is defined as $\mathbb I_\theta(\P) := \E_\theta[|s_\theta|]$, where $s_\theta$ is the score at $\theta$ in the model $\P$. One may wonder if this quantity, which was shown to play an important role in determining the rate of convergence in semi-parametric problems, also characterizes the difficulty of parametric estimation problems and how it compares to the optimal Fisher-Information $I_{\theta,\alpha}^*(\P) := \sup_{Q\in\mathcal Q_\alpha(\X)}I_\theta(Q\P)$. From the risk bounds in \citet{Duchi23} one might conjecture that $4\mathbb I_\theta(\P)^{-2}$ is the asymptotic variance of $\sqrt{n\alpha^2}(\hat{\theta}_n-\theta)$, for an $\alpha$-locally private and efficient estimator $\hat{\theta}_n$. This means that $\sqrt{n}(\hat{\theta}_n-\theta)$ has asymptotic variance $\frac{4}{\alpha^2}\mathbb I_\theta(\P)^{-2}$. Our theory (see the discussion in Section~\ref{sec:ConvMinMax}) predicts that this variance is equal to $I_{\theta,\alpha}^*(\P)^{-1}$. We can evaluate these expressions at least in the Bernoulli example, where we obtain $\frac{4}{\alpha^2}\mathbb I_\theta(\P)^{-2} = \frac{1}{\alpha^2}$, whereas 
$I_{\theta,\alpha}^*(\P)^{-1} =  \frac{e^\alpha}{(e^\alpha-1)^2} + \theta(1-\theta)$. Although the two quantities are clearly not equal for all $\theta\in\Theta$ even after arbitrary rescaling, we see that at least in the high-privacy regime $\alpha\to0$, we have
$$
\frac{I_{\theta,\alpha}^*(\P)^{-1}}{\frac{4}{\alpha^2}\mathbb I_\theta(\P)^{-2}} \quad\xrightarrow[\alpha\to0]{}\quad 1.
$$
Hence, we might wonder if the $L^1$-information can at least characterize the optimal variance for small values of $\alpha$.
\end{remark}

\subsection{Binomial $(2,\theta)$ data}

Suppose now that the model $\P=(p_\theta)_{\theta\in\Theta}$ is only slightly more complicated than above and defined by $p_\theta(x) = \binom{2}{x}\theta^x(1-\theta)^{2-x}$, for $x\in\X=\{0,1,2\}$ and $\theta\in\Theta=(0,1)$. By analogy with the Bernoulli model, an intuitively appealing candidate for the optimal $\alpha$-private channel distribution in this case is
\begin{align*}
Q_0 \triangleq \begin{pmatrix}
e^\alpha &1 &1\\
1 &e^\alpha &1\\
1 &1 &e^\alpha
\end{pmatrix}\frac{1}{e^\alpha+2}.
\end{align*}
However, at least for certain values of the parameter $\theta$, this can be improved. In fact, \citet{Hucke19} showed \citep[using results from][]{Kairouz16} that the following mechanism maximizes $Q\mapsto I_\theta(Q\P)$ over all $\alpha$-private channels with finite state space $\Z$:
\begin{align}\label{eq:QBinom}
Q_{\theta}^*=
\begin{cases}
	\begin{pmatrix}
		e^\alpha &1  &1\\
		1 &e^\alpha  &e^\alpha\\
		0 &0  &0
	\end{pmatrix}\frac{1}{e^\alpha + 1}, &\text{if } 0<\theta\le \frac{1}{2} - c_\alpha, \\
	\begin{pmatrix}
		e^\alpha  &1  &1\\
		1  &e^\alpha  &1\\
		1  &1 &e^\alpha
	\end{pmatrix}\frac{1}{e^\alpha + 2}, &\text{if } \frac{1}{2} - c_\alpha<\theta< \frac{1}{2} + c_\alpha, \\
	\begin{pmatrix}
		e^\alpha  &e^\alpha  &1\\
		1  &1  &e^\alpha \\
		0  &0  &0
	\end{pmatrix}\frac{1}{e^\alpha + 1}, &\text{if } \frac{1}{2} + c_\alpha\le\theta<1,
\end{cases}
\end{align}
for a constant $c_\alpha\in[0,\frac{1}{2})$ that satisfies $c_\alpha=0$ if, and only if, $\alpha\le \log(3)$. In view of Lemma~\ref{lemma:finiteReduction} below, the restriction of finite state space can be dropped and the channel $Q_\theta^*$ is even optimal among all $\alpha$-private mechanisms in $\mathcal Q_\alpha(\X)$. \citet{Hucke19} also showed that if $\alpha<\log(3)$, then $I_\theta(Q_0\P) < I_\theta(Q_\theta^*\P)$, for all $\theta\in(0,1)$, showing that in the high privacy regime the intuitive choice $Q_0$ is suboptimal at every point of the parameter space. Furthermore, irrespective of the value of $\alpha\in(0,\infty)$, there can not be another channel $Q\in\mathcal Q_\alpha(\X)$ that is independent of $\theta$ and satisfies $I_\theta(Q\P) = I_\theta(Q_\theta^*\P)$, for all $\theta\in(0,1)$. To see this last claim, we first note that \citet{Hucke19} also established non-differentiability of $\theta\mapsto I_\theta(Q_\theta^*\P)$ at $\frac{1}{2} - c_\alpha$ and $\frac{1}{2} + c_\alpha$. However, Lemma~\ref{lemma:BinomialDiff} in the appendix shows that $\theta\mapsto I_\theta(Q\P)$ is differentiable on $(0,1)$ for every channel $Q\in\mathcal Q_\alpha(X)$.

This analysis leads to two conclusions. First, somewhat surprisingly, it can be beneficial to consider channels with smaller output/state space $\Z$ than input/sample space $\X$, at least in some parts of the parameter space, as is the case in \eqref{eq:QBinom}, where for $\theta$ close to zero and one, $Q^*_\theta$ has an entire row containing only zeros. Second, no non-interactive $\alpha$-differentially private mechanism with identical marginals can be optimal for estimation over the entire parameter space. The optimal (marginal) choice $Q_\theta^*$ necessarily depends on the unknown parameter $\theta$ and any a priori choice $Q^*_{\theta_0}$ will be suboptimal for certain other parameters $\theta\neq\theta_0$. The last observation naturally leads to the idea of a data dependent choice of channel 
$$
Q_{\tilde{\theta}_{n_1}}^*
$$ 
based on some preliminary $\alpha$-private estimator $\tilde{\theta}_{n_1}$ from sanitized data $Z_1,\dots, Z_{n_1}$. Thus, the randomization mechanism used by the second group $i=n_1+1,\dots, n$ of the data owners must depend on sanitized data from the first group of data providers. In other words, if we want to (at least asymptotically) reach the minimal variance $I_\theta(Q_\theta^*\P)^{-1}$, we have to consider sequentially interactive mechanisms.

Notice, however, that turning to sequentially interactive procedures, we have left the class of statistical procedures within which we were trying to find an optimal one. In this section so far we had only considered iid models for data vectors $X=(X_1,\dots, X_n)$ of the form $\P^n = (P_\theta^n)_{\theta\in\Theta}$ and marginal mechanisms $Q \in\mathcal Q_\alpha$ that can be used by each data owner independently. Hence, we only considered non-interactive channels with identical marginals $Q^{(n)} = \bigotimes_{i=1}^n Q_i$, $Q_i=Q\in\mathcal Q_\alpha$. Therefore, the model of the sanitized observations $Z=(Z_1,\dots, Z_n)$ was again an iid model $Q^{(n)}\P^n = (Q^{(n)}P_\theta^n)_{\theta\in\Theta} = ([QP_\theta]^n)_{\theta\in\Theta}$ and the joint Fisher-Information is simply $I_\theta(Q^{(n)}\P^n) = n I_\theta(Q\P)$. So the question arises if it may be possible to carefully pick an $\alpha$-sequentially-interactive channel $Q^{(n)}$ such that $\frac{1}{n}I_\theta(Q^{(n)}\P^n) > \sup_{Q\in\mathcal Q_\alpha}I_\theta(Q\P)$. However, if the original data model is of the iid form $\P^n$, this hope is not justified, because one can show that 
$$\frac{1}{n}I_\theta(Q^{(n)}\P^n) \le \sup_{Q\in\mathcal Q_\alpha}I_\theta(Q\P).$$ 
Nevertheless, in what follows, we develop all our lower bound results for the larger class of $\alpha$-sequentially-interactive channels, since our efficient estimation procedure will also be of this type. We find that, indeed, the optimal asymptotic variance among all $\alpha$-sequentially-interactive procedures is given by
$$
\left[\sup_{Q\in\mathcal Q_\alpha}I_\theta(Q\P)\right]^{-1}.
$$

%=================================================================

\section{Asymptotic lower bounds for estimation in locally private models}
\label{sec:LB}

In this section we show that the sequence of experiments $\mathcal E_n = (\mathcal Z^{(n)}, \mathcal G^{(n)}, Q^{(n)}\P^n)$, $n\ge1$, is locally asymptotically mixed normal \citep[LAMN, cf.][]{Jeganathan80, Jeganathan82}, provided only that $Q^{(n)}$ is $\alpha$-sequentially-interactive and $\P$ is differentiable in quadratic mean (DQM). To be precise, we show a little bit less than that, namely, that every subsequence $n'$ has a further subsequence $n''$ along which $\mathcal E_{n''}$ is LAMN. This result will be sufficient to connect to the existing literature on convolution and local asymptotic minimax theorems for LAMN experiments. In the following subsection we begin by studying regularity properties of the marginal models $\P$ and $Q\P$, for $Q\in\mathcal Q_\alpha$. In Subsection~\ref{sec:LAMN} we then present our LAMN result and discuss its implications in Subsection~\ref{sec:ConvMinMax}.

\subsection{Regularity of locally private models}
\label{sec:reg}

Consider a statistical model $\P=(P_\theta)_{\theta\in\Theta}$ on the sample space $(\X,\mathcal F)$ with parameter space $\Theta\subseteq\R^p$. In the iid setting we consider here, the final data generating model will then simply be the $n$-fold product $\P^n=(P_\theta^n)_{\theta\in\Theta}$, and a common way to define regularity of such a model and study issues of asymptotic efficiency is by means of differentiability in quadratic mean (DQM). We recall the following classical definition.

\begin{definition}[Differentiability in Quadratic Mean]
\label{def:DQM}
A model $\P=(P_\theta)_{\theta\in\Theta}$ with $\Theta\subseteq\R^p$, sample space $(\X,\mathcal F)$ and ($\sigma$-finite) dominating measure $\mu$ is called \emph{differentiable in quadratic mean (DQM) at the point $\theta\in\Theta$}, if $\theta$ is an interior point of $\Theta$ and the $\mu$-densities $p_\theta = \frac{dP_\theta}{d\mu}$ satisfy
$$
\Psi_\theta(h) := \int_{\X} \left( \sqrt{p_{\theta+h}(x)} - \sqrt{p_\theta(x)} - \frac{1}{2} h^T s_\theta(x)\sqrt{p_\theta(x)}\right)^2 \mu(dx) \;=\; o(\|h\|^2)
$$
as $h\to0$, for some measurable vector valued function $s_\theta:\X\to\R^p$. The function $s_\theta$ is called the \emph{score} at $\theta$.
\end{definition}

Notice that this definition is independent of the choice of dominating measure $\mu$, because the quantity $\Psi_\theta(h)$ does not depend on $\mu$. To see this, simply consider two dominating measures $\mu_1$ and $\mu_2$ and express the two resulting integrals as an integral with respect to $\mu=\mu_1+\mu_2$. 

We say that a model is DQM or simply \emph{regular}, if it is DQM at every point $\theta\in\Theta$. For convenience we sometimes write $dx$ instead of $\mu(dx)$. We also recall that if the model $\P$ is DQM at $\theta$, then the score satisfies $\E_\theta[s_\theta] = 0$ and the Fisher information matrix $I_\theta(\P) := \E_\theta[s_\theta s_\theta^T]\in\R^{p\times p}$ exists and is finite \citep[see][Theorem~7.2]{vanderVaart07}. 
The cited theorem also implies that the sequence of resulting product experiments $\mathcal E_n = (\X^n, \mathcal F^n, \P^n)$ is locally asymptotically normal (LAN). 
It is known \citep[][Proposition~A.5]{Bickel93} that if $\P=(P_\theta)_{\theta\in\Theta}$ is DQM and $T:\X\to\mathcal Z$ is measurable, then the model of the transformed data $Z=T(X)$, which is given by $T\P = (P_\theta T^{-1})_{\theta\in\Theta}$, is also DQM with score $t_\theta(z) = \E[s_\theta(X)|Z=z]$, where $(X,Z) \thicksim (X,T(X))$ under $P_\theta$. For differential privacy we need to consider randomized functions $T$, or equivalently, Markov kernels. The following lemma provides the desired extension. It is a simple consequence of a more refined result that we deferred to Appendix~\ref{sec:app:LB}.

%See Proposition~A.5 of \citet{Bickel93}. Maybe combine with \citet{vonWeiz74} to get the following result more generally.

\begin{lemma} \label{lemma:DQM}
Fix $\alpha\in(0,\infty)$. Suppose the model $\P=(P_\theta)_{\theta\in\Theta}$ is DQM at $\theta\in\Theta\subseteq\R^p$, with score function $s_\theta$, and $Q\in\mathcal Q_\alpha(\X\to\mathcal Z)$ is a privacy mechanism. Then the model $Q\P=(QP_\theta)_{\theta\in\Theta}$ is also DQM at $\theta$ 
%(for every dominating measure of the form $Q(dz|x^*), x^*\in\mathcal X$) 
with score function 
$$
t_\theta(z) := \E[s_\theta(X)|Z=z],
$$
that is, a regular conditional expectation of $s_\theta(X)$ given $Z=z$, where the joint distribution of $X$ and $Z$ is determined by $X\thicksim P_\theta$ and $(Z|X=x) \thicksim Q(dz|x)$.  
Furthermore, a version of the score can be chosen such that $\sup_{z\in\mathcal Z}\|t_\theta(z)\| \le e^{2\alpha}\sqrt{\trace{I_\theta(\P)}}$ and the Fisher-Information matrix $I_\theta(Q\mathcal P)$ in the model $Q\mathcal P$ satisfies
$$
I_\theta(Q\mathcal P) = \Var[\E[s_\theta(X)|Z]] = I_\theta(\mathcal P) - \E[(s_\theta(X) - t_\theta(Z))(s_\theta(X) - t_\theta(Z))^T].
$$
\end{lemma}
\begin{proof}
Let $q:\mathcal Z\times\X\to[e^{-\alpha}, e^\alpha]$ be as in Lemma~\ref{lemma:Qdensities} in the supplement. Now apply Lemma~\ref{lemma:DQMbound} with fixed, $z\in\mathcal Z$, $P_0 := P_\theta$, $c:= \E_{P_\theta}[\|s_\theta\|^2] = \trace{I_\theta(\P)}<\infty$, $P_1:=P_{\theta+h}$, and $t_\theta(z) := \int_\X s_\theta(x)\frac{q(z|x)p_0(x)}{q_0(z)} \mu(dx)$, for sufficiently small $\|h\|$, such that $\theta+h\in\Theta$ and $\|h\|\sqrt{c}< e^{-2\alpha}$. Notice that the resulting upper bound in Lemma~\ref{lemma:DQMbound} is $o(\|h\|^2)$ uniformly in $z\in\mathcal Z$, which (together with the fact that $\nu$ is a probability measure) implies DQM of $Q\P$ with score function $t = t_\theta(z)$. Since $q(z|x)p_\theta(x)$ is a $\nu\otimes\mu$-density of $(Z,X)$, we see that the score is a version of the conditional expectation, as claimed. Appropriate boundedness of the score follows from
$$
\left\|\int_\X s_\theta(x)\frac{q(z|x)p_\theta(x)}{\int_\X q(z|y)p_\theta(y)\mu(dy)} \mu(dx)\right\|
\le e^{2\alpha} \int_\X \|s_\theta(x)\| p_\theta(x) \mu(dx) \le e^{2\alpha} \sqrt{\E_\theta[\|s_\theta\|^2]},
$$
using the point-wise bounds on $q(z|x)$. The representation of $I_\theta(Q\mathcal P)$ is an immediate consequence of the variance decomposition formula
$
I_\theta(\mathcal P) = \Var[s_\theta(X)] = \Var[\E[s_\theta(X)|Z]] + \E[\Var[s_\theta(X)|Z]].
$
\end{proof}

\subsection{LAMN along subsequences}
\label{sec:LAMN}

Recall from the discussion at the end of Section~\ref{sec:Examples} that we need to work with sequentially interactive mechanisms $Q^{(n)}$ in order to have any hope of achieving the asymptotic variance $[\sup_{Q\in\mathcal Q_\alpha} I_\theta(Q\P)]^{-1}$ in general. This, however, requires us to consider dependent and possibly non-identically distributed sanitized observations $Z_1,\dots, Z_n$. The corresponding model $Q^{(n)}\P^n$ will therefore typically no longer be LAN. The necessary generalization is given by the notion of local asymptotic mixed normality \citep[LAMN, cf.][]{Jeganathan80, Jeganathan82}. For a given parameter space $\Theta\subseteq \R^p$, let $\mathcal E_n=(\Omega_n,\mathcal A_n, \{R_{n,\theta}:\theta\in\Theta\})$, $n\ge1$, be a sequence of statistical experiments, where $R_{n,\theta}$ is a probability measure on the measurable space $(\Omega_n, \mathcal A_n)$. For $\theta,\theta_0\in\Theta$, define the associated log-likelihood ratio $\Lambda_n(\theta,\theta_0) := \log \frac{dR_{n,\theta}}{dR_{n,\theta_0}}$, where $\frac{dR_{n,\theta}}{dR_{n,\theta_0}}$ is the Radon-Nikodym derivative of the absolutely continuous part of $R_{n,\theta}$ with respect to $R_{n,\theta_0}$ and set $\Lambda_n(\theta,\theta_0) := \Lambda_n(\theta_0,\theta_0) = 0$ if $\theta\notin\Theta$.

\begin{definition}[Local Asymptotic Mixed Normality]
\label{def:LAMN}
Fix a parameter space $\Theta\subseteq\R^p$ and a sequence $(\delta_n)_{n\in\N}$ of positive definite $p\times p$ matrices with $\|\delta_n\|\to0$ as $n\to\infty$. We say the sequence of statistical experiments $\mathcal E_n = (\Omega_n, \mathcal A_n, \{R_{n,\theta}:\theta\in\Theta\})$ satisfies the LAMN condition at $\theta\in\Theta$, if $\theta$ is an interior point of $\Theta$ and the following two conditions hold:
\begin{enumerate}
\item There exists a sequence $(\Delta_{n,\theta})_{n\in\N}$ of random vectors $\Delta_{n,\theta} : \Omega_n \to\R^p$ and a sequence $(\Sigma_{n,\theta})_{n\in\N}$ of random nonnegative-definite matrices $\Sigma_{n,\theta}: \Omega_n \to\R^{p\times p}$, such that 
%$R_{n,\theta}(\Sigma_{n,\theta} \text{ is p.d.})=1$, for every $n\in\N$, and
$$
\Lambda_n(\theta+\delta_n h, \theta) - \left( h^T \Delta_{n,\theta} - \frac12 h^T\Sigma_{n,\theta}h \right)
$$
converges to zero in $R_{n,\theta}$ probability for every $h\in\R^p$.
\item We have
$$
\begin{pmatrix}
\Delta_{n,\theta}\\
\Sigma_{n,\theta}
\end{pmatrix} \; \stackrel{R_{n,\theta}}{\rightsquigarrow}\; 
\begin{pmatrix}
\Sigma_{\theta}^{1/2}\Delta\\ \Sigma_{\theta}\end{pmatrix},
$$
where $\Delta \thicksim N(0,J_p)$ is independent of $\Sigma_\theta$.
\end{enumerate}
\end{definition}

If the limiting matrix $\Sigma_\theta$ in the above definition is non-random, the LAMN property reduces to the more classical notion of local asymptotic normality (LAN). It is a classical result \citep[cf.][Theorem~7.2]{vanderVaart07} that if the model $\mathcal P=(P_\theta)_{\theta\in\Theta}$ on sample space $\X$ is DQM at $\theta\in\Theta$, then the sequence of experiments $\mathcal E_n=(\X^n,\mathcal F^n, \{P_\theta^n:\theta\in\Theta\})$ is LAN at $\theta$ with $\delta_n = J_p/\sqrt{n}$ and $\Sigma_\theta$ equal to the Fisher-Information matrix $I_\theta(\mathcal P)$, which is non-random. An immediate consequence of that observation is the following simple contiguity result that holds for any sequence $Q^{(n)}$ of Markov kernels from $(\X^n,\mathcal F^n)$ to $(\mathcal Z_n,\mathcal G_n)$. 
%We write $R_{n,\theta} = Q^{(n)}P_\theta^n$ for the joint distribution of the observed data $Z_1,\dots, Z_n$ and $\E_{n,\theta}$ for the corresponding expectation operator.

\begin{lemma}\label{lemma:contiguity}
Consider a model $\P$ on sample space $(\X,\mathcal F)$ that is DQM at $\theta\in\Theta\subseteq\R^p$. Fix $h\in\R^p$ and an arbitrary Markov kernel $Q^{(n)}$ from $(\X^n,\mathcal F^n)$ to $(\mathcal Z_n,\mathcal G_n)$. Then $R_{n,\theta} = Q^{(n)}P_\theta^n$ and $R_{n,\theta+h/\sqrt{n}}= Q^{(n)}P_{\theta+h/\sqrt{n}}^n$ are mutually contiguous.
\end{lemma}
\begin{proof}
Write $\theta_n := \theta+h/\sqrt{n}$. The LAN property of $\P$ implies $P_{\theta_n}^n \vartriangleleft\vartriangleright P_{\theta}^n$ \citep[][Example~6.5]{vanderVaart07}. Now fix a sequence of events $A_n\in\mathcal G^{(n)}$ such that $R_{n,\theta}(A_n) \to 0$. In particular, for $T_n(x) := Q^{(n)}(A_n|x)\in[0,1]$, we have $\E_{P_\theta^n}[T_n] = R_{n,\theta}(A_n) \to 0$, or, equivalently, $T_n \to 0$ in $P_\theta^n$-probability. But by contiguity we get $T_n \to 0$ in $P_{\theta_n}^n$-probability, which implies $R_{n,\theta_n}(A_n) = \E_{P_{\theta_n}^n}[T_n] \to 0$. Reversing the roles of $\theta$ and $\theta_n$ finishes the proof.
\end{proof}

Now, for $n\in\N$, consider an $\alpha$-sequentially-interactive privacy mechanism $Q^{(n)}$ from $\X^n$ to $(\Z^{(n)},\mathcal G^{(n)})$, and a (marginal) model $\P= (P_\theta)_{\theta\in\Theta}$ on $\X$. The statistical experiment generating the sanitized observations $Z_1,\dots, Z_n$ is therefore given by $\mathcal E_n = (\Z^{(n)}, \mathcal G^{(n)}, \{Q^{(n)}P_\theta^n:\theta\in\Theta\})$. The following theorem shows that this sequence of experiments possesses the LAMN property, provided that the underlying data generating model $\P$ is DQM and an appropriate sequence of nonnegative-definite matrices $\Sigma_{n,\theta}$ has a stochastic (weak) limit $\Sigma_\theta$. 
\begin{theorem}\label{thm:LAMN}
Fix $\alpha>0$ and for each $n\in\N$, let $Q^{(n)}$ be $\alpha$-sequentially-interactive \eqref{eq:Seq} from $(\X^n,\mathcal F^n)$ to $(\mathcal Z^{(n)}, \mathcal G^{(n)})$ and assume that all the $\sigma$-fields $\mathcal F$, $\mathcal G_i$, $i=1,\dots, n$, are countably generated. Suppose the model $\P=(P_\theta)_{\theta\in\Theta}$ on sample space $(\X,\mathcal F)$ is DQM at $\theta\in\Theta\subseteq\R^p$ with score $s_\theta$ and define
$$
\Sigma_{n,\theta}(z_{1:n-1}) := \frac{1}{n} \sum_{i=1}^n I_\theta(Q_{z_{1:i-1}}\mathcal P), \quad z_{1:n-1}=(z_1,\dots, z_{n-1})\in\mathcal Z^{(n-1)}.
$$
Then $z_{1:i-1}\mapsto I_\theta(Q_{z_{1:i-1}}\P)$ is measurable and there exists a measurable function $z_{1:i}\mapsto  t_{i,\theta}(z_i|z_{1:i-1})$ such that for all $z_{1:i-1}\in\mathcal Z^{(i-1)}$, $z\mapsto t_{i,\theta}(z|z_{1:i-1})$ is a score in the model $Q_{z_{1:i-1}}\P$ at $\theta$.

If $\Sigma_{n,\theta}\stackrel{Q^{(n)}P^n_{\theta}}{\rightsquigarrow}\Sigma_\theta$ for some random symmetric $p\times p$ matrix $\Sigma_\theta$ and $I_\theta(\P)\neq0$, then the sequence $\mathcal E_n = (\Z^{(n)}, \mathcal G^{(n)}, \{Q^{(n)}P_\theta^n:\theta\in\Theta\})$, $n\in\N$, is LAMN at $\theta$ with $\delta_n = J_p/\sqrt{n}$, $\Sigma_{n,\theta}$ as defined above and
$$
\Delta_{n,\theta}(z_{1:n}) = \frac{1}{\sqrt{n}} \sum_{i=1}^n t_{i,\theta}(z_i|z_{1:i-1}).
$$
\end{theorem}

The regularity assumptions of Theorem~\ref{thm:LAMN} are natural. The only real technical restriction that is new compared to the classical theory is the separability of the $\sigma$-fields. We need it to show existence of a score $t_{i,\theta}$ in the conditional model $Q_{z_{1:i-1}}\P$ that is also jointly measurable in $z_1,\dots, z_i$. Furthermore, the weak convergence of $\Sigma_{n,\theta}$ may seem restrictive. It is clearly a necessary condition for LAMN and it may indeed fail if we choose a very erratic sequence $Q^{(n)}$ (see Example~\ref{ex:noWeakConv} in the supplement). Notice, however, that in view of $0\preccurlyeq I_\theta(Q_{z_{1:i-1}}\mathcal P) \preccurlyeq I_\theta(\P)$ (cf. Lemma~\ref{lemma:DQM}), the sequence $(\Sigma_{n,\theta})_{n\in\N}$ is bounded and hence tight. Thus, by Prokhorov's theorem, for every subsequence $n'$ there is a further subsequence $n''$ such that $\Sigma_{n'',\theta}$ is weakly convergent and Theorem~\ref{thm:LAMN} yields that $(\mathcal E_{n''})_{n''\in\N}$ is LAMN at $\theta$. We will show next that this LAMN property along subsequences is all we need to derive asymptotic lower bound results.

\begin{remark}\normalfont To prove Theorem~\ref{thm:LAMN}, one could try to use sufficient conditions for the LAMN property that are available in the literature. For example, to establish the representation of the log-likelihood-ratio $\Lambda_n$ in part~1 of Definition~\ref{def:LAMN}, one could essentially follow \citet{Jeganathan82}, who also operates under a DQM framework. However, checking those assumptions in our privacy framework is hardly any easier than directly proving the required property of $\Lambda_n$. Furthermore, to establish the weak convergence in part~2 of Definition~\ref{def:LAMN}, \citet{Jeganathan82} (who relies on the martingale CLT of \citet{Hall77}) requires that $\Sigma_{n,\theta}\to \Sigma_\theta$ in probability (on some extended probability space on which we can define $\Sigma_\theta$). However, it is not clear if this stronger assumption can be satisfied (at least along subsequences) in the present setting of $\alpha$-sequentially interactive privacy without restricting the class of possible privacy mechanisms. 
In contrast, \citet{Basawa83} use a quite elegant direct argument to establish part~2 of Definition~\ref{def:LAMN}, requiring only weak convergence of $\Sigma_{n,\theta}$, but they work with classical differentiability assumptions on the densities $p_\theta$ rather than with DQM. We therefore decided to provide a direct and fully self-contained proof of both parts of Definition~\ref{def:LAMN}, borrowing ideas from \citet{Jeganathan82} for the first part and ideas from \citet{Basawa83} for the second part. In particular, following \citet{Basawa83}, we use moment generating functions to establish the joint weak convergence of $(\Delta_{n,\theta},\Sigma_{n,\theta})$, which is possible here because $\alpha$-differential privacy imposes convenient boundedness properties. This makes for a much shorter and more elegant argument than checking assumptions of an appropriate martingale CLT; see Appendix~\ref{sec:app:LB} for the proof of Theorem~\ref{thm:LAMN}.
\end{remark}

\subsection{A private convolution theorem}
\label{sec:ConvMinMax}

In this subsection we provide a private version of a convolution theorem, which shows that any regular $\alpha$-sequentially interactive estimation procedure for $\theta\in\Theta\subseteq\R^p$ has an asymptotic distribution $D_\theta$ which is given as an (inverse) scale-mixture of the convolution of a centered Gaussian with some other distribution $L$. Here, the mixing measure is the distribution of any weak accumulation point of the sequence $(\Sigma_{n,\theta})_{n\in\N}$ defined in Theorem~\ref{thm:LAMN}. After the proof we explain how this result can be used in the case $p=1$ for determining the most concentrated limiting distribution. A regular private estimation procedure reaching that limit will thus be called asymptotically efficient. Recall that a $\alpha$-sequentially interactive estimation procedure for $\theta\in\Theta\subseteq\R^p$ consists of a sequence $Q^{(n)}$ of $\alpha$-sequentially-interactive privacy mechanisms as in \eqref{eq:Seq} from $(\X^n,\mathcal F^n)$ to $(\mathcal Z^{(n)}, \mathcal G^{(n)})$ as well as of a sequence of estimators $\hat{\theta}_n:\mathcal Z^{(n)}\to\R^p$.

\begin{theorem}\label{thm:convolution}
Fix $\alpha>0$ and for each $n\in\N$, let $Q^{(n)}$ be $\alpha$-sequentially-interactive \eqref{eq:Seq} from $(\X^n,\mathcal F^n)$ to $(\mathcal Z^{(n)}, \mathcal G^{(n)})$ and assume that all the $\sigma$-fields $\mathcal F$, $\mathcal G_i$, $i=1,\dots, n$ are countably generated. Suppose the model $\P=(P_\theta)_{\theta\in\Theta}$ on sample space $(\X,\mathcal F)$ is DQM at $\theta\in\Theta\subseteq\R^p$ and let $\Sigma_{n,\theta}$ be defined as in Theorem~\ref{thm:LAMN}. Consider the joint distribution of the sanitized data $R_{n,\theta}:=Q^{(n)}P_\theta^n$ and suppose that for every $n\in\N$, $R_{n,\theta}(\Sigma_{n,\theta} \text{ is positive definite} )=1$ and every weak accumulation point $\Sigma_\theta$ of $(\Sigma_{n,\theta})_{n\in\N}$ under $R_{n,\theta}$ is almost surely positive definite. Let $\hat{\theta}_n:\mathcal Z^{(n)}\to\R^p$ be an estimator sequence and $D_\theta$ a random $p$-vector, such that
\begin{equation}\label{eq:regularEst}
\sqrt{n}\left(\hat{\theta}_n - [\theta+h/\sqrt{n}]\right) \stackrel{R_{n,\theta+h/\sqrt{n}}}{\rightsquigarrow} D_\theta, \quad\forall h\in\R^p.
\end{equation}
Then the limiting distribution $D_\theta$ admits the representation
$$
P(D_\theta\in A) = \int_{\R^{p\times p}} \left[N(0,\sigma^{-1})\star L_{\theta}(\sigma)\right](A) P_{\Sigma_\theta}(d\sigma), \quad\forall A\in\mathcal B(\R^p),
$$
where $P_{\Sigma_\theta}$ is the distribution of any weak accumulation point $\Sigma_\theta$, $\{L_{\theta}(\sigma):\sigma\in\R^{p\times p}\}$ is a family of probability distributions on $\R^p$ possibly depending on $P_{\Sigma_\theta}$ and $\star$ denotes convolution.
\end{theorem}

\begin{proof}
Fix $h\in\R$. Notice that any subsequence of $$D_n := D_{n,\theta+h/\sqrt{n}} := \sqrt{n}\left(\hat{\theta}_n - [\theta+h/\sqrt{n}]\right)$$ under $R_{n,\theta+h/\sqrt{n}}$ has limiting distribution equal to $D_\theta$. By boundedness of $\Sigma_n := \Sigma_{n,\theta}$, the sequence $(\Sigma_n)_n$ is tight under $R_{n,\theta+h/\sqrt{n}}$. Since marginal tightness implies joint tightness, also $(\Sigma_n,D_n)_n$ is tight under $R_{n,\theta+h/\sqrt{n}}$. Thus,
$$
(\Sigma_{n'},D_{n'}) \stackrel{R_{n',\theta+h/\sqrt{n'}}}{\rightsquigarrow} (\Sigma_\theta, D_\theta),
$$
along a subsequence $(n')$. This verifies the regularity condition~7.8(b) of \citet{Hopfner14}. By contiguity (cf. Lemma~\ref{lemma:contiguity}) we also have $$\Sigma_{n'} \stackrel{R_{n',\theta}}{\rightsquigarrow} \Sigma_\theta.$$ Because $\Sigma_{n,\theta} \preccurlyeq I_\theta(\P)$ and the lower bound is positive definite (a.s.), we have $I_\theta(\P)\neq 0$ and Theorem~\ref{thm:LAMN} shows that $\mathcal E_{n'} = (\Z^{(n')}, \mathcal G^{(n')}, \{Q^{(n')}P_\theta^{n'}:\theta\in\Theta\})$, $n'\in\N$, is LAMN at $\theta$. In view of the positive definiteness assumptions our definition of LAMN coincides with the one of \citet{Hopfner14}. The proof is now finished by an application of Theorem~7.10(b) of \citet{Hopfner14}. 
\end{proof}

Now consider the case $p=1$. Notice that since $\Sigma_{n,\theta}\le \sup_{Q\in\mathcal Q_\alpha}I_\theta(Q\P)$ we must have $P(\Sigma_\theta\le \sup_{Q\in\mathcal Q_\alpha}I_\theta(Q\P))=1$, for every weak accumulation point $\Sigma_\theta$, and thus the most concentrated limiting distribution in the sense of Anderson's Lemma \citep[see, e.g., Corollary~5.7 in][]{Hopfner14} is
$$
N\left(0,\left[\sup_{Q\in\mathcal Q_\alpha}I_\theta(Q\P)\right]^{-1}\right).
$$
Hence, a regular estimator, in the sense of \eqref{eq:regularEst}, achieving this limiting distribution is called efficient. In Section~\ref{sec:MLE} below (see especially Theorem~\ref{thm:MLEefficiency}), we present such an $\alpha$-sequentially interactive estimation procedure. Notice that even though we allow for any sequentially interactive mechanism $Q^{(n)}$, the minimal asymptotic variance $\left[\sup_{Q\in\mathcal Q_\alpha}I_\theta(Q\P)\right]^{-1}$ is obtained by optimizing only over non-interactive mechanisms. Thus, we can not asymptotically beat non-interactive mechanisms  by using a sequential one. However, this does not necessarily mean that there is always an efficient non-interactive mechanism, because a maximizer of $Q\mapsto I_\theta(Q\P)$ will typically depend on the unknown $\theta$ and we need some interaction to estimate it.

In a similar way, one can derive an almost sure version of the convolution theorem, holding for Lebesgue--almost all $\theta\in\Theta$ and avoiding the regularity assumption \eqref{eq:regularEst}. Moreover, one can derive a local asymptotic minimax theorem. See Theorem~2 of \citet{Jeganathan81} and Theorem~7.12 of \citet{Hopfner14} for the corresponding classical results requiring only the LAMN property.

\begin{remark}\normalfont
Notice that we consider the case $p=1$ only in order for the maximization of Fisher-Information
$$
\sup_{Q\in\mathcal Q_\alpha}I_\theta(Q\P),
$$
or the equivalent minimization of asymptotic variance
$$
\inf_{Q\in\mathcal Q_\alpha}\left[I_\theta(Q\P)\right]^{-1} = \left[\sup_{Q\in\mathcal Q_\alpha}I_\theta(Q\P)\right]^{-1},
$$
to be well defined. In case of general $p\in\N$, it would be desirable to minimize $Q\mapsto \left[I_\theta(Q\P)\right]^{-1}\in\R^{p\times p}$ in the partial Loewner ordering. However, we currently don't know if there even exists such a minimal element. Of course, one can always fix a functional of interest $\eta:\R^{p\times p}\to\R$, such as the determinant or the trace, and try to solve 
$$
\inf_{Q\in\mathcal Q_\alpha} \eta \left(\left[I_\theta(Q\P)\right]^{-1}\right).
$$
We will consider this and other extensions elsewhere.
\end{remark}

%=================================================================================
%================================= attainability =======================================

\section{Efficiency of the private two-step MLE}
\label{sec:MLE}

In this section, at least in the single parameter case $p=1$, we develop an $\alpha$-sequentially interactive procedure that is regular and achieves the asymptotically optimal variance obtained in Subsection~\ref{sec:ConvMinMax}. In other words, we present an efficient differentially private estimation procedure. As explained in the introductory example Section~\ref{sec:Examples}, we rely on a two-step sample splitting procedure where the first (much smaller) part of the data is used in a differentially private way to consistently estimate (not necessarily at $\sqrt{n}$-rate) the unknown parameter $\theta\in\Theta\subseteq\R$. Based on this preliminary estimate, the second (larger) part of the data is then used to construct an $\alpha$-private MLE that has minimal asymptotic variance.

\subsection{Consistent non-interactive $\alpha$-private estimation}
\label{sec:consistency}

An important part of our efficient two-step procedure is a consistent and private preliminary estimator. We now present two different ways for how to obtain such consistent $\alpha$-private estimators. Throughout this subsection, we consider a parametric model $\mathcal P=(P_\theta)_{\theta\in\Theta}$, with parameter space $\Theta\subseteq\R^p$ and sample space $(\X,\mathcal F)$, that is dominated by $\mu$ with corresponding densities $p_\theta$.

\subsubsection{Maximum Likelihood}

Fix a (marginal) channel $Q\in\mathcal Q_\alpha(\X)$, pick $\nu(dz) := Q(dz|x^*)$ and $q(z|x)$ as in Lemma~\ref{lemma:Qdensities} and write $q_\theta(z) = \int_\X q(z|x)p_\theta(x)\mu(dx)$ for $\nu$-densities of the model $Q\mathcal P$. In particular, we have $q_\theta(z)\in[e^{-\alpha},e^\alpha]$. The log-likelihood of the observed (non-interactively) sanitized data $Z_1,\dots, Z_n$ is given by
$$
\theta\mapsto \ell_n(\theta) := \ell_n(\theta; z_1,\dots, z_n) := \sum_{i=1}^n \log q_\theta(z_i),
$$
and we define the corresponding non-interactive $\alpha$-private MLE $\hat{\theta}_n$ to be a measurable maximizer.
For ease of reference, we first recall Wald's classical consistency result as presented in \citet[][Theorem~5.14]{vanderVaart07}.
\begin{theorem}\label{thm:vdVConsistency}
Fix $\theta_0\in\Theta\subseteq\R^p$. Let $\theta\mapsto \log p_\theta(x)$ be upper-semicontinuous on $\Theta$ for $P_{\theta_0}$ almost all $x\in\X$ and suppose that for every sufficiently small ball $U\subseteq\Theta$ the function $x\mapsto \sup_{\theta\in U} \log p_\theta(x)$ is measurable and satisfies
$$
\E_{\theta_0}\left[\sup_{\theta\in U} \log p_\theta\right] <\infty.
$$
If $\theta\mapsto \E_{\theta_0} [\log p_\theta]$ has a unique maximizer at $\theta_0$, then any estimator $\hat{\theta}_n$ that maximizes the log-likelihood $\theta\mapsto \sum_{i=1}^n \log p_\theta(x_i)$ satisfies
$$
P_{\theta_0}^n(\|\hat{\theta}_n-\theta_0\|_2>\eps \land \hat{\theta}_n\in K)\xrightarrow[n\to\infty]{} 0,
$$
for every $\eps>0$ and every compact set $K\subseteq\Theta$.
\end{theorem}

This can be translated into the following $\alpha$-private version, which comes with even weaker assumptions than the classical one.

\begin{theorem}\label{thm:PrivConsistent}
If $\P$ is DQM at every $\theta\in\Theta\subseteq\R^p$ (in particular, $\Theta$ is open) and the channel $Q\in\mathcal Q_\alpha(\X)$ is such that the model $Q\P$ is identifiable, then any non-interactive $\alpha$-private MLE satisfies
$$
(QP_{\theta_0})^n(\|\hat{\theta}_n-\theta_0\|_2>\eps \land \hat{\theta}_n\in K)\xrightarrow[n\to\infty]{} 0,
$$
for every $\theta_0\in\Theta$, every $\eps>0$ and every compact set $K\subseteq\Theta$.
\end{theorem}
\begin{proof}
We check the assumptions of Theorem~\ref{thm:vdVConsistency} with $q_\theta$ instead of $p_\theta$. We can even establish continuity of $\theta\mapsto \log q_\theta(z)$ for every $z\in\Z$, by using the DQM property to get
\begin{align*}
&|q_{\theta+h}(z) - q_\theta(z)| \\
&\quad= \left|\int_\X q(z|x) \left[\sqrt{p_{\theta+h}(x)}-\sqrt{p_\theta(x)}\right]\left[\sqrt{p_{\theta+h}(x)}+\sqrt{p_\theta(x)}\right]\mu(dx)\right|\\
&\quad\le e^\alpha 2\sqrt{\int_\X\left( \sqrt{p_{\theta+h}(x)}-\sqrt{p_\theta(x)}\right)^2\mu(dx)} \xrightarrow[h\to0]{}0.
\end{align*}
This also implies that $z\mapsto \sup_{\theta\in U}\log q_\theta(z)$ is measurable for any ball $U\subseteq\Theta\subseteq \R^p$, because the supremum can be replaced by a countable one. $QP_{\theta_0}$-integrability of the supremum is now trivial, in view of $\log q_\theta(z)\in[-\alpha,\alpha]$. Uniqueness of the maximizer of $\theta\mapsto \E_{QP_{\theta_0}}[\log q_\theta]$ follows, for example, from \citet[][Lemma~5.35]{vanderVaart07}, using identifiablity.
\end{proof}

The proof of Theorem~\ref{thm:PrivConsistent} nicely highlights the smoothing and bounding effect that differential privacy has on the original model. Since $q_\theta(z) = \int_\X q(z|x)p_\theta(x)\mu(dx)$ is an integrated version of $p_\theta$, the continuity in $L^2$ (or even DQM) of $\P$ nicely translates into continuity of $\theta\mapsto q_\theta(z)$. The boundedness condition of Theorem~\ref{thm:vdVConsistency} immediately holds for $q_\theta$ because of the convenient property that $e^{-\alpha}\le q(z|x) \le e^\alpha$, which we get from Lemma~\ref{lemma:Qdensities}.

\subsubsection{Method of moments}

Consider first the case where the original data $X_1,\dots, X_n$ are observed. If, for some measurable function $g:\X\to G$ (often a power function $g(x) = (x,x^2,\dots,x^p$)) with range $G=\{g(x):x\in\X\}\subseteq\R^p$, the function $f:\Theta\to G$, $f(\theta) := \E_\theta[g]$ has a continuous inverse $f^{-1}:G\to\Theta\subseteq\R^p$, then the moment estimator $\hat{\theta}_n := f^{-1}(\frac{1}{n}\sum_{i=1}^n g(X_i))$ is consistent. We can turn this into a non-interactive $\alpha$-private procedure by considering the truncation operation $\Pi_{\tau} (y) := y\mathds 1_{\{\|y\|_1\le \tau\}}$, $y\in\R^p$.

\begin{theorem}
Suppose that there exists a measurable function $g:\X\to G$ with range $G\subseteq\R^p$, such that the function $f:\Theta\to G$, $f(\theta) := \E_\theta[g]$ has a continuous inverse $f^{-1}:G\to\Theta$. Generate $Z_i := \Pi_{\tau_n}[g(X_i)] + \frac{2\tau_n}{\alpha}W_i$, for independent random $p$-vectors $W_i$ with iid standard Laplace components $W_{ij}\thicksim Lap(1)$. Then the estimator
$$
\hat{\theta}_n = f^{-1}\left(\frac{1}{n}\sum_{i=1}^n Z_i\right)
$$
is non-interactive $\alpha$-private. If $\tau_n\to\infty$ and $\tau_n^2/n\to0$ as $n\to\infty$, then it is also consistent.
\end{theorem}
\begin{proof}
For $\alpha$-privacy, simply notice that a conditional Lebesgue density of $Z_i$ given $X_i$ is given by $q(z|x) = \frac{\alpha}{4\tau_n}\exp\left( -\frac{\alpha}{2\tau_n}\|z-\Pi_{\tau_n}(x)\|_1\right)$, which is easily seen to satisfy $q(z|x)/q(z|x')\le e^\alpha$, using the reverse triangle inequality.

Now $\frac{1}{n}\sum_{i=1}^n Z_i = \frac{1}{n}\sum_{i=1}^n \Pi_{\tau_n}[g(X_i)] + \frac{2\tau_n}{\alpha}\frac{1}{n}\sum_{i=1}^n W_i$. The mean of the noise term $\frac{2\tau_n}{\alpha}\bar{W}_n$ is zero and its components have variance equal to $\frac{8\tau_n^2}{n\alpha^2}$, which converges to zero as $n\to\infty$. For the $X_i$ term note that $\Pi_{\tau_n}[g(X_i)] - g(X_i) = -g(X_i)\mathds 1_{\{\|g(X_i)\|_1> \tau_n\}}$. Hence, 
$$
\E \left\|\frac{1}{n}\sum_{i=1}^n \Big[\Pi_{\tau_n}[g(X_i)] - g(X_i)\Big]\right\|_1 \le \E \left\| g(X_1)\mathds 1_{\{\|g(X_1)\|_1> \tau_n\}}\right\|_1 \xrightarrow[n\to\infty]{}0,
$$
by dominated convergence. Hence $\frac{1}{n}\sum_{i=1}^n Z_i = \frac{1}{n}\sum_{i=1}^n g(X_i) + o_{P}(1)$ converges in probability to $f(\theta)$ and the proof is finished by an application of the continuous mapping theorem.
\end{proof}

We want to point out that the truncation approach actually comes at the price of a convergence rate of only $\tau_n/\sqrt{n}$ rather than the usual $n^{-1/2}$. This is not an issue for the theory of our two-step procedure, since we don't require a $\sqrt{n}$-consistent preliminary estimator, but it may of course be an issue for applications. Notice, however, that it may often be possible to avoid the truncation $\Pi_{\tau_n}$ if the function $g$ can be chosen such that its range space $G\subseteq\R^p$ is bounded by, say, $\|g(x)\|_1\le c_p$, for all $x\in\X$, because then we can choose $\tau_n=c_p$. This typically works if $\X$ itself is bounded and $g(x)=x$, but can also be used for unbounded $\X$.

\subsection{Consistent estimation of the optimal privacy mechanism}
\label{sec:consistentMechanism}
Recall that in order to privately estimate $\theta\in\Theta\subseteq\R$ from iid $Z_1,\dots, Z_n \thicksim QP_\theta$, we would ideally like to choose $Q\in\mathcal Q_\alpha(\X)$ such that $I_\theta(Q\P)\in\R$ is maximal. However, as we have seen in Section~\ref{sec:Examples}, a maximizer, if it exists, will typically depend on the unknown parameter $\theta$. In the previous section we saw how $\theta$ can be consistently estimated in a locally differentially private way, given that a preliminary privacy mechanism $Q_0\in\mathcal Q_\alpha(\X\to\mathcal Z_0)$ can be found such that the model $Q_0\P$ is identifiable. We will use such a consistent $\alpha$-private estimator $\tilde{\theta}_{n_1}$ to construct a sequence $\hat{Q}_{n_1}\in\mathcal Q_\alpha(\X)$ such that $I_\theta(\hat{Q}_{n_1}\P)\to\sup_{Q\in\mathcal Q_\alpha(\X)} I_\theta(Q\P)$ in $(Q_0P_\theta)^{n_1}$-probability, as the size $n_1$ of the first part of the sample grows to infinity. An ad-hoc idea to obtain $\hat{Q}_{n_1}$ would be to solve $\sup_{Q\in\mathcal Q_\alpha(\X)} I_{\tilde{\theta}_{n_1}}(Q\P)$. However, direct optimization over the vast set $\mathcal Q_\alpha$ is numerically challenging. Instead, we solve a discretized version of this optimization problem of which we can show that it approximates the original one well enough.

Before we outline the procedure for constructing $\hat{Q}_{n_1}$, we provide a complete list of all the regularity conditions imposed in this subsection. Even though our main results on maximizing Fisher-Information apply only in the case $p=1$, we here list our assumptions with general $p\in\N$ because some of our preliminary results hold generally and may be of independent interest.

\begin{enumerate}
        \setlength\leftmargin{-20pt}
\renewcommand{\theenumi}{(C\arabic{enumi})}
\renewcommand{\labelenumi}{\textbf{\theenumi}}

\item \label{C.DQM} The model $(\X,\mathcal F, \P)$ is DQM at every $\theta\in\Theta\subseteq\R^p$ with score $s_\theta$ and dominating measure $\mu$.

\item \label{C.measurable} The mappings $(x,\theta)\mapsto p_\theta(x): \X\times\Theta\to\R_+$ and $(x,\theta)\mapsto s_\theta(x): \X\times\Theta\to\R^p$ are measurable.

\item \label{C.L2cont} The mapping $\theta\mapsto s_\theta\sqrt{p_\theta}: \Theta\to L_2(\mu,\|\cdot\|_2)$ is continuous.

\end{enumerate}

Conditions~\ref{C.DQM} and \ref{C.L2cont} together with positive definiteness of $I_\theta(\P)$ for all $\theta\in\Theta$, correspond exactly to what \citet{Bickel93} call \emph{regularity} of the parametric model $\P$.
In our leading case $p=1$, positivity of the Fisher-Information is not needed for consistent estimation of the optimal privacy mechanism, because this problem is trivial if $I_\theta(\P) = 0$. However, later on it will be necessary to ensure $I_\theta(\P)\ge \sup_{Q\in\mathcal Q_\alpha(\X)} I_\theta(Q\P)>0$ in order to prove efficiency of the two-step private MLE. The measurability condition~\ref{C.measurable} will be used to construct a measurable discretization of the original model and to show that the Fisher-Information of the discretized model and a maximizer thereof are appropriately measurable.

\subsubsection{Outline of general approach}

For a given $k\in\N$, consider a measurable (discretization) map $T_k:\X\to\X_k:=[k]$. Before applying a privacy mechanism, we discretize each observation $X_i$ by recording only $Y_i = T_k(X_i)$.\footnote{Clearly, local pre-processing preserves privacy (cf. Lemma~\ref{lemma:pre-postProcessing}).} Thus, the model describing the quantized observations $Y_i$ is given by $T_k\P = (P_\theta T_k^{-1})_{\theta\in\Theta}$.
With such a discretization at hand, we solve the optimization problem
\begin{equation}\label{eq:finitedimlOPt}
\max_{Q\in\mathcal Q_\alpha(\X_k\to\X_k)} I_{\tilde{\theta}_{n_1}}(QT_k\P),
\end{equation}
using the LP formulation of Theorem~4 in \citet{Kairouz16} (see also Lemma~\ref{lemma:finitedimlOpti} below, which shows that the Fisher-Information satisfies the assumptions of that reference and is continuous in $Q$ on the compact set $\mathcal Q_\alpha(\X_k\to\X_k)$), to obtain $\hat{Q}_{n_1}\in\mathcal Q_\alpha(\X_k\to\X_k)\subseteq\R^{k\times k}$. In order for the procedure to apply in high generality, we will even allow $k$ and the quantizer $T_k$ to depend on the preliminary estimate $\tilde{\theta}_{n_1}$ and therefore we write $\hat{k}_{n_1}\in\N$ and $\hat{T}_{{n_1}}:\X\to[\hat{k}_{n_1}]$.\footnote{In the applications of Section~\ref{sec:appl} below, however, we find that $k$ can be chosen independently of the preliminary estimate.} We then show that 
\begin{equation}\label{eq:consistentQ}
I_\theta(\hat{Q}_{n_1}\hat{T}_{{n_1}}\P) \to \sup_{Q\in\mathcal Q_\alpha(\X)} I_\theta(Q\P), \quad\quad \text{in }(Q_0P_\theta)^{n_1}\text{-probability},
\end{equation}
as $n_1\to\infty$. Thus, given this channel $\hat{Q}_{n_1}$ and the discretization map $\hat{T}_{{n_1}}$ (both of which are constructed using the preliminary estimate $\tilde{\theta}_{n_1}$ based on sanitized data from a first group of individuals $i=1,\dots, n_1$) the second group of data owners ($i=n_1+1,\dots,n$) can locally generate $Z_i$ by first discretizing their data $X_i$ to get $Y_i=\hat{T}_{{n_1}}(X_i)$ and then draw $Z_i$ from $\hat{Q}_{n_1}(dz_i|Y_i)$ (cf. Figure~\ref{fig:Qn}). Notice that the composition $\hat{Q}_{n_1}\hat{T}_{{n_1}}$ can be understood as a collection of Markov kernels indexed by the first group data $z_{1:n_1}$, or simply by $\tilde{\theta}_{n_1}$. In other words, we can think of it as a function $(C,x_i,z_{1:n_1})\mapsto [\hat{Q}_{n_1}\hat{T}_{{n_1}}](C|x_i,z_{1:n_1})$, for $C\subseteq \N$, $x_i\in\X$ and $z_{1:n_1}\in\mathcal Z_0^{n_1}$. In particular, $[\hat{Q}_{n_1}\hat{T}_{{n_1}}]_{z_{1:{n_1}}} \in \mathcal Q_\alpha(\X\to[\hat{k}_{n_1}(z_{1:n_1})])$, for all first group data samples $z_{1:n_1}\in\mathcal Z_0^{n_1}$ by construction. Consequently, the overall privacy protocol can be formally expressed as an $\alpha$-sequential mechanism 
\begin{align}\label{eq:Q^n}
Q^{(n)}(dz_{1:n}|x_{1:n}):=\left[\bigotimes_{i=n_1+1}^n[\hat{Q}_{n_1}\hat{T}_{{n_1}}](dz_i|x_i,z_{1:{n_1}})\right] \left[\bigotimes_{j=1}^{n_1} Q_0(dz_j|x_j)\right],
\end{align}
provided we can also show that $[\hat{Q}_{n_1}\hat{T}_{{n_1}}] \in \mathfrak P(\X\times \mathcal Z_0^{n_1}\to\N)$, that is, provided we can show measurability of $(x_i, z_{1:{n_1}})\mapsto[\hat{Q}_{n_1}\hat{T}_{n_1}](C|x_i,z_{1:{n_1}})$, for all $C\subseteq \N$ (see Lemma~\ref{lemma:measurability} and the ensuing discussion below).
That efficient estimation based on sanitized $Z_i$ generated from $Q^{(n)}$ as above is indeed possible will be established in Subsection~\ref{sec:MLEefficiency}. The necessary consistency claim in \eqref{eq:consistentQ}, including the mentioned measurability, will be rigorously proved in the following Subsection~\ref{sec:Technical}.

%==========

\subsubsection{Technical aspects}
\label{sec:Technical}
We begin the rigorous proof of the consistency claim in \eqref{eq:consistentQ} by a discuss of the precise requirements on the discretization map $T_k$. The following definition is motivated by Lemma~\ref{lemma:Delta_k} below.

\begin{definition}\label{def:quantizer}
Suppose Condition~\ref{C.DQM} holds and let $\dot{p}_\theta(x) := s_\theta(x) p_\theta(x)\in\R^p$. A sequence of measurable mappings $(x,\theta)\mapsto T_{m}(x,\theta): \X\times\Theta\to\N$, $m\in\N$, is called a \emph{consistent quantizer} of the model $\P$, if there exists another sequence of measurable mappings $k_m:\Theta\to\N$, such that the following hold true:
\begin{enumerate}[a)]
\item For all $m\in\N$, $\theta\in\Theta$ and $x\in\X$ we have $T_{m}(x,\theta) \in [k_m(\theta)] := \{1,\dots, k_m(\theta)\}.$
\item Define the following quantities: $T_{m,\theta}(x) := T_m(x,\theta)$, $B_j := B_{j,m}(\theta):= T_{m,\theta}^{-1}(\{j\})$, $J:=J_m(\theta):= \{l\in[k_m(\theta)]: 0<\mu(B_{l,m}(\theta))<\infty\}$, $K:=K_{J,m}(\theta) := \bigcup_{j\in J}B_{j,m}(\theta)$,
$$
\bar{p}_\theta (x) := \sum_{j\in J} \frac{\int_{B_j}p_\theta(y)\mu(dy)}{\mu(B_j)} \mathds 1_{B_j}(x), 
\; \text{ and }\;
\bar{\dot{p}}_\theta(x) := \sum_{j\in J} \frac{\int_{B_j}\dot{p}_\theta(y)\mu(dy)}{\mu(B_j)} \mathds 1_{B_j}(x) \in\R^p.
$$
We then have $\Delta_m(\theta_m)\to0$ as $m\to\infty$, for every sequence $(\theta_m)_{m\in\N}$ in $\Theta$ whose limit is also in $\Theta$, and where
\begin{align}
\Delta_m(\theta):=\left[\Big\|(p_\theta - \bar{p}_\theta)\mathds 1_{K(\theta)}\Big\|_{L_1(\mu)} + \Big\|\|\dot{p}_\theta - \bar{\dot{p}}_\theta\|_{\ell_2}\mathds 1_{K(\theta)}\Big\|_{L_1(\mu)} + \sqrt{P_\theta(K(\theta)^c)} \right].\label{eq:T_kErr}
\end{align} 
\end{enumerate}
\end{definition}

The idea of a consistent quantizer is that we want to be able to approximate the Fisher-Information $I_\theta(Q\P)$ of a privatized model $Q\P$ by the Fisher-Information $I_\theta(QT_{m,\theta}\P)$  of a quantized and privatized model at the value of the preliminary estimate $\theta=\tilde{\theta}_{n_1}$, hence, the convergence is required along arbitrary sequences $(\theta_m)_{m\in\N}$ and not just pointwise. To approximate Fisher-Information, we need to approximate the density $p_\theta$ and its `derivative' $\dot{p}_\theta$, which determine the Fisher-Information at $\theta$, by histograms. But to obtain negligibility of $\Delta_m(\theta_m)$ for every convergent sequence $\theta_m\in\Theta$, it may be necessary to allow the corresponding partition, and in some cases even the resolution $k_m$, to depend on $\theta_m$. For instance, if the quantizer, and thus the set $K$, wasn't allowed to depend on $\theta_m$, then we would need to impose tightness of the model $\P$ in order to guarantee that $P_{\theta_m}(K_m^c) \to 0$ holds for every convergent sequence.  

The following lemma shows that consistent quantizers generally exist and the proof of that result, which is deferred to the supplement (Section~\ref{sec:proofsEstQ}), also provides an explicit construction. However, in specific models one can often find simpler constructions with a $\theta$-independent $k_m$ that also scales slowly with $m$ while the approximation error in \eqref{eq:T_kErr} converges faster (see Section~\ref{sec:appl}). The magnitude of $k_m$ is, of course, essential for the numerical solvability of \eqref{eq:finitedimlOPt}.

\begin{lemma}\label{lemma:ConsQuantExists} If $\X\subseteq\R^d$ and Conditions~\ref{C.DQM} and \ref{C.measurable} hold for a dominating measure $\mu$ that is finite on compact sets, then a consistent quantizer of $\P$ exists.
\end{lemma}

Assuming that we have a consistent quantizer at hand, the next step is to ensure that the objective function in \eqref{eq:finitedimlOPt} is well defined, that is, we need to show that for $Q\in\mathcal Q_\alpha([k_m(\theta)]\to\mathcal Z)$ the model $QT_{m,\theta}\P$ is DQM and thus possesses a finite Fisher-Information. But this follows upon noticing that $Q[T_{m,\theta}\P] = [QT_{m,\theta}]\P$, because $Q[T_{m,\theta}P_{\theta'}] = \int_{[k_m(\theta)]} Q(\cdot|y) [P_{\theta'} T_{m,\theta}^{-1}](dy) = \int_{\X} Q(\cdot|T_{m,\theta}(x)) P_{\theta'}(dx) = [QT_{m,\theta}]P_{\theta'}$, and using Lemma~\ref{lemma:pre-postProcessing}, Lemma~\ref{lemma:DQM} and the fact that $\P$ is assumed to be DQM.

Now let us study some properties of the finite dimensional optimization problem \eqref{eq:finitedimlOPt}. Among other things, the following lemma shows that we are maximizing a continuous convex function on a convex compact set. In particular, the maximum is achieved. The proofs of all the following results are deferred to Section~\ref{sec:proofsEstQ}.

\begin{lemma}\label{lemma:finitedimlOpti}
Define the set $\mathcal M_\alpha(\ell,k)$ consisting of all $\ell\times k$ column stochastic matrices $Q$ (i.e., $Q_{ij}\ge0$ and $\sum_{i=1}^\ell Q_{ij} = 1$) with the following property: Two elements $Q_{ij}$ and $Q_{ij'}$ of any given row $i$ of $Q$ satisfy $e^{-\alpha}Q_{ij'}\le Q_{ij} \le e^\alpha Q_{ij'}$. Furthermore, consider a model $(\X,\mathcal F,\P)$ that is DQM at $\theta\in\Theta\subseteq\R^p$ with score $s_\theta$ and a measurable mapping $T:\X\to[k]$. Write $p_\theta := (P_\theta T^{-1}(\{j\}))_{j\in[k]}\in\R^k$ and $\dot{p}_\theta := (\int_{T^{-1}(\{j\})} s_\theta(x)p_\theta(x)\mu(dx))_{j\in[k]}\in\R^{p\times k}$.
\begin{enumerate}[i)]
\item\label{lemma:finitedimlOpti:i} The set $\mathcal Q_\alpha([k]\to[\ell])$ can be identified with $\mathcal M_\alpha(\ell,k)$ as follows: For $M\in\mathcal M_\alpha(\ell,k)$, define $Q(C|y) := \sum_{i\in C} M_{i,y}$, $C\subseteq[\ell]$, $y\in[k]$. For $Q \in \mathcal Q_\alpha([k]\to[\ell])$, define $M_{i,y} := Q(\{i\}|y)$, $i\in[\ell]$, $y\in[k]$.
\item\label{lemma:finitedimlOpti:ii} $\mathcal M_\alpha(\ell,k)\subseteq\R^{\ell\times k}$ is a compact polytope. \citep[Lemma~23 of][]{Kairouz16}
\item\label{lemma:finitedimlOpti:iii} Consider the convex cone $\mathcal C_\alpha(k) := \{v\in\R^k: 0\le v_j\le 1, v_j\le e^\alpha v_{j'}, \forall j,j'\in[k]\}$ and for $\theta\in\Theta$ and $v\in\mathcal C_\alpha(k)$, define $g_\theta(v) := \frac{\dot{p}_\theta vv^T\dot{p}_\theta^T}{v^Tp_\theta}\in\R^{p\times p}$, $g_\theta(0):=0$. Then $I_\theta(QT\P) = \sum_{i=1}^\ell g_\theta(Q_{i\cdot}^T)$ for all $Q\in\mathcal M_\alpha(\ell,k)$, and $g_\theta:\mathcal C_\alpha(k) \to \R^{p\times p}$ is continuous. Here, $Q_{i\cdot}\in\R^k$ denotes the $i$-th row of the matrix $Q$.
\item\label{lemma:finitedimlOpti:iv} If $p=1$, then $g_\theta: \mathcal C_\alpha(k) \to \R$ is convex and sublinear.
\end{enumerate} 
\end{lemma}

To rigorously resolve the measurability issue mentioned in the previous subsection, let us reconsider the optimization problem \eqref{eq:finitedimlOPt} but now explicitly including the dependence of the quantizer $T_{m,\theta}$ on $\theta$. The proof of the following result is deferred to the supplement.

\begin{lemma}\label{lemma:measurability}
If Conditions~\ref{C.DQM} and \ref{C.measurable} hold with $p=1$ and $k:\Theta\to\N$ and $T:\X\times\Theta\to\N$ are measurable with $T_{\theta}(x) := T(x,\theta)\le k(\theta)$ for all $(x,\theta)\in\X\times\Theta$, then there exists a Markov kernel $Q\in\mathfrak P(\N\times\Theta\to\N)$, such that $Q_\theta(C|y) := Q(C|y,\theta)$ ($\theta\in\Theta$, $C\subseteq[k(\theta)]$, $y\in [k(\theta)]$) satisfies for all $\theta\in\Theta$, $Q_\theta\in\mathcal Q_\alpha([k(\theta)]\to[k(\theta)])$ and
\begin{equation}\label{eq:lemma:measurability}
I_\theta(Q_\theta T_{\theta}\P) = \max_{Q\in\mathcal Q_\alpha([k(\theta)]\to[k(\theta)])}I_\theta(QT_{\theta}\P).
\end{equation}
Furthermore, for every $\theta_0\in\Theta$ the function $\theta\mapsto I_{\theta_0}(Q_\theta T_\theta\P):\Theta\to\R_+$ is measurable.
\end{lemma}

With Lemma~\ref{lemma:measurability} at hand, we can now complete the definition of the optimal sequential mechanism $Q^{(n)}$ in \eqref{eq:Q^n} by setting $\hat{Q}_{n_1}(C|y,z_{1:n_1}) := Q(C|y,\tilde{\theta}_{n_1}(z_{1:n_1}))$ and $\hat{T}_{n_1}(x) := T(x,\tilde{\theta}_{n_1}(z_{1:n_1}))$ for $T$ and $Q$ as in Lemma~\ref{lemma:measurability}. In particular, we see that the composition $[\hat{Q}_{n_1}\hat{T}_{n_1}](C|x,z_{1:n_1}) = Q(C|T(x,\tilde{\theta}_{n_1}(z_{1:n_1})),\tilde{\theta}_{n_1}(z_{1:n_1}))$ defines a Markov kernel in $\mathfrak P(\X\times\mathcal Z_0^{n_1}\to\N)$ with $[\hat{Q}_{n_1}\hat{T}_{n_1}]_{z_{1:n_1}}\in\mathcal Q_\alpha(\X\to\N)$ for all $z_{1:n_1}\in\mathcal Z_0^{n_1}$. Obviously, $T$ will eventually be replaced by a consistent quantizer $T_m$ for $m=n_1$. We have thus finished the definition and proof of existence of the optimal privacy mechanism. The remainder of this section is now concerned with the consistency property in \eqref{eq:consistentQ}.

\medskip

We begin with a result which shows that optimizing $Q\mapsto I_\theta(Q\P)$ over $\mathcal Q_\alpha(\X)$, and $Q_k\mapsto I_\theta(Q_kT_k\P)$ over $\mathcal Q_\alpha([k])$ is nearly equivalent in an appropriate sense. To make this precise, fix an arbitrary $Q\in\mathcal Q_\alpha(\X\to\mathcal Z)$, $x^*\in\X$ and define $\nu(dz):= Q(dz|x^*)$. Let $\mu$ be the $\sigma$-finite measure that dominates $\P$ and for $F\in\mathcal F$, $B_j := T_k^{-1}(\{j\})$ and $j\in J:= \{l\in[k]: 0<\mu(B_l)<\infty\}$, define the probability measure $P_j(F) := \frac{\mu(B_j\cap F)}{\mu(B_j)}$. Equipped with this notation we can now define a `projection' of $Q$ onto $\mathcal Q_\alpha([k])$ with respect to the partition $(B_j)_{j\in[k]}$ by
\begin{align}\label{eq:ProjQ}
Q_k(C|j) := \begin{cases}
QP_j(C), &j\in J,\\
\nu(C), &\text{else,}
\end{cases}\quad C\in\mathcal G, j\in[k].
\end{align}
Clearly, $Q_k$ as defined above is a Markov kernel, that is, $Q_k\in\mathfrak P([k]\to\mathcal Z)$. Lemma~\ref{lemma:Q_kPriv} in the supplement shows that it is also $\alpha$-differentially private, hence $Q_k\in\mathcal Q_\alpha([k]\to\mathcal Z)$. 

\begin{lemma}\label{lemma:Delta_k}
Fix $k\in\N$, $\alpha\in(0,\infty)$ and $Q\in\mathcal Q_\alpha(\X\to \mathcal Z)$. Suppose that $(\X,\mathcal F,\P)$ is DQM at $\theta\in\Theta\subseteq\R^p$ with respect to $\mu$ and with score $s_\theta$ and set $\dot{p}_\theta(x) := s_\theta(x) p_\theta(x)\in\R^p$. If $T:\X\to[k]$ is measurable with resulting partition $B_j := \{x\in\X: T(x)=j\}$, $j=1,\dots, k$, and $Q_k\in \mathcal Q_\alpha([k]\to\mathcal Z)$ is the corresponding projection of $Q$ defined in \eqref{eq:ProjQ}, then we have
\begin{align*}
&\| I_\theta(Q\P) - I_\theta(Q_kT\P)\|_2 \\
&\quad \le C_\alpha (1\lor\trace I_\theta(\P)) \left[ \Big\|(p_\theta - \bar{p}_\theta)\mathds 1_{K_J}\Big\|_{L_1(\mu)} + \Big\|\|\dot{p}_\theta - \bar{\dot{p}}_\theta\|_2\mathds 1_{K_J}\Big\|_{L_1(\mu)} + \sqrt{P_\theta(K_J^c)}\right],
\end{align*} 
where $C_\alpha>0$ depends only on $\alpha$, $J:= \{l\in[k]: 0<\mu(B_l)<\infty\}$, $K_J := \bigcup_{j\in J}B_j$, 
$$
\bar{p}_\theta (x) := \sum_{j\in J} \frac{\int_{B_j}p_\theta(x)\mu(dx)}{\mu(B_j)} \mathds 1_{B_j}(x), 
\quad \text{ and }\quad
\bar{\dot{p}}_\theta(x) := \sum_{j\in J} \frac{\int_{B_j}\dot{p}_\theta(x)\mu(dx)}{\mu(B_j)} \mathds 1_{B_j}(x) \in\R^p.
$$
\end{lemma}

\begin{remark}\normalfont
From this lemma it is apparent that consistent quantizers as in Definition~\ref{def:quantizer} are exactly what we need to consider for $T$.
\end{remark}

At first sight, one may think that Lemma~\ref{lemma:Delta_k} is already sufficient to relate the infinite dimensional optimization problem 
$$
\sup_{Q\in\mathcal Q_\alpha(\X)} I_{\tilde{\theta}_{n_1}}(Q\P)
$$
to the parametric problem \eqref{eq:finitedimlOPt}, provided that the error term is small. However, the channel $Q_k$ in that lemma still uses an arbitrary (possibly infinite dimensional) output space $\Z$, rather than generating sanitized observations from the finite (!) and fixed input space $\X_k=[k]$, like those in \eqref{eq:finitedimlOPt}. Thus, we still have to show that it actually suffices to consider only $\alpha$-private mechanisms in $\mathcal Q_\alpha([k]\to[k])$. Fortunately, this is only a technical step which requires some work but no further assumptions.

\begin{lemma}\label{lemma:finiteReduction}
Fix $k\in\N$, $\alpha\in(0,\infty)$ and suppose the model $\P=(P_\theta)_{\theta\in\Theta}$ is DQM at $\theta\in\Theta\subseteq\R$. Then, for any measurable $T:\X\to[k]$, we have
$$
\sup_{Q\in \mathcal Q_\alpha([k])} I_\theta(QT\P) = \sup_{Q\in\mathcal Q_\alpha([k]\to[k])} I_\theta(QT\P).
$$
\end{lemma}

Finally, we need to relate the feasible objective function $Q\mapsto I_{\tilde{\theta}_{n_1}}(Q\P)$ with a consistent estimate $\tilde{\theta}_{n_1}\approx\theta$ to the true but infeasible one $Q\mapsto I_\theta(Q\P)$. That is, we need to establish an appropriate continuity property.

\begin{lemma} \label{lemma:ContInTheta}
Consider a model $\P=(P_\theta)_{\theta\in\Theta}$, $\Theta\subseteq\R^p$, on sample space $\X$ that is DQM at the points $\theta, \theta'\in\Theta$ with scores $s_\theta$ and $s_{\theta'}$. Then
\begin{align}
&\sup_{Q\in\mathcal{Q}_\alpha(\X)}\left\|I_\theta(Q \P) - I_{\theta'}(Q \P)\right\|_2 \notag\\
&\le 
e^{2\alpha}\left(\trace[I_\theta(\P)]\lor \trace[I_{\theta'}(\P)]\lor1\right)
\left[
2\Big\|\|\dot{p}_\theta-\dot{p}_{\theta'}\|_{\ell_2}\Big\|_{L_1(\mu)} + 3\|{p}_\theta - {p}_{\theta'}\|_{L_1(\mu)}
\right] =: \varphi(\theta,\theta'),\label{eq:UniformBoundIeps}
\end{align}
with $\dot{p}_\theta(x) := s_\theta(x) p_\theta(x)\in\R^p$. Moreover, if Conditions~\ref{C.DQM} and \ref{C.L2cont} hold, then $\theta'\mapsto \varphi(\theta,\theta')$ is continuous at $\theta$ for every $\theta\in\Theta$.
\end{lemma}

Lemma~\ref{lemma:ContInTheta} will be applied with the true parameter $\theta$ and $\theta'=\tilde{\theta}_{n_1}$ a consistent $\alpha$-private preliminary estimator. Thus, by the continuous mapping theorem, the error term $\varphi(\theta,\tilde{\theta}_{n_1})$ will be asymptotically negligible.
The main result of this subsection, the formal version of \eqref{eq:consistentQ}, can now be stated and proved. 

\begin{theorem}\label{thm:Qhat}
Suppose that Conditions~\ref{C.DQM}, \ref{C.measurable} and \ref{C.L2cont} are satisfied with $p=1$ and let $(T_m)_{m\in\N}$ be a consistent quantizer with associated sequence $(k_m)_{m\in\N}$ as in Definition~\ref{def:quantizer}. Moreover, fix some $Q_0\in\mathcal Q_\alpha(\X\to\mathcal Z_0)$, $n_1\in\N$, let $\tilde{\theta}_{n_1}:\mathcal Z_0^{n_1}\to\Theta$ be an estimator and write $\hat{T}_{{n_1}}(x) := T_{{n_1}}(x,\tilde{\theta}_{n_1})$ and $\hat{k}_{n_1} := k_{n_1}(\tilde{\theta}_{n_1})$. Furthermore, for $\mathcal M_\alpha(k,k)\subseteq\R^{k\times k}$ as in Lemma~\ref{lemma:finitedimlOpti}, let $\hat{Q}_{n_1}\in\mathcal M_\alpha(\hat{k}_{n_1},\hat{k}_{n_1})$ be a measurable maximizer (in the sense of Lemma~\ref{lemma:measurability}) of
$$
\max_{Q\in\mathcal M_\alpha(\hat{k}_{n_1},\hat{k}_{n_1})} I_{\tilde{\theta}_{n_1}}(Q\hat{T}_{{n_1}}\P).
$$
Then $I_{\theta_{n_1}}(\hat{Q}_{n_1}\hat{T}_{{n_1}}\P)$ is measurable as a function of the first group data $z_{1:n_1}\in\mathcal Z_0^{n_1}$ and
$$
0\le \sup_{Q\in\mathcal Q_\alpha(\X)} I_{\theta}(Q\P) - I_{\theta_{n_1}}(\hat{Q}_{n_1}\hat{T}_{{n_1}}\P) \le 
2\varphi(\theta,\tilde{\theta}_{n_1})+ \varphi(\theta,{\theta}_{n_1}) + C_\alpha (1\lor I_{\tilde{\theta}_{n_1}}(\P))\Delta_{{n_1}},
$$
for any $\theta,\theta_{n_1}\in\Theta$, where $\varphi$ is as in Lemma~\ref{lemma:ContInTheta}, $C_\alpha>0$ depends only on $\alpha$ and $\Delta_{n_1} := \Delta_{n_1}(\tilde{\theta}_{n_1})$ is as in Definition~\ref{def:quantizer}. Moreover, the upper bound converges to zero in $[Q_0P_\theta]^{n_1}$-probability, as $n_1\to\infty$, provided that $\theta_{n_1}\to\theta$ and $[Q_0P_\theta]^{n_1}(|\tilde{\theta}_{n_1} - \theta| > \eps) \to 0$ as $n_1\to\infty$, for every $\eps>0$.
\end{theorem}

\begin{proof}
Measurability of $I_{\theta_{n_1}}(\hat{Q}_{n_1}\hat{T}_{{n_1}}\P)$ follows immediately from Lemma~\ref{lemma:measurability}. For the upper bound we begin with an application of Lemma~\ref{lemma:ContInTheta}, 
%using the fact that $[\hat{Q}_{n_1}\hat{T}_{k_{n_1}}]_{z_{1:n_1}}\in\mathcal Q_\alpha(\X)$, $\forall z_{1:n_1}\in\mathcal Z_0^{n_1}$, which holds by construction and Lemma~\ref{lemma:pre-postProcessing}, 
to get
$$
\sup_{Q\in\mathcal Q_\alpha(\X)} I_{\theta}(Q\P) \le \sup_{Q\in\mathcal Q_\alpha(\X)} I_{\tilde{\theta}_{n_1}}(Q\P) + \varphi(\theta,\tilde{\theta}_{n_1}).
$$
Next, use Lemma~\ref{lemma:Delta_k} with $\theta=\tilde{\theta}_{n_1}$, $k=\hat{k}_{n_1}$ and $T=\hat{T}_{{n_1}}$, to get
$$
\sup_{Q\in\mathcal Q_\alpha(\X)} I_{\tilde{\theta}_{n_1}}(Q\P)  \le \sup_{Q\in\mathcal Q_\alpha([\hat{k}_{n_1}])} I_{\tilde{\theta}_{n_1}}(Q\hat{T}_{{n_1}}\P)  + C_\alpha (1\lor I_{\tilde{\theta}_{n_1}}(\P)){\Delta}_{n_1}(\tilde{\theta}_{n_1}).
$$
Combining the previous two inequalities and applying Lemmas~\ref{lemma:finiteReduction} and \ref{lemma:finitedimlOpti} and using the fact that $\hat{Q}_{n_1}$ is a maximizer, yields
\begin{align*}
\sup_{Q\in\mathcal Q_\alpha(\X)} I_{\theta}(Q\P) &\le
\sup_{Q\in\mathcal Q_\alpha([\hat{k}_{n_1}]\to[\hat{k}_{n_1}])}I_{\tilde{\theta}_{n_1}}(Q\hat{T}_{{n_1}}\P)+
\varphi(\theta,\tilde{\theta}_{n_1}) + C_\alpha (1\lor I_{\tilde{\theta}_{n_1}}(\P)){\Delta}_{n_1}(\tilde{\theta}_{n_1})\\
&=
I_{\tilde{\theta}_{n_1}}(\hat{Q}_{n_1}\hat{T}_{{n_1}}\P)+
\varphi(\theta,\tilde{\theta}_{n_1}) + C_\alpha (1\lor I_{\tilde{\theta}_{n_1}}(\P)){\Delta}_{n_1}(\tilde{\theta}_{n_1}).
\end{align*}
Applying Lemma~\ref{lemma:ContInTheta} twice, we can further bound this quantity by
$$
I_{{\theta}_{n_1}}(\hat{Q}_{n_1}\hat{T}_{{n_1}}\P)+
2\varphi(\theta,\tilde{\theta}_{n_1})+ \varphi(\theta,{\theta}_{n_1}) + C_\alpha (1\lor I_{\tilde{\theta}_{n_1}}(\P)){\Delta}_{n_1}(\tilde{\theta}_{n_1}).
$$
Using Tonelli's theorem it is easy to see from Definition~\ref{def:quantizer} that under our present assumptions $\theta\mapsto \Delta_{n_1}(\theta)$ is measurable. Moreover, by almost sure convergence along subsequences we can show that ${\Delta}_{n_1}(\tilde{\theta}_{n_1})$ converges to zero in $(Q_0P_\theta)^{n_1}$-probability. 
The desired convergence to zero of the upper bound now follows from continuity of $\varphi$ (cf. Lemma~\ref{lemma:ContInTheta}) and Lemma~\ref{lemma:continuity}.\eqref{lemma:continuity:trace}.
\end{proof}

%==================================================================

\subsection{Asymptotic efficiency of the two-step $\alpha$-private MLE}
\label{sec:MLEefficiency}
In this subsection we prove that the MLE based on sanitized data from the mechanism $\hat{Q}_{n_1}\hat{T}_{n_1}$ introduced in Subsection~\ref{sec:consistentMechanism} is regular and efficient. We begin by describing the complete locally private two-step estimation procedure: Given are private data $X_1,\dots, X_n \stackrel{iid}{\thicksim} P_{\theta}$, $\theta\in\Theta\subseteq\R$, from a regular statistical model $\P=\{P_\theta:\theta\in\Theta\}$ on sample space $\X\subseteq\R^d$ dominated by $\mu$.

\begin{enumerate}
\item Fix a preliminary privacy mechanism $Q_0\in\mathcal Q_\alpha(\X)$, such that $Q_0\P$ is identifiable.

\item Consider a non-interactive $\alpha$-private preliminary estimator $\tilde{\theta}_{n_1}$ based on sanitized versions $Z_1,\dots, Z_{n_1}$ of $X_1,\dots, X_{n_1}$ generated from $Q_0$, which is consistent for $\theta$ (cf. Subsection~\ref{sec:consistency}), that is, which satisfies
$$
\left[ Q_0P_{\theta}\right]^{n_1} \left( |\tilde{\theta}_{n_1} - \theta| > \eps\right) \xrightarrow[n_1\to\infty]{} 0, \quad \forall \eps>0, \forall \theta\in\Theta.
$$

\item Fix a consistent quantizer $(T_m)_{m\in\N}$ as in Definition~\ref{def:quantizer} with associated sequence $(k_m)_{m\in\N}$ and set $\hat{T}_{n_1}(x) := T_{n_1}(x,\tilde{\theta}_{n_1})$ and $\hat{k}_{n_1}:= k_{n_1}(\tilde{\theta}_{n_1})$.

\item Let $\hat{Q}_{n_1}$ be a measurable maximizer 
\begin{equation}\label{eq:max}
\hat{Q}_{n_1}\in\argmax_{Q\in\mathcal M_\alpha(\hat{k}_{n_1},\hat{k}_{n_1})} I_{\tilde{\theta}_{n_1}}(Q\hat{T}_{{n_1}}\P)
\end{equation}
(cf. Lemma~\ref{lemma:measurability} with $\theta=\tilde{\theta}_{n_1}$ and Lemma~\ref{lemma:finitedimlOpti}).

\item Generate $Z_{n_1+1},\dots, Z_n$ non-interactively from $X_{n_1+1},\dots, X_n$ using the mechanism $\hat{Q}_{n_1}\hat{T}_{{n_1}}$. More precisely, individual $i\in\{n_1+1,\dots, n\}$ computes $Y_i = \hat{T}_{{n_1}}(X_i)$ and generates $Z_i\in \{1,\dots, \hat{k}_{n_1}\}$ from $\hat{Q}_{n_1}(dz|Y_i)$ (equivalently, from the discrete distribution given by the vector $[\hat{Q}_{n_1}]_{\cdot,Y_i}\in\R^{\hat{k}_{n_1}}$).
\item Let $\hat{\theta}_{n_2}$ be an MLE in the model $[\hat{Q}_{n_1}\hat{T}_{{n_1}}\P]^{n_2}$, that is, a measurable maximizer (where it exists, and defined arbitrarily otherwise) of the conditional log-likelihood
\begin{align*}
\theta\mapsto \ell_{n_2}(\theta) &:= \ell_{n_2}(\theta; z_{n_1+1},\dots, z_n| \tilde{\theta}_{n_1}) := \frac{1}{n_2}\sum_{i=n_1+1}^n \log \hat{q}_\theta(z_i),\quad z_i\in\{1,\dots, \hat{k}_{n_1}\},
\end{align*}
over $\Theta$, where $n_2 = n-n_1$, 
$$
\hat{q}(z|x) := [\hat{Q}_{n_1}]_{z,\hat{T}_{{n_1}}(x)} = \sum_{j=1}^{\hat{k}_{n_1}}[\hat{Q}_{n_1}]_{z,j} \cdot \mathds 1_{\hat{T}_{{n_1}}^{-1}(\{j\})}(x)
$$ 
and
$$
\hat{q}_\theta(z) := \int_\X \hat{q}(z|x) p_\theta(x) \mu(dx) =  \sum_{j=1}^{\hat{k}_{n_1}}[\hat{Q}_{n_1}]_{z,j} \cdot P_\theta \left(\hat{T}_{{n_1}}^{-1}(\{j\})\right).
$$
\end{enumerate}

Notice that steps 1-5 above define the $\alpha$-sequentially interactive mechanism $Q^{(n)}$ extensively discussed in Subsection~\ref{sec:consistentMechanism}. This mechanism releases a locally private view $Z_1,\dots, Z_n$ of $X_1,\dots, X_n$. Specifically, conditional on $Z_1,\dots, Z_{n_1}$, the $Z_{n_1+1},\dots, Z_n$ are iid with counting density (pmf) $\hat{q}_\theta$ and $\hat{Q}_{n_1}\hat{T}_{{n_1}}\in\mathcal Q_\alpha(\X\to\{1,\dots, \hat{k}_{n_1}\})$ (cf. Figure~\ref{fig:Qn}). In order to fully control $\hat{q}_\theta$ also as a function in the first group data $z_{1:n_1}$, our main theorem operates under more traditional and stronger differentiability conditions than only DQM. The conclusion, however, is classical: If the MLE is consistent, then it is also efficient. 

\begin{theorem}\label{thm:MLEefficiency}
Suppose that Conditions~\ref{C.DQM}, \ref{C.measurable} and \ref{C.L2cont} are satisfied with $\Theta\subseteq\R$. Moreover, suppose that $p_\theta(x)$ is three times continuously differentiable with respect to $\theta$ for every $x\in\X$ and such that $\theta\mapsto \|\dot{p}_\theta\|_1$, $\theta\mapsto \|\ddot{p}_\theta\|_1$ and $\theta\mapsto \|\dddot{p}_\theta\|_1$ are continuous and finite and $\theta\mapsto \dddot{p}_\theta$ is continuous as a function from $\Theta\to L_1(\mu)$. 
If $I^*_{\alpha, \theta_0} := \sup_{Q\in\mathcal Q_\alpha(\X)}I_{\theta_0}(Q\P)>0$ for some $\theta_0\in\Theta$, $n_1 = n_1(n)\to\infty$, $n_1/n\to0$ and the two-step MLE $\hat{\theta}_{n_2}$ converges to $\theta_0$ in $Q^{(n)}P_{\theta_0}^n$-probability, then it is also regular and efficient at $\theta_0$, that is,
$$
\sqrt{n} \left( \hat{\theta}_{n_2} - [\theta_0 + h/\sqrt{n}] \right)  \stackrel{Q^{(n)}P_{\theta_0+h/\sqrt{n}}^n}{\rightsquigarrow} N\left(0,\left[I^*_{\alpha, \theta_0}\right]^{-1}\right), \quad \forall h\in\R,
$$
where $Q^{(n)}$ is the two-step $\alpha$-sequentially interactive mechanism generating $Z_1,\dots, Z_n$ as described above. 
\end{theorem}

\begin{proof}
For $h\in\R$ write $\theta_n := \theta_0+h/\sqrt{n}$, $P_n = Q^{(n)}P_{\theta_n}^n$ for the joint distribution of $Z_1,\dots, Z_n$ under $\theta_n$ and $\E_n$ for the corresponding expectation operator. We begin with deterministic considerations, keeping the sample $Z_1,\dots, Z_n$ fixed. First, notice that $\hat{Q}_{n_1}$ may contain zero entries, but by the $\alpha$-differential privacy and positivity constraints we have $0\le (\hat{Q}_{n_1})_{z,j} \le e^\alpha (\hat{Q}_{n_1})_{z,j'}$, for all $z,j,j'$, and hence all entries in the same row of a zero entry must also be equal to zero. Thus, for $i\ge n_1+1$ and without loss of generality, we can restrict to $z_i\in\hat{\mathcal Z} := \{z\in[\hat{k}_{n_1}]|\exists j : (\hat{Q}_{n_1})_{z,j}>0\} = \{z\in[\hat{k}_{n_1}]|\forall j : (\hat{Q}_{n_1})_{z,j}>0\}$, because, irrespective of $X_i$, data $Z_i$ with values outside of $\hat{\mathcal Z}$ will never be generated. In particular, for $z\in\hat{\mathcal Z}$, also $\hat{q}(z|x)>0$ and $\hat{q}_\theta(z)>0$.

Note that in view of Lemma~\ref{lemma:Qdiff}, $\theta\mapsto \hat{q}_\theta(z)$ is three times continuously differentiable for all $z\in\hat{\mathcal Z}$ with derivatives $\dot{\hat{q}}_\theta(z) = \int_\X \hat{q}(z|x) \dot{p}_\theta(x)\mu(dx)$, $\ddot{\hat{q}}_\theta(z) = \int_\X \hat{q}(z|x) \ddot{p}_\theta(x)\mu(dx)$ and $\dddot{\hat{q}}_\theta(z) = \int_\X \hat{q}(z|x) \dddot{p}_\theta(x)\mu(dx)$. In particular, a fraction with any one of these three functions in the numerator and $\hat{q}_\theta(z)$ in the denominator, is bounded in absolute value by $e^{2\alpha}$ times the $L_1(\mu)$-norm of the corresponding partial derivative of $p_\theta$. To see this, simply divide numerator and denominator by $\hat{q}(z|x^*)$, for some arbitrary $x^*\in\X$, and note that $e^{-\alpha}\le \frac{\hat{q}(z|x)}{\hat{q}(z|x^*)} \le e^\alpha$, by the differential privacy constraint ($\hat{Q}_{n_1}\in\mathcal M_\alpha$). Next, expand the conditional log-likelihood $\ell_{n_2}$ as follows:
\begin{align*}
0 = \dot{\ell}_{n_2}(\hat{\theta}_{n_2}) = \dot{\ell}_{n_2}(\theta_n) + \ddot{\ell}_{n_2}(\theta_n)(\hat{\theta}_{n_2}-\theta_n) + \frac{1}{2} \dddot{\ell}_{n_2}(\bar{\theta}_n)(\hat{\theta}_{n_2}-\theta_n)^2,
\end{align*}
for an intermediate value $\bar{\theta}_n$ between $\theta_n$ and $\hat{\theta}_{n_2}$.\footnote{Notice that the intermediate value $\bar{\theta}_n$ itself is not guaranteed to be measurable, but by construction clearly the remainder term is.}
Here
\begin{align*}
\dddot{\ell}_{n_2}(\theta) = \frac{1}{n_2} \sum_{i=n_1+1}^n \frac{\partial^3}{\partial \theta^3} \log \hat{q}_\theta(z_i) = 
\frac{1}{n_2} \sum_{i=n_1+1}^n \left[ \frac{\dddot{\hat{q}}_{\theta}}{\hat{q}_{\theta}} - \frac{3\dot{\hat{q}}_{\theta}\ddot{\hat{q}}_{\theta}}{\hat{q}^2_{\theta}} + \frac{2\dot{\hat{q}}^3_{\theta}}{\hat{q}^3_{\theta}}\right](z_i),
\end{align*}
which is bounded in absolute value by the continuous function $\psi(\theta):=e^{2\alpha}\|\dddot{p}_{\theta}\|_{1} + 3 e^{4\alpha}\|\dot{p}_{\theta}\|_{1}\|\ddot{p}_{\theta}\|_{1} + 2e^{6\alpha}\|\dot{p}_{\theta}\|^3_{1}$. The remainder term above can thus be written as $R_n\cdot (\hat{\theta}_{n_2}-\theta_n)$, for a measurable $R_n$ with $|R_n| \le \frac12 \sup_{t\in[\hat{\theta}_{n_2},\theta_n]}\psi(t)(\hat{\theta}_{n_2}-\theta_n)$. By contiguity (cf. Lemma~\ref{lemma:contiguity}), $\hat{\theta}_{n_2}$ also converges to $\theta_0$ (and hence also to $\theta_n$) under the sequence of local alternatives $P_n$ and thus $R_n = o_{P_n}(1)$. Consequently, we have
\begin{align*}
- \dot{\ell}_{n_2}(\theta_n) &= \left[\ddot{\ell}_{n_2}(\theta_n) + R_n\right](\hat{\theta}_{n_2}-\theta_n) 
= \left[\ddot{\ell}_{n_2}(\theta_n) + o_{P_n}(1)\right](\hat{\theta}_{n_2}-\theta_n).
\end{align*}
Moreover,
\begin{align*}
\ddot{\ell}_{n_2}(\theta_n) = \frac{1}{n_2} \sum_{i=n_1+1}^n \frac{\partial^2}{\partial \theta^2} \log \hat{q}_\theta(Z_i) \bigg|_{\theta=\theta_n} = 
\frac{1}{n_2} \sum_{i=n_1+1}^n \left[ \frac{\ddot{\hat{q}}_{\theta_n}}{\hat{q}_{\theta_n}} - \frac{\dot{\hat{q}}^2_{\theta_n}}{\hat{q}^2_{\theta_n}}\right](Z_i).
\end{align*}
Conditional on $Z_1,\dots, Z_{n_1}$, this is an average of iid terms, each of which is bounded in absolute value by $e^{2\alpha}\|\ddot{p}_{\theta_n}\|_1 + e^{4\alpha}\|\dot{p}_{\theta_n}\|^2_1$ and has conditional expectation
$$
\sum_{z\in\hat{\mathcal Z}} \left[\ddot{\hat{q}}_{\theta_n}(z) - \frac{\dot{\hat{q}}^2_{\theta_n}(z)}{\hat{q}^2_{\theta_n}(z)}\hat{q}_{\theta_n}(z)\right]
=
\left[\frac{\partial^2}{\partial \theta^2}\sum_{z\in\hat{\mathcal Z}} \hat{q}_{\theta}(z)\right]_{\theta=\theta_n} - I_{\theta_n}(\hat{Q}_{n_1}\hat{T}_{{n_1}}\P) = - I_{\theta_n}(\hat{Q}_{n_1}\hat{T}_{{n_1}}\P),
$$
where we used that $\frac{\partial^2}{\partial \theta^2}\sum_{z\in\hat{\mathcal Z}} \hat{q}_{\theta}(z) = \frac{\partial^2}{\partial \theta^2} 1 = 0$.
The expression on the right hand side of the previous display converges to $-I^*_{\alpha,\theta_0}\in[-I_{\theta_0}(\P),0)$ by Theorem~\ref{thm:Qhat} in $P_n$-probability, using contiguity again.
Thus, by the conditional Markov inequality, $\ddot{\ell}_{n_2}(\theta_n) \to - I^*_{\alpha,\theta_0}$ in $P_n$-probability. We therefore arrive at
$$
\dot{\ell}_{n_2}(\theta_n) = \left[I^*_{\alpha,\theta_0} + o_{P_n}(1)\right](\hat{\theta}_{n_2}-\theta_n).
$$
Since $I^*_{\alpha,\theta_0}>0$, we have, at least on an event of probability converging to $1$,
$$
\sqrt{n}(\hat{\theta}_{n_2}-\theta_n) = \frac{\sqrt{n}}{\sqrt{n_2}}\left[I^*_{\alpha,\theta_0} + o_{P_n}(1)\right]^{-1} \frac{1}{\sqrt{n_2}} \sum_{i=n_1+1}^n \frac{\dot{\hat{q}}_{\theta_n}(Z_i)}{\hat{q}_{\theta_n}(Z_i)}.
$$
We conclude the proof by studying the cdf 
\begin{align*}
P_n\left( \frac{1}{\sqrt{n_2}} \sum_{i=n_1+1}^n \frac{\dot{\hat{q}}_{\theta_n}(Z_i)}{\hat{q}_{\theta_n}(Z_i)} \le t\right)
=
\E_n\left[ P_n\left( \frac{1}{\sqrt{n_2}} \sum_{i=n_1+1}^n \frac{\dot{\hat{q}}_{\theta_n}(Z_i)}{\hat{q}_{\theta_n}(Z_i)} \le t\Big| Z_1,\dots, Z_{n_1}\right)\right].
\end{align*}
Conditionally, we are dealing with a scaled sample mean of iid random variables with mean zero, variance equal to $I_{\theta_n}(\hat{Q}_{n_1}\hat{T}_{{n_1}}\P)$ and absolute third moment bounded by $e^{6\alpha}\|\dot{p}_{\theta_n}\|_1^3$. Thus, by the Berry-Esseen bound, we have for all $t\in\R$,
$$
\left| P_n\left( \frac{1}{\sqrt{n_2}} \sum_{i=n_1+1}^n \frac{\dot{\hat{q}}_{\theta_n}(Z_i)}{\hat{q}_{\theta_n}(Z_i)} \le t\cdot \sqrt{I_{\theta_n}(\hat{Q}_{n_1}\hat{T}_{{n_1}}\P)}\Big| Z_1,\dots, Z_{n_1}\right) - \Phi(t)\right| \le \frac{C(\alpha,\theta_n)}{\sqrt{n_2}},
$$
at least on the event $\{I_{\theta_n}(\hat{Q}_{n_1}\hat{T}_{{n_1}}\P)>0\}$ of probability converging to $1$ and where $\Phi$ is the cdf of the standard normal distribution. Since the (deterministic) sequence $C(\alpha,\theta_n)$ is bounded by a constant multiple of $e^{6\alpha}\|\dot{p}_{\theta_n}\|_1^3$, this finishes the proof.
\end{proof}
%==================================================================

\section{Practical aspects}
\label{sec:appl}

The two-step $\alpha$-private estimation procedure introduced in Section~\ref{sec:MLE} has two practical limitations. First, the construction of a consistent quantizer (cf. Definition~\ref{def:quantizer} as well as Lemma~\ref{lemma:ConsQuantExists} and its proof) is quite involved and computationally challenging. This can be much simplified in more specific models. Second, the optimization problem in \eqref{eq:max} of the estimation procedure outlined in Subsection~\ref{sec:MLEefficiency} can be solved by the linear programming approach of \citet{Kairouz16}, but the runtime of that program is exponential in the resolution $\hat{k}_{n_1}$ of the quantizer $\hat{T}_{n_1}$. In the examples we consider below, however, we find that even for small values of $\hat{k}_{n_1}$, for which the linear program can still be solved on standard hardware, the solution approximates the optimal mechanism $Q^*\in\argmax_{Q\in\mathcal Q_\alpha(\X)} I_\theta(Q\P)$ sufficiently well.

\subsection{Location families}

In this subsection we consider models $\P_{loc}$ described by Lebesgue densities on $\X=\R$ of the form 
$$
p_\theta(x) = p(x-\theta), \quad\theta\in\Theta=\R,
$$ 
for some sufficiently regular density $p$ as follows.

\begin{enumerate}
        \setlength\leftmargin{-20pt}
\renewcommand{\theenumi}{(A\arabic{enumi})}
\renewcommand{\labelenumi}{\textbf{\theenumi}}

\item \label{A.diffable} The density $p:\X\to(0,\infty)$ is three times continuously differentiable with $p^{(j)}\in L_1$, $j=1,2,3$, and $\frac{p'}{\sqrt{p}} \in L_2(\lambda)$.
\end{enumerate}

The proof of the following result is deferred to the supplement.
 
\begin{lemma}\label{lemma:location}
If Assumption~\ref{A.diffable} is satisfied then the assumptions of Theorem~\ref{thm:MLEefficiency} hold for the location model $\P_{loc}$. Moreover, let $(k_m)_{m\in\N}$ be a sequence of positive integers. If for each $m\in\N$, $T_m:\R\to[k_m]$ is a measurable function such that for
$B_j:= B_{j,m}:= T_{m}^{-1}(\{j\})$, $J:=J_m\subseteq \{l\in[k_m]: 0<\lambda(B_{l,m})<\infty\}$, $K:=K_{J,m} := \bigcup_{j\in J}B_{j,m}$,
$$
\bar{p}(x) := \sum_{j\in J} \frac{\int_{B_j}p(y)\lambda(dy)}{\lambda(B_j)} \mathds 1_{B_j}(x), 
\; \text{ and }\;
\bar{p'}(x) := \sum_{j\in J} \frac{\int_{B_j}p'(y)\lambda(dy)}{\lambda(B_j)} \mathds 1_{B_j}(x) \in\R^p,
$$
we have
\begin{align}
\left[\Big\|(p - \bar{p})\mathds 1_{K}\Big\|_{L_1(\lambda)} + \Big\|(p' - \bar{p'})\mathds 1_{K}\Big\|_{L_1(\lambda)} + \sqrt{\|p\mathds 1_{K^c}\|_{L_1(\lambda)}} \right]\rightarrow0
\end{align} 
as $m\to\infty$, then $T_{m,\theta}(x) := T_m(x,\theta) := T_m(x-\theta)$ is a consistent quantizer for $\P_{loc}$ with associated sequence $(k_m)_{m\in\N}$ and
$$
I_\theta(QT_{m,\theta}\P_{loc}) = I_0(QT_{m,0}\P_{loc}), \quad \forall Q\in\mathcal M_\alpha(k_m,k_m), \forall \theta\in\R.
$$
\end{lemma}

Lemma~\ref{lemma:location} shows that in the location model under Assumption~\ref{A.diffable} the two-step private MLE of Subsection~\ref{sec:MLEefficiency} (if consistent) is asymptotically efficient and can be implemented with a consistent quantizer whose resolution level $k_m$ does not depend on the preliminary estimate $\tilde{\theta}_{n_1}$. All we need to do is approximate $p$ and its derivative $p'$ by histograms and then shift the resulting partition by $\tilde{\theta}_{n_1}$. Furthermore, the lemma also shows that in the location model the approximately optimal privacy mechanism $\hat{Q}_{n_1}$ in \eqref{eq:max} can be computed independently of the outcome of the preliminary estimation $\tilde{\theta}_{n_1}$. All we need to do is find
\begin{equation}\label{eq:maxLocation}
Q_{n_1}\in\argmax_{Q\in\mathcal M_\alpha(k_{n_1},k_{n_1})} I_{0}(QT_{n_1,0}\P_{loc}).
\end{equation}
In the second step we then non-interactively generate sanitized data $Z_i$ from $Q_{n_1}(dz|Y_i)$, where $Y_i = \hat{T}_{n_1}(X_i) = T_{n_1}(X_i-\tilde{\theta}_{n_1})$, that is, every individual shifts its own data $X_i$ by the preliminary estimate $\tilde{\theta}_{n_1}$, quantizes that value and applies $Q_{n_1}$ to produce a sanitized output. To give an explicit example of a consistent quantizer and of a numerical solution to \eqref{eq:maxLocation}, we further specialize to the Gaussian location model.

\subsection{The Gaussian location model}
\label{sec:GaussLoc}
Consider $p(x) = \frac{1}{\sqrt{2\pi}}e^{-\frac12 x^2}>0$ which has continuous derivatives $p'(x) = -xp(x)$, $p''(x) = -p(x) + x^2p(x)$ and $p'''(x) = 3xp(x) - x^3p(x)$. Clearly, $p',p'',p'''\in L_1$ and $p'(x)/\sqrt{p(x)} = -x\sqrt{p(x)}$, the square of which integrates to one. Hence, Assumption~\ref{A.diffable} is satisfied. The next lemma provides an explicit construction of a consistent quantizer.

\begin{lemma}\label{lemma:gauss}
Let $(k_m)_{m\in\N}$ be a sequence of even integers with $k_m\to\infty$ as $m\to \infty$ and define $B_{j} := (\Phi^{-1}(\frac{j-1}{k_m}), \Phi^{-1}(\frac{j}{k_m})]$, $j=1,\dots, k_m-1$, and $B_{k_m} := (\Phi^{-1}(1-\frac{1}{k_m}), \infty)$. Then $T_m(x) := \sum_{j=1}^{k_m} j \mathds 1_{B_j}(x)$ satisfies the condition in Lemma~\ref{lemma:location}, that is, $T_m(x,\theta) := T_m(x-\theta)$, with the sequence $(k_m)_{m\in\N}$, is a consistent quantizer for the Gaussian location model.
\end{lemma}

We can now tackle the optimization problem \eqref{eq:maxLocation} with this specific quantizer. Using Lemma~\ref{lemma:finitedimlOpti} and the abbreviations $k=k_{n_1}$, $r := [P_0T_{n_1,0}^{-1}(\{j\})]_{j=1}^{k} = [P_0(B_j)]_{j=1}^{k} = \frac{1}{k} (1,\dots, 1)^T\in\R^{k}$ and $\dot{r} := [\int_{B_j} -p'(x) dx]_{j=1}^{k}= [p\circ\Phi^{-1}(\frac{j-1}{k}) - p\circ\Phi^{-1}(\frac{j}{k})]_{j=1}^{k}\in\R^k$, we can write the objective function explicitly as
\begin{equation}\label{eq:maxGauss}
Q \mapsto I_0(QT_{n_1}\P) = \sum_{i=1}^{k} g(Q_{i\cdot}^T) = \sum_{i=1}^{k} \frac{(\sum_{j=1}^{k} Q_{ij}\dot{r}_j)^2}{\frac{1}{k}\sum_{j=1}^{k} Q_{ij}},
\end{equation}
where $g(0):=0$. Theorem~4 of \citet{Kairouz16} allows us to rewrite the optimization of this objective as a linear program:
\begin{equation}\label{eq:LP}
\begin{aligned}
	\max_{\gamma\in\R^{2^{k}}} \quad&\sum_{j=1}^{2^k}g\left( S_j^{(k)}\right)\gamma_j\\
	\text{subject to } \quad&S^{(k)}\gamma = (1,\dots, 1)^T\quad\text{and}\quad \gamma\in\R_+^{2^k},
\end{aligned}
\end{equation}
where $S^{(k)}\in\R^{k\times 2^k}$ is the stair-case matrix whose $j$-th column is defined as $S_j^{(k)} := (e^\alpha-1)b_{j-1} + (1,\dots, 1)^T\in\R^k$, where $b_j\in\R^k$ is the binary vector corresponding to the binary representation of $j$ for $j\le 2^k-1$. If we set $\hat{Q} = [S^{(k)}\text{diag}(\hat{\gamma}_1,\dots, \hat{\gamma}_{2^k})]^T$, where $\hat{\gamma}\in\R^{2^k}$ is a solution to \eqref{eq:LP}, then $\hat{Q}$ will have at most $k$ non-zero rows. Notice that by the differential privacy constraint, a zero entry in a matrix $Q\in\mathcal M_\alpha(\ell,k)$ implies that all entries in that row must be zero and removing a zero-row from $Q$ does not change the objective function \eqref{eq:maxGauss}. Thus, removing $2^k-k$ zero rows, we obtain a $\hat{Q}$ that maximizes \eqref{eq:maxGauss}. 

In Figure~\ref{fig:sim} we plot optimal objective function values of the linear program \eqref{eq:LP} (or equivalently of \eqref{eq:maxGauss}) for different values of the quantizers resolution $k$. Combining Lemma~\ref{lemma:gauss} and Theorem~\ref{thm:Qhat} with $\tilde{\theta}_{n_1} = \theta_{n_1} = \theta = 0$ and $\hat{k}_{n_1}=k$, we see that the optimal objective function values converge to the global optimum $\sup_{Q\in\mathcal Q_\alpha(\X)} I_0(Q\P)$ as $k\to\infty$, which is a supremum over an infinite dimensional space of Markov kernels. Looking at the plots in Figure~\ref{fig:sim} the approximation to this global optimum seems to be reasonably accurate already for quantizers with resolution $k=18$, which is still computationally feasible (runtime of a few minutes on an ordinary personal computer). In fact, for small values of the privacy parameter $\alpha$ (that is, when the privacy protection is strong) we even observe that the optimal objective function values are all equal whenever $k$ is even. This suggests that the optimum is already reached at $k=2$. Another curious phenomenon of the high privacy regime is exhibited in Figure~\ref{fig:outsize}. For large values of $\alpha$ the maximizer $\hat{Q}$ of \eqref{eq:maxGauss} over $\mathcal M_\alpha(k,k)$ has no zero rows and thus the sanitized data $Z_i$ may take any value in $\{1,\dots,k\}$ with positive probability. However, for smaller values of $\alpha$, we see that actually only a few rows of the optimizer $\hat{Q}$ are non-zero. In other words, sanitized data $Z_i$ will only be generated from a small set $\hat{\mathcal Z}$ of possible outcomes. In the extreme cases of $\alpha=2$ and $\alpha=1$ we see that for $k$ even, an optimal mechanism generates only binary outputs $Z_i$. Together with our observations from Figure~\ref{fig:sim}, this suggests that at least for small $\alpha$ the global optimum $\sup_{Q\in\mathcal Q_\alpha(\X)} I_0(Q\P)$ may be reached by a very simple binary mechanism that first binarizes the original Gaussian data $X_i$ (i.e., the quantizer is given by $T_m(x-\theta) = \sign(x-\theta)$) and then applies a simple randomized response mechanism as in Subsection~\ref{sec:Bernoulli}. Very recently \citet{Kalinin24} have proved this conjecture via a delicate duality argument using the linear program \eqref{eq:LP} of \citet{Kairouz16}. We refer to \citet{Kalinin24} for more details on efficient Gaussian mean estimation under local differential privacy.

\begin{figure}
\includegraphics[width=0.9\textwidth]{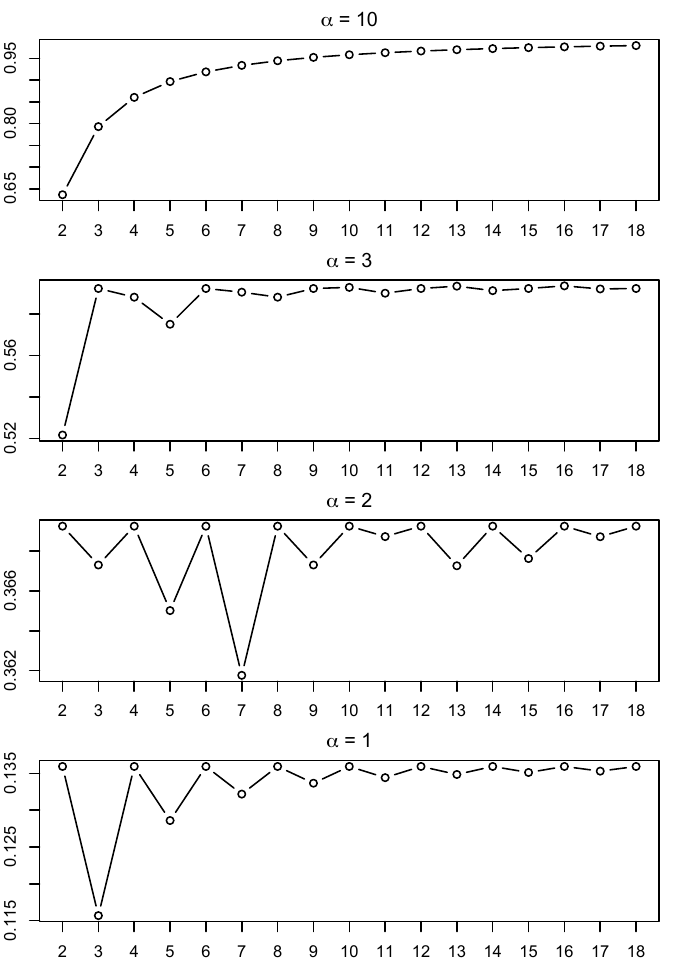} 
\caption{Optimal values of \eqref{eq:maxGauss} for different resolution levels $k$ (horizontal axis) and privacy parameters $\alpha$. For $k\to\infty$ theory predicts that these values converge to the global optimum $\sup_{Q\in\mathcal Q_\alpha(\X)} I_0(Q\P)$.}
\label{fig:sim}
\end{figure}

\begin{figure}
\includegraphics[width=0.9\textwidth]{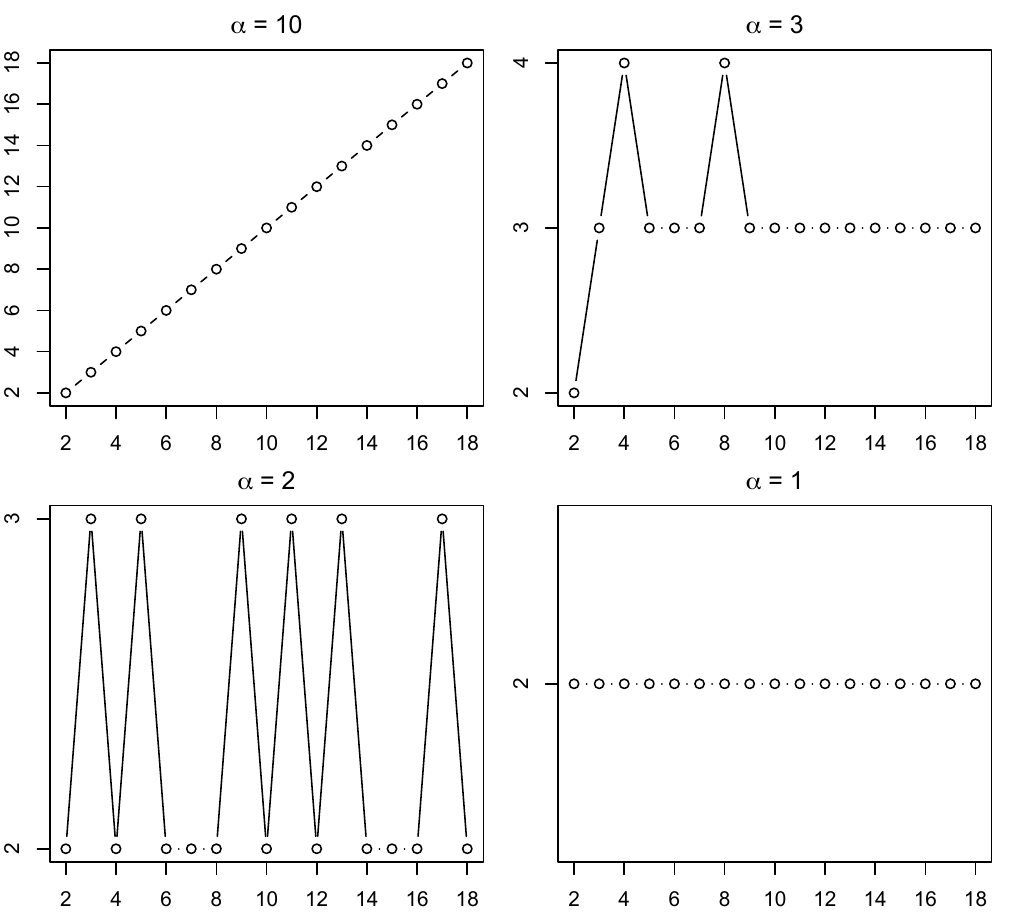} 
\caption{Number of non-zero rows (vertical axis), i.e, size of the output alphabet $\hat{\mathcal Z} \subseteq [k]$, of the optimal privacy mechanism $\hat{Q}\in\mathcal Q_\alpha([k]\to\hat{\mathcal Z})$ that maximizes \eqref{eq:maxGauss}, for different resolution levels $k$ (horizontal axis) and privacy parameters $\alpha$.}
\label{fig:outsize}
\end{figure}

\subsection{Scale families}

In this subsection we consider models $\P_{scale}$ described by Lebesgue densities on $\X=\R$ or $\X=[0,\infty)$ of the form 
$$
p_\theta(x) = \theta^{-1/2}p\left(\theta^{-1/2}x\right), \quad\theta\in\Theta:=(0,\infty),
$$ 
parametrized by the variance $\theta$, for a density $p$ that satisfies the following assumption:
\begin{enumerate}
        \setlength\leftmargin{-20pt}
\renewcommand{\theenumi}{(A\arabic{enumi})}
\renewcommand{\labelenumi}{\textbf{\theenumi}}
\stepcounter{enumi}
\item \label{A.diffableScale} The density $p:\X\to(0,\infty)$ is three times continuously differentiable such that the functions $x\mapsto x^jp^{(j)}(x)$, for $j=1,2,3$, are in $L_1$ and $x\mapsto x\frac{p'(x)}{\sqrt{p(x)}}$ is in $L_2$.
\end{enumerate}

The proof of the following result is deferred to the supplement.
 
\begin{lemma}\label{lemma:scale}
If Assumption~\ref{A.diffableScale} is satisfied then the assumptions of Theorem~\ref{thm:MLEefficiency} hold for the scale model $\P_{scale}$. Moreover, let $(k_m)_{m\in\N}$ be a sequence of positive integers. If for each $m\in\N$, $T_m:\R\to[k_m]$ is a measurable function such that for
$B_j:= B_{j,m}:= T_{m}^{-1}(\{j\})$, $J:=J_m:= \{l\in[k_m]: 0<\lambda(B_{l,m})<\infty\}$, $K:=K_{J,m} := \bigcup_{j\in J}B_{j,m}$,
$$
\bar{p}(x) := \sum_{j\in J} \frac{\int_{B_j}p(y)\lambda(dy)}{\lambda(B_j)} \mathds 1_{B_j}(x), 
\; \text{ and }\;
\bar{p'}(x) := \sum_{j\in J} \frac{\int_{B_j}p'(y)\lambda(dy)}{\lambda(B_j)} \mathds 1_{B_j}(x) \in\R^p,
$$
we have
\begin{align}
\Big\|(p - \bar{p})\mathds 1_{K}\Big\|_{L_1(\lambda)} + \Big\|(p' - \bar{p'})\mathds 1_{K}\Big\|_{L_1(\lambda)} + \sqrt{\|p\mathds 1_{K^c}\|_{L_1(\lambda)}} \rightarrow0
\end{align} 
as $m\to\infty$, then $T_{m,\theta}(x) := T_m(x,\theta) := T_m(\theta^{-1/2}x)$ is a consistent quantizer for $\P_{scale}$ with respect to $(k_m)_{m\in\N}$ and
$$
I_\theta(QT_{m,\theta}\P_{scale}) = \frac{1}{\theta^2}I_1(QT_{m,1}\P_{scale}), \quad \forall Q\in\mathcal M_\alpha(k_m,k_m), \forall \theta\in\Theta=(0,\infty).
$$
\end{lemma}

Lemma~\ref{lemma:scale} shows that in the scale model under Assumption~\ref{A.diffableScale} the two-step private MLE of Subsection~\ref{sec:MLEefficiency} is asymptotically efficient and can be implemented with a consistent quantizer whose resolution level $k_m$ does not depend on the preliminary estimate $\tilde{\theta}_{n_1}$. All we need to do is approximate $p$ and its derivative $p'$ by histograms and then scale the resulting partition by $\tilde{\theta}_{n_1}$ (cf. the quantizer $T_{m,\theta}(x) = T_m(\theta^{-1/2}x)$). Furthermore, the lemma also shows that in the scale model the approximately optimal privacy mechanism $\hat{Q}_{n_1}$ in \eqref{eq:max} can be obtained independently of the outcome of the preliminary estimation $\tilde{\theta}_{n_1}$, because $\theta$ enters the Fisher-Information only as a constant scaling. Thus, all we need to do is find
\begin{equation}\label{eq:maxScale}
Q_{n_1}\in\argmax_{Q\in\mathcal M_\alpha(k_{n_1},k_{n_1})} I_{1}(QT_{n_1,1}\P_{scale}).
\end{equation}
In the second step we then non-interactively generate sanitized data $Z_i$ from $Q_{n_1}(dz|Y_i)$, where $Y_i = \hat{T}_{n_1}(X_i) = T_{n_1}(\tilde{\theta}_{n_1}^{-1/2}X_i)$, that is, every individual scales its own data $X_i$ by the preliminary estimate $\tilde{\theta}_{n_1}$, quantizes that value and applies $Q_{n_1}$ to produce a sanitized output. To give an explicit example of a consistent quantizer and of a numerical solution to \eqref{eq:maxScale}, we further specialize to the Gaussian scale model.

\subsection{The Gaussian scale model}
\label{sec:GaussScale}
As in Subsection~\ref{sec:GaussLoc}, we consider $p(x) = \frac{1}{\sqrt{2\pi}}e^{-\frac12 x^2}>0$, which is easily seen to satisfy Assumption~\ref{A.diffableScale}. Since the conditions for a consistent quantizer in Lemma~\ref{lemma:location} and in Lemma~\ref{lemma:scale} are the same, Lemma~\ref{lemma:gauss} also provides a consistent quantizer for the scale model. Thus, we can solve the optimization problem \eqref{eq:maxScale} with this same quantizer. Using Lemma~\ref{lemma:finitedimlOpti} and the abbreviations $k=k_{n_1}$, $x_j := \Phi^{-1}(\frac{j}{k})$, $r := [P_1T_{n_1,1}^{-1}(\{j\})]_{j=1}^{k} = [P_1(B_j)]_{j=1}^{k} = \frac{1}{k} (1,\dots, 1)^T\in\R^{k}$ and $\dot{r} := -[\int_{B_j} p(x)-xp'(x) dx]_{j=1}^{k}= [x_{j-1}p(x_{j-1}) - x_jp(x_j)]_{j=1}^{k}\in\R^k$, we can write the objective function explicitly as
\begin{equation*}
Q \mapsto I_1(QT_{n_1,1}\P) = \sum_{i=1}^{k} g(Q_{i\cdot}^T) = \sum_{i=1}^{k} \frac{(\sum_{j=1}^{k} Q_{ij}\dot{r}_j)^2}{\frac{1}{k}\sum_{j=1}^{k} Q_{ij}},
\end{equation*}
where $g(0):=0$. Notice that this is the same as \eqref{eq:maxGauss} for the Gaussian location model, except that $\dot{r}$ has a slightly different form. As before, we implemented the LP of \citet{Kairouz16}. The results are shown in Figures~\ref{fig:simScale} and \ref{fig:outsizeScale}. Again, it seems that we are pretty close to convergence already for rather small values of $k$. Thus, computing time does not seem to be an issue in spite of the exponential runtime of the LP. Moreover, similar to the location model, we again observe that the optimal privacy mechanism $\hat{Q}_{n_1}$ releases binary sanitized data $Z_i$ in the high privacy regime (cf. Figure~\ref{fig:outsizeScale}). Finally, notice the pathology in the plots in the case $k=2$. In that case the quantizer corresponds to the partition $B_1=(-\infty,0]$, $B_2=(0,\infty)$. Thus, when computing $I_1(QT_{n_1,1}\P_{scale})$, we only record the sign of a centered Gaussian random variable, which does not contain any information about its variance. Thus, $I_1(QT_{n_1,1}\P_{scale})=0$, irrespective of the mechanism $Q$.

\begin{figure}
\includegraphics[width=0.9\textwidth]{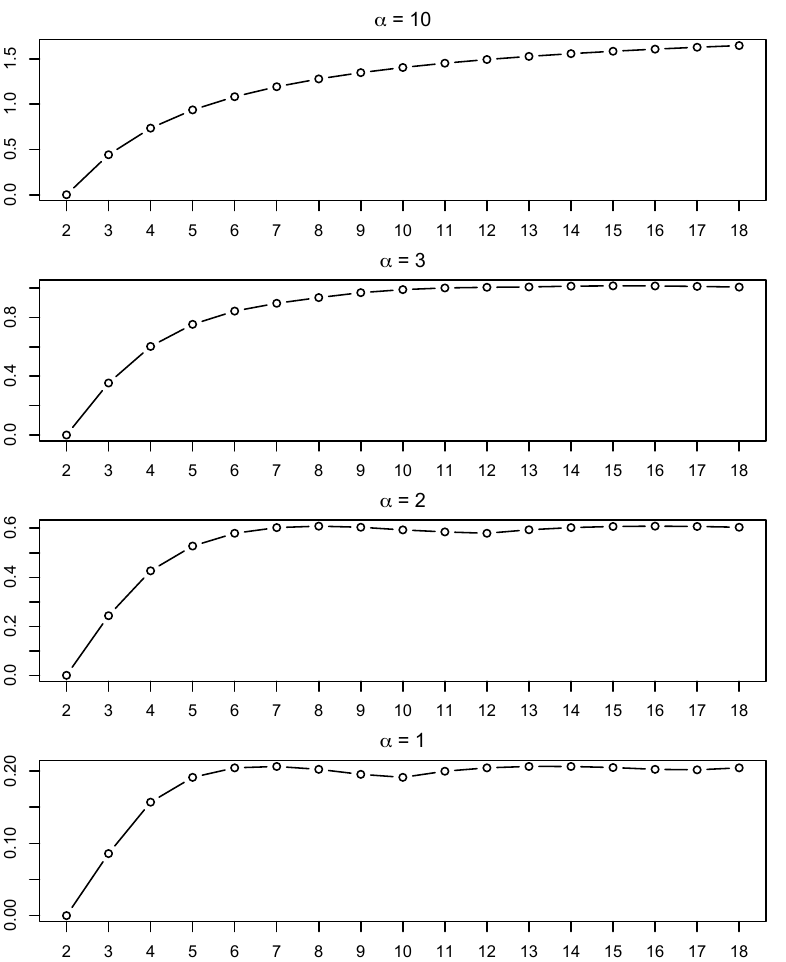} 
\caption{Optimal values of private Fisher-Information $Q\mapsto I_1(QT_{n_1,1}\P_{scale})$ in the Gaussian scale model for different resolution levels $k$ (horizontal axis) and privacy parameters $\alpha$. Theory predicts that these values converge to the global optimum $\sup_{Q\in\mathcal Q_\alpha(\X)} I_1(Q\P_{scale})$, for $k\to\infty$.}
\label{fig:simScale}
\end{figure}

\begin{figure}
\includegraphics[width=0.9\textwidth]{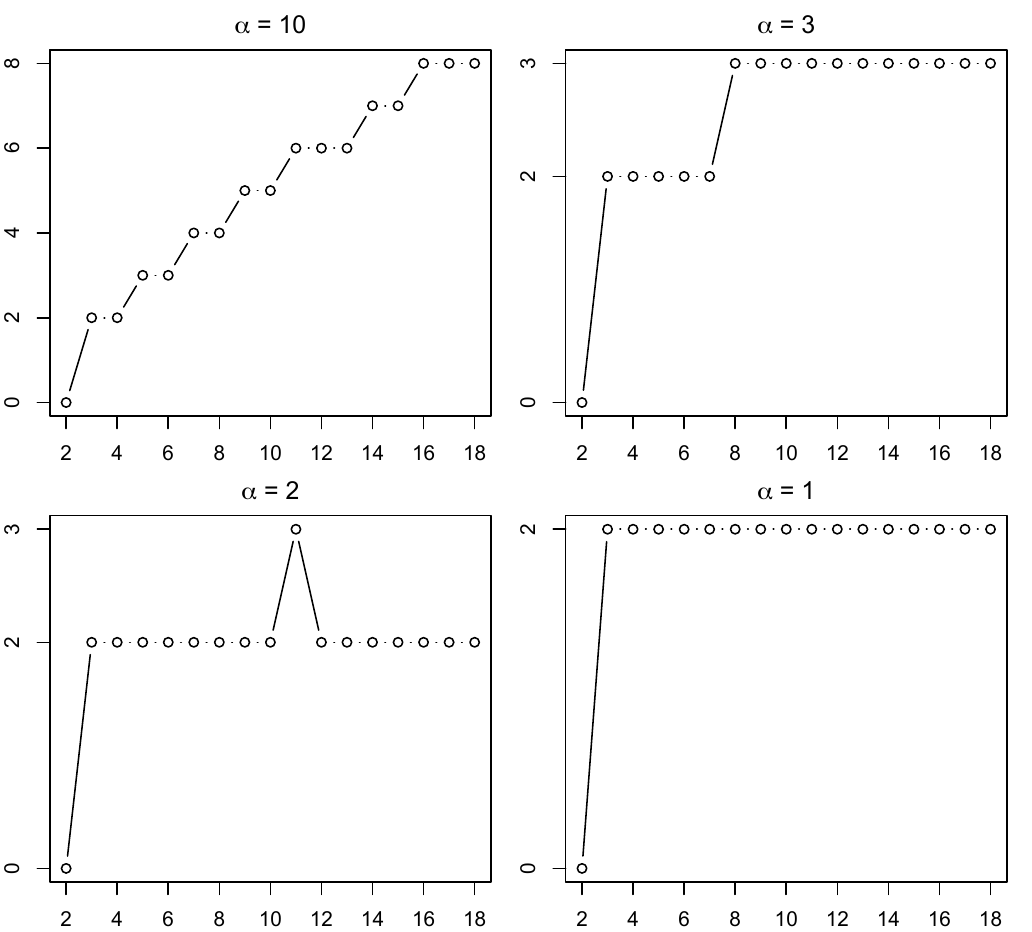} 
\caption{Number of non-zero rows, i.e, size of the output alphabet $\hat{\mathcal Z} \subseteq [k]$, of the optimal privacy mechanism $\hat{Q}\in\mathcal Q_\alpha([k]\to\hat{\mathcal Z})$ in the Gaussian scale model, for different resolution levels $k$ (horizontal axis) and privacy parameters $\alpha$.}
\label{fig:outsizeScale}
\end{figure}

%==================================================================

\begin{funding}
This research was supported by the Austrian Science Fund (FWF): I~5484-N, as part of the Research Unit 5381 of the German Research Foundation. 
\end{funding}

%==================================================================

\appendix

\section{Technical lemmas and proofs of Section~\ref{sec:LB}}
\label{sec:app:LB}

\begin{lemma}\label{lemma:DQMbound}
Fix $\alpha>0$ and consider a measurable space $(\X,\mathcal F)$ and probability measures $P_0,P_1$ on $(\X,\mathcal F)$ that are both dominated by the $\sigma$-finite measure $\mu$ with densities $p_j(x)=\frac{dP_j}{d\mu}(x)$, $j=0,1$. Let $q:\X\to[0,\infty)$ and $s:\X\to\R^p$ be measurable functions such that $e^{-\alpha}\le q\le e^\alpha$, $\mu$-almost surely, and $c:=\E_{P_0}[\|s\|^2]<\infty$, and set $q_j:= \int_\X q(x)p_j(x)\mu(dx)$, $j=0,1$. Then $q_j\in[e^{-\alpha}, e^\alpha]$, $j=0,1$, and
\begin{equation}\label{eq:DQMbound:score}
t := \int_\X s(x)\frac{q(x)p_0(x)}{q_0} \mu(dx)
\end{equation}
is well defined and bounded by $\|t\|\le e^{2\alpha}\sqrt{\E_{P_0}[\|s\|^2]}<\infty$. Moreover, for all $h\in\R^p$ with $\|h\|\sqrt{c} < e^{-2\alpha}$, we have
\begin{align*}
\left( \sqrt{q_1} - \sqrt{q_0} - \frac{1}{2} h^T t\sqrt{q_0}\right)^2
&\le
2e^{3\alpha}\left[ 2\Psi(h) + \frac{1}{2}\|h\|^2 c \left(\Psi(h)
+ \frac14 \|h\|^2c\right)\right]\\
&\quad\quad+ \frac{1}{32}(e^{-\alpha} - e^\alpha \|h\|\sqrt{c})^{-3}\|h\|^4e^{4\alpha}c^2,
\end{align*}
where
$$
\Psi(h) := \int_\X \left( \sqrt{p_{1}(x)} - \sqrt{p_0(x)} - \frac12 h^Ts(x)\sqrt{p_0(x)}\right)^2\mu(dx).
$$
\end{lemma}

\begin{proof}
Without loss of generality we can assume that $q:\X\to[e^{-\alpha}, e^\alpha]$ because changing $q$ on a $\mu$-null set does not affect $q_j$ and $t$. Notice that $t$ is well defined and finite, because $\int_\X \|s(x)\|^2 \frac{q(x)p_0(x)}{q_0}\mu(dx) \le e^{2\alpha} \E_{P_0}[\|s\|^2] <\infty$. For sufficiently small $\|h\|$ as in the statement of the lemma, we have $q_0 + h^Tt q_0\ge e^{-\alpha} - e^\alpha \|h\|\E_{P_0}[\|s\|]\ge e^{-\alpha} - e^\alpha \|h\|\sqrt{c} > 0$. Now
\begin{align*}
&\left|\sqrt{q_{1}} - \sqrt{q_0 + h^Tt q_0}\right| 
=
\left|\frac{q_{1} - q_0 - h^Tt q_0}{\sqrt{q_{1}}+\sqrt{q_0 + h^Tt q_0}} \right|\\
&\le
e^{\alpha/2} \int_\X q(x) \left| p_{1}(x) - p_0(x) - h^Ts(x)p_0(x)\right|\mu(dx) \\
&\le
e^{3\alpha/2}
\int_\X \left| p_{1} - p_0 - h^Tsp_0\right|d\mu\\
&\le e^{3\alpha/2}
\Bigg[\int_\X \left| p_{1} - p_0 - \frac12 h^Ts\sqrt{p_0}(\sqrt{p_1}+\sqrt{p_0})\right|d\mu\\
&\hspace{4cm}+ \int_\X \left| \frac12 h^Ts\sqrt{p_0}(\sqrt{p_1}-\sqrt{p_0})\right|d\mu
\Bigg].
\end{align*}
Furthermore, we have
\begin{align*}
&\int_\X \left| p_{1} - p_0 - \frac12 h^Ts\sqrt{p_0}(\sqrt{p_1}+\sqrt{p_0})\right|d\mu\\
&= \int_\X \left| \sqrt{p_1}+\sqrt{p_0}\right| \left| \sqrt{p_{1}} - \sqrt{p_0} - \frac12 h^Ts\sqrt{p_0}\right|d\mu\\
&\le \sqrt{\int_\X \left| \sqrt{p_1}+\sqrt{p_0}\right|^2d\mu} \sqrt{\int_\X \left| \sqrt{p_{1}} - \sqrt{p_0} - \frac12 h^Ts\sqrt{p_0}\right|^2d\mu}\\
&\le 2\sqrt{\Psi(h)},
\end{align*}
and
\begin{align*}
&\int_\X \left| \frac12 h^Ts\sqrt{p_0}(\sqrt{p_1}-\sqrt{p_0})\right|d\mu\\
&\le \frac12 \sqrt{\int_\X \left| h^Ts\sqrt{p_0}\right|^2d\mu} \sqrt{\int_\X \left|\sqrt{p_1}-\sqrt{p_0}\right|^2d\mu}\\
&\le \frac{1}{\sqrt{2}} \|h\| \sqrt{c} \sqrt{\int_\X \left|\sqrt{p_1}-\sqrt{p_0}-\frac12 h^Ts\sqrt{p_0}\right|^2d\mu
+ \int_X\left|\frac12 h^Ts\sqrt{p_0}\right|^2d\mu}\\
&\le
\frac{1}{\sqrt{2}} \|h\| \sqrt{c} \sqrt{\Psi(h)
+ \frac14 \|h\|^2c}.
\end{align*}
Hence, we have
\begin{align}
\left|\sqrt{q_{1}} - \sqrt{q_0 + h^Tt q_0}\right|
\le
e^{3\alpha/2}\left[ 2\sqrt{\Psi(h)} + \frac{1}{\sqrt{2}}\|h\| \sqrt{c} \sqrt{\Psi(h)
+ \frac14 \|h\|^2c}\right].\label{eq:DQMbound:1}
\end{align}
Next, we do a second order expansion of the square root $\sqrt{q_0 + h^Tt q_0}$ around $q_0$ to get
\begin{align*}
\left( \sqrt{q_0 + h^Tt q_0} - \sqrt{q_{0}} - \frac12 h^T t\sqrt{q_0}\right)^2
&=
\left( 
- \frac18 \zeta^{-3/2} [h^Ttq_0]^2
\right)^2\\
&\le
\frac{1}{64} \zeta^{-3} \|h\|^4 e^{4\alpha} (\E_{P_0}[\|s\|^2])^2,
\end{align*}
where $\zeta$ is an intermediate value between $q_0 + h^Tt q_0$ and $q_0$. At the beginning of the proof we already had $q_0 + h^Tt q_0\ge e^{-\alpha} - e^\alpha \|h\|\sqrt{c}>0$, and thus 
$$
|\zeta^{-3/2}|\le (e^{-\alpha} - e^\alpha \|h\|\sqrt{c})^{-3/2}.
$$
We therefore arrive at
\begin{align*}
\left( \sqrt{q_1} - \sqrt{q_0} - \frac{1}{2} h^T t\sqrt{q_0}\right)^2  
&\le 2 \left( \sqrt{q_1} - \sqrt{q_0 + h^Tt q_0}\right)^2  \\
&\quad +
2 \left( \sqrt{q_0 + h^Tt q_0} - \sqrt{q_{0}} - \frac12 h^T t\sqrt{q_0}\right)^2\\
&\le
2e^{3\alpha}\left[ 4\Psi(h) + \frac{1}{2}\|h\|^2 c \left(\Psi(h)
+ \frac14 \|h\|^2c\right)\right]\\
&\quad+ \frac{1}{32}(e^{-\alpha} - e^\alpha \|h\|\sqrt{c})^{-3}\|h\|^4e^{4\alpha}c^2.
\end{align*}
\end{proof}

\subsection{Proof of Theorem~\ref{thm:LAMN}}
Fix $\theta\in\Theta$ and write $X_1,\dots, X_n$ for the original data, which are iid random variables with distribution $P_\theta$, and let $(Z_1,\dots, Z_n)\thicksim Q^{(n)}P_\theta^n =:R_{n,\theta}$ denote the sanitized data. We adopt the notation of Lemma~\ref{lemma:QdensitiesSI}. For $i\in[n]$, let $r_{i,\theta}$ denote a joint $\nu^{(i)}$-density of $Z_1,\dots, Z_i$ as in part~\ref{lemma:QdensitiesSI:a}) of that Lemma. By part~\ref{lemma:QdensitiesSI:b}) of the same lemma we can pick a measurable function $q_i:\mathcal Z_i\times\X\times\mathcal Z^{(i-1)}\to[0,\infty)$ independent of $\theta$, such that for every $\theta\in\Theta$ and for every $z_{1:i-1}\in\mathcal Z^{(i-1)}$, $(z,x)\mapsto q_i(z|x,z_{1:i-1}) p_\theta(x)$ is a joint $\nu_{z_{1:i-1}}\otimes\mu$-density of $Q_{z_{1:i-1}}(dz|x)P_\theta(dx)$, where $p_\theta = \frac{dP_\theta}{d\mu}$. Moreover, for every $z_{1:i-1}\in\mathcal Z^{(i-1)}$ and for $\nu_{z_{1:i-1}}\otimes\mu$-almost all $(z,x)$, we have $q_i(z|x,z_{1:i-1}) \in[e^{-\alpha},e^\alpha]$ and write $N(z_{1:i-1})$ for the corresponding null set and $N(z_{1:i-1},x)$ for its $x$-section and $N(z_{1:i-1},z)$ for its $z$-section, that is, $z\in N(z_{1:i-1},x)$ if, and only if, $x\in N(z_{1:i-1},z)$, if and only if $(z,x)\in N(z_{1:i-1})$. In particular, we have $0 = \nu_{z_{1:i-1}}\otimes \mu (N(z_{1:i-1})) = \int_\X \nu_{z_{1:i-1}}(N(z_{1:i-1},x))\mu(dx)$, such that $N(z_{1:i-1},x)$ is a $\nu_{z_{1:i-1}}$-null set for all $x$ outside of a $\mu$-null set $N_\mu(z_{1:i-1})$ and similarly, $N(z_{1:i-1},z)$ is a $\mu$-null set for all $z$ outside of a $\nu_{z_{1:i-1}}$-null set $N_\nu(z_{1:i-1})$. 

Now, for fixed $z_{1:i-1}\in\mathcal Z^{(i-1)}$ and $z\notin N_\nu(z_{1:i-1})$, we have $\mu(\{x: e^{-\alpha}\le q_i(z|x,z_{1:i-1}) \le e^\alpha\}) \ge \mu(N(z_{1:i-1},z)^c) = 1$. Thus, we may apply Lemma~\ref{lemma:DQMbound} to get for every $\theta\in\Theta$, $h\in\R^p$ with $\|h\|\sqrt{c}\le e^{-2\alpha}$ and $\theta+h\in\Theta$,
\begin{align}
&\left( \sqrt{q_{i,\theta+h}(z|z_{1:i-1})} - \sqrt{q_{i,\theta}(z|z_{1:i-1})} - \frac12 h^T t_{i,\theta}(z|z_{1:i-1})\sqrt{q_{i,\theta}(z|z_{1:i-1})}\right)^2\notag \\
&\quad\le
2e^{3\alpha}\left[ 2\Psi(h) + \frac{1}{2}\|h\|^2 c \left(\Psi(h)
+ \frac14 \|h\|^2c\right)\right]
+ \frac{1}{32}(e^{-\alpha} - e^\alpha \|h\|\sqrt{c})^{-3}\|h\|^4e^{4\alpha}c^2, \label{eq:CondDQM}
\end{align}
where $q_{i,\theta}(z|z_{1:i-1}) := \int_\X q_i(z|x,z_{1:i-1}) p_\theta(x) \mu(dx)$, $c:= c_\theta := \E_\theta[\|s_\theta\|^2]$, 
$$
\Psi(h) := \int_\X \left( \sqrt{p_{\theta+h}(x)} - \sqrt{p_\theta(x)} - \frac12 h^Ts_\theta(x)\sqrt{p_\theta(x)}\right)^2\mu(dx),
$$
and 
$$
t_{i,\theta}(z|z_{1:i-1}) := \int_\X s_{\theta}(x) \frac{q_i(z|x,z_{1:i-1})p_{\theta}(x)}{q_{i,\theta}(z|z_{1:i-1})}\mu(dx),
$$ 
which is bounded by $\|t_{i,\theta}(z|z_{1:i-1})\|\le e^{2\alpha}\sqrt{I_\theta(\P)}$.
Also observe that $z\mapsto q_{i,\theta}(z|z_{1:i-1})$ is a $\nu_{z_{1:i-1}}$-density of $Q_{z_{1:i-1}}P_\theta$, a fact we will use repeatedly throughout this proof. Therefore, since $N_\nu(z_{1:i-1})$ is a $\nu_{z_{1:i-1}}$-null set and the upper bound in \eqref{eq:CondDQM} is $o(\|h\|^2)$ if $\P$ is DQM at $\theta\in\Theta$, we conclude that $z\mapsto t_{i,\theta}(z|z_{1:i-1})$ is a score at $\theta$ in the model $Q_{z_{1:i-1}}\P$ which is measurable in $z_{1:i}\in\mathcal Z^{(i)}$, by construction. In particular, the Fisher-Information is given by the variance of the score, that is, for all $z_{1:i-1}\in\mathcal Z^{(i-1)}$,
$$
I_\theta(Q_{z_{1:i-1}}\P) = \int_{\mathcal Z_i} t_{i,\theta}(z|z_{1:i-1})t_{i,\theta}(z|z_{1:i-1})^T [Q_{z_{1:i-1}}P_\theta](dz)
$$
and $\|t_{i,\theta}(z|z_{1:i-1})\|\le e^{2\alpha}\sqrt{I_\theta(\P)}$ for $\nu_{z_{1:i-1}}$-almost all $z$. Chasing the definitions we conclude that the Fisher-Information matrix is measurable as a function of $z_{1:i-1}\in\mathcal Z^{(i-1)}$.

Next, we aim for part~1 of Definition~\ref{def:LAMN}. As in part~\ref{lemma:QdensitiesSI:c}) of Lemma~\ref{lemma:QdensitiesSI}, we have $q_{i,\theta}(z_i|z_{1:i-1}) r_{i-1,\theta}(z_{1:i-1}) = r_{i,\theta}(z_1,\dots, z_i)$, $\nu^{(i)}$-almost surely, and thus $r_{n,\theta}(z_1,\dots, z_n) = \prod_{i=1}^n q_{i,\theta}(z_i|z_{1:i-1})$ is a $\nu^{(n)}$-density of the distribution $R_{n,\theta}$ of the sanitized data $Z_1,\dots, Z_n$. For any $\theta,\theta'\in\Theta$ and $i\in[n]$ the probability measures $\nu^{(i)}$, $Q^{(i)}P_\theta^i$ and $Q^{(i)}P_{\theta'}$ are all bounded by a multiple of $e^{i\alpha}$ of each other (cf. Lemma~\ref{lemma:QdensitiesSI}). In particular, this entails that $0<r_{i,\theta} \le e^{i\alpha} r_{i,\theta'}$ holds $\nu^{(i)}$-almost surely. Fix $h\in\R^p$, set $h_n:=h/\sqrt{n}$ and notice that for all $i\in[n]$ and outside of a $\nu^{(i)}$-null set $N_i\in\mathcal G^{(i)}$,
$$
\eta_{i,\theta,h_n}(z_{1:i}) := \sqrt{\frac{q_{i,\theta+h_n}(z_i|z_{1:i-1})}{q_{i,\theta}(z_i|z_{1:i-1})}} - 1 
=
\sqrt{\frac{r_{i,\theta+h_n}(z_{1:i})}{r_{i,\theta}(z_{1:i})}\frac{r_{i-1,\theta}(z_{1:i-1})}{r_{i-1,\theta+h_n}(z_{1:i-1})}} - 1 
 \ge e^{-i\alpha}-1.
$$
Now expand $x\mapsto \log(1+x)$ around $0$ to third order and operate on the complement of the $\nu^{(n)}$-null set $N= \bigcup_{i=1}^n \left(N_i\times \mathcal Z_{i+1} \times \cdots \times \mathcal Z_n\right)$, to get
\begin{align}
\Lambda_n(\theta+&h_n,\theta)(z) = \log\frac{d R_{n,\theta+h_n}}{d R_{n,\theta}}(z) = \log\prod_{i=1}^n\frac{q_{i,\theta+h_n}(z_i|z_{1:i-1})}{q_{i,\theta}(z_i|z_{1:i-1})}\notag\\
&=2 \sum_{i=1}^n \log\left(1 +  \eta_{i,\theta,h_n}(z_{1:i})\right)\notag\\
&=2\sum_{i=1}^n \eta_{i,\theta,h_n}(z_{1:i}) - \sum_{i=1}^n \eta_{i,\theta,h_n}(z_{1:i})^2  + \sum_{i=1}^n \frac{2}{3[1+\xi_{i,\theta,h_n}(z_{1:i})]^3} \eta_{i,\theta,h_n}(z_{1:i})^3,\label{eq:proof:LAMN:expansion}
\end{align}
where $\xi_{i,\theta,h_n}(z_{1:i})$ is an intermediate value between $0$ and $\eta_{i,\theta,h_n}(z_{1:i})\ge e^{-\alpha}-1$. Thus, $\frac23 [1+\xi_{i,\theta,h_n}(z_{1:i})]^{-3} \le \frac23 e^{3\alpha}$ and the last term in the previous display (which is clearly measurable as all other terms in the equation are) is bounded in absolute value by
\begin{equation}\label{eq:proof:LAMN:o1}
\frac23 e^{3\alpha} \max_{j=1,\dots, n} |\eta_{j,\theta,h_n}(z_{1:j})| \sum_{i=1}^n \eta_{i,\theta,h_n}(z_{1:i})^2.
\end{equation}
Now, with $U_i := \eta_{i,\theta,h_n} - \frac12 h_n^Tt_{i,\theta} = (\sqrt{q_{i,\theta+h_n}}  - \sqrt{q_{i,\theta}} -\frac12 h_n^T t_{i,\theta}\sqrt{q_{i,\theta}})/\sqrt{q_{i,\theta}}$,
\begin{align*}
&R_{n,\theta} \left( \max_{i=1,\dots, n} |\eta_{i,\theta,h_n}| >\eps \right) 
\le
\sum_{i=1}^n R_{n,\theta}\left( |\eta_{i,\theta,h_n}| >\eps\right)\\
&\quad\le
\sum_{i=1}^n \left[R_{n,\theta}\left( \left|\eta_{i,\theta,h_n} - \frac12 h_n^Tt_{i,\theta}\right| >\frac{\eps}{2}\right)
 +R_{n,\theta}\left( \left|\frac12 h_n^Tt_{i,\theta}\right| >\frac{\eps}{2}\right)\right]\\
 &\quad\le 
\left(\frac{2}{\eps}\right)^2\E_{n,\theta}\left[ \sum_{i=1}^n  \E_{n,\theta}\left[ U_i^2| Z_{1:i-1}\right] \right]
 + n \mathds 1_{\left\{e^{2\alpha}\sqrt{\trace{I_\theta(\P)}}\|h\| > \eps\sqrt{n} \right\}}.
\end{align*}
For all sufficiently large $n$, the last term in the previous display is constant equal to zero. Moreover, we already showed in \eqref{eq:CondDQM}  that a version of $\E_{n,\theta}[U_i^2|Z_{1:i-1}]$ is bounded, independently of $i$, by a deterministic quantity that is $o(\|h_n\|^2) = o(n^{-1})$ as $n\to\infty$. We conclude that $\max_{i=1,\dots, n} |\eta_{i,\theta,h_n}| = o_P(1)$. 

For the sum of squared $\eta_{i,\theta,h_n}$ we write
\begin{align}
\sum_{i=1}^n &\eta_{i,\theta,h_n}(z_{1:i})^2 - \frac14 h_n^T\Sigma_{n,\theta}(z_{1:n-1})h_n \label{eq:sumsquEta}\\
&=
\sum_{i=1}^n  \left[\eta_{i,\theta,h_n}(z_{1:i})^2 - \frac14 h_n^TI_{\theta}(Q_{z_{1:i-1}}\P)h_n\right]\notag\\
&=
\sum_{i=1}^n  \left[\eta_{i,\theta,h_n}(z_{1:i})^2 - \left( \frac12 h_n^T t_{i,\theta}(z_{1:i})\right)^2\right]  \label{eq:sumsquEta1}\\
&\quad+
\sum_{i=1}^n \left[\left( \frac12 h_n^T t_{i,\theta}(z_{1:i})\right)^2  - \frac14 h_n^T I_{\theta}(Q_{z_{1:i-1}}\P)h_n\right] .\label{eq:sumsquEta2}
\end{align}
For the summands in \eqref{eq:sumsquEta1} use Lemma~\ref{lemma:c-d} to bound their expected absolute value by
\begin{align*}
(1+\gamma) \E_{n,\theta}\left[ \left| \eta_{i,\theta,h_n} - \frac12 h_n^T t_{i,\theta}\right|^2\right] +\frac{1}{\gamma}\E_{n,\theta}\left[\left( \frac12 h_n^T t_{i,\theta}\right)^2 \right],
\end{align*}
for arbitrary $\gamma>0$.
But we already showed that the first of these two expectations is bounded independently of $i$ by $o(n^{-1})$. For the second we use the previously established almost sure bound $\|t_{i,\theta}\|\le e^{2\alpha}\sqrt{\trace{I_{\theta}(\P)}}$ and $\|h_n\|^2 = \|h\|^2/n$. Thus,
$$
\limsup_{n\to\infty} E_{n,\theta}(|\eqref{eq:sumsquEta1}|) \le \frac{\|h\|^2e^{4\alpha}\trace{I_{\theta}(\P)}}{4\gamma} \xrightarrow[\gamma\to\infty]{}0.
$$ 
For \eqref{eq:sumsquEta2}, write $Y_i(z_{1:i}) := \left(\frac12 h_n^T t_{i,\theta}(z_{1:i})\right)^2  - \frac14 h_n^T I_{\theta}(Q_{z_{1:i-1}}\P)h_n$ and recall that $I_\theta(Q_{Z_{1:i-1}}\P) = \E_{n,\theta}\left[t_{i,\theta}t_{i,\theta}^T\Big| Z_{1:i-1}\right]$, almost surely. Therefore, $\E_{n,\theta}[Y_i|Z_1,\dots, Z_{i-1}] =0$ almost surely, for $i<j$ we have $\E_{n,\theta}[Y_iY_j] = \E_{n,\theta}(Y_i\E_{n,\theta}[Y_j|Z_1,\dots, Z_{j-1}])=0$ and $Y_i^2 \le \frac{1}{16} \|h_n\|^4 (e^{8\alpha}+1)[\trace I_\theta(\P)]^2 = O(n^{-2})$, where we used that $\trace{I_{\theta}(Q_{z_{1:i-1}}\P)}\le \trace{I_\theta(\P)}$ in view of Lemma~\ref{lemma:DQM}. Hence \eqref{eq:sumsquEta} converges to zero in $R_{n,\theta}$-probability. Since the second term of \eqref{eq:sumsquEta} is bounded by $\frac14h^TI_\theta(\P)h$ (again, use Lemma~\ref{lemma:DQM}), we also see that $\sum_{i=1}^n \eta_{i,\theta,h_n}^2 = O_P(1)$ and thus \eqref{eq:proof:LAMN:o1} is $o_P(1)$. Altogether we already showed that
\begin{align*}
\Lambda_n(\theta+h_n,\theta) 
=2\sum_{i=1}^n \eta_{i,\theta,h_n} - \frac14 h^T\Sigma_{n,\theta}h+ o_P(1).
\end{align*}
Finally, write
\begin{align}
&2\sum_{i=1}^n \eta_{i,\theta,h_n} - \left( h^T\Delta_{n,\theta} - \frac14 h^T\Sigma_{n,\theta}h\right)\notag\\
&=
2\sum_{i=1}^n \left(U_i - \E_{n,\theta}[U_i|Z_{1:i-1}]\right) + 2\sum_{i=1}^n \left(\E_{n,\theta}[U_i|Z_{1:i-1}] +\frac12 \E_{n,\theta}\left[ \left( \frac12 h_n^Tt_{i,\theta}\right)^2\Bigg| Z_{1:i-1}\right] \right).\label{eq:proof:sumEtai}
\end{align}
The first sum has mean zero and variance bounded by the expectation of $\sum_{i=1}^n \E_{n,\theta}[U_i^2|Z_{1:i-1}]$ of which we already showed that it is almost surely bounded by $n\cdot o(n^{-1})$. For the second sum, notice that $\E_{n,\theta}[U_i|Z_{1:i-1}] = \E_{n,\theta}[\eta_{i,\theta,h_n} |Z_{1:i-1}] - \frac12 h_n^T\E_{n,\theta}[ t_{i,\theta}|Z_{1:i-1}]$. Because $t_{i,\theta}(z|z_{1:i-1})$ is a score in $Q_{z_{1:i-1}}\P$, we have
\begin{align*}
&\E_{n,\theta}[ t_{i,\theta}|Z_{1:i-1}=z_{1:i-1}] =  
\int_{\mathcal Z_i} t_{i,\theta}(z|z_{1:i-1}) Q_{z_{1:i-1}}P_\theta(dz)=0.
\end{align*}
Moreover,
\begin{align*}
&\E_{n,\theta}[ \eta_{i,\theta,h_n} |Z_{1:i-1}=z_{1:i-1}] =  \int_{\mathcal Z_i} \eta_{i,\theta,h_n}(z_{1:i})q_{i,\theta}(z_i|z_{1:i-1}) \nu_{z_{1:i-1}}(dz_i)\\
&\quad =\int_{\mathcal Z_i} \sqrt{q_{i,\theta+h_n}(z_i|z_{1:i-1})q_{i,\theta}(z_i|z_{1:i-1})} \nu_{z_{1:i-1}}(dz_i) - 1\\
&\quad= -\frac12 \int_{\mathcal Z_i} \left[q_{i,\theta+h_n}(z_i|z_{1:i-1}) - 2\sqrt{q_{i,\theta+h_n}(z_i|z_{1:i-1}) q_{i,\theta}(z_i|z_{1:i-1})} + q_{i,\theta}(z_i|z_{1:i-1}) \right]\nu_{z_{1:i-1}}(dz_i)  \\
&\quad= -\frac12 \int_{\mathcal Z_i} \left[\sqrt{q_{i,\theta+h_n}(z_i|z_{1:i-1})} - \sqrt{q_{i,\theta}(z_i|z_{1:i-1})}\right]^2 \nu_{z_{1:i-1}}(dz_i) \\
&\quad= -\frac12 \E_{n,\theta}\left[\eta_{i,\theta,h_n}^2 \Big|Z_{1:i-1}=z_{1:i-1}\right].
\end{align*}
Thus, $\E_{n,\theta}[U_i|Z_{1:i-1}] = -\frac12 \E_{n,\theta}[\eta_{i,\theta,h_n}^2 |Z_{1:i-1}]$ and the second sum in \eqref{eq:proof:sumEtai} can be handled in the exact same way as \eqref{eq:sumsquEta1}. We conclude that
\begin{align*}
\Lambda_n(\theta+h_n,\theta) 
=h^T\Delta_{n,\theta} - \frac12 h^T\Sigma_{n,\theta}h+ o_P(1)
\end{align*}
as required for the first part of Definition~\ref{def:LAMN}.

For the weak convergence of the second part, recall that because of $0\preccurlyeq I_\theta(Q_{z_{1:i-1}}\mathcal P) \preccurlyeq I_\theta(\P)$ the entries of $\Sigma_{n,\theta}$ are uniformly bounded by a finite deterministic constant, $C_{\theta}$, say. Moreover, since, as above, $\E_{n,\theta}[t_{i,\theta}|Z_{1:i-1}]=0$ and $|h^Tt_{i,\theta}|\le \|h\|e^{2\alpha}\sqrt{I_\theta(\P)}$ almost surely, Hoeffding's lemma yields
$$
\E_{n,\theta}[\exp(h^Tt_{i,\theta}/\sqrt{n})|Z_{1:i-1}] \le \exp\left(\frac{\|h\|^2e^{4\alpha}I_\theta(\P)}{2n}\right).
$$
Because of $\Delta_{n,\theta} = \sum_{i=1}^n t_{i,\theta}/\sqrt{n}$ we therefore get
$$
\mathcal M_{\Delta}(h):=\E_{n,\theta}[\exp(h^T\Delta_{n,\theta})] = \E_{n,\theta}\left[\prod_{i=1}^n \exp(h^T t_{i,\theta}/\sqrt{n})\right]
\le
\exp\left(\frac{\|h\|^2e^{4\alpha}I_\theta(\P)}{2}\right).
$$
Thus, the joint moment generating function of $(\Delta_{n,\theta},\Sigma_{n,\theta})$ is given by
$$
\mathcal M_n(\hbar) := \E_{n,\theta}\left[ \exp\left(\sum_{l=1}^p\hbar_l[\Delta_{n,\theta}]_l+ \sum_{j=1}^{p} \sum_{k=1}^j \hbar_{p+(j-1)j/2+k}[\Sigma_{n,\theta}]_{j,k}\right)\right],
$$
and is finite for all $\hbar\in\R^{p+p(p+1)/2}$. In particular, this can be written as
$\mathcal M_n(\hbar) = \E_{n,\theta}[u(\Sigma_{n,\theta})\exp(h^T\Delta_{n,\theta})]$, where $u:\R^{p\times p}\to\R$ is a continuous and bounded function and $h\in\R^p$ are the first $p$ entries of $\hbar$. Next, abbreviate $\Gamma_{n,\theta} := \Lambda_n(\theta+h_n,\theta) - h^T\Delta_{n,\theta} + \frac12 h^T\Sigma_{n,\theta}h$, of which we already know that it is $o_P(1)$ under $R_{n,\theta}$. With this notation we have
\begin{align*}
\mathcal M_n(\hbar) &= \E_{n,\theta+h_n}\left[u(\Sigma_{n,\theta})\exp\left(h^T\Delta_{n,\theta}-\Lambda_n(\theta+h_n,\theta)\right) \right]\\
&=
\E_{n,\theta+h_n}\left[u(\Sigma_{n,\theta})\exp\left(\frac12 h^T\Sigma_{n,\theta}h\right) \exp\left(- \Gamma_{n,\theta}\right) \right].
\end{align*}
Notice that $u(\Sigma_{n,\theta})\exp\left(\frac12 h^T\Sigma_{n,\theta}h\right) = v(\Sigma_{n,\theta})$ for a bounded continuous function $v:\R^{p\times p}\to\R$ of $\Sigma_{n,\theta}$ and by our weak convergence assumption on $\Sigma_{n,\theta}$, we have $\E_{n,\theta}\left[v(\Sigma_{n,\theta})- \E[v(\Sigma_\theta)]\right]\to0$. In view of mutual contiguity of $R_{n,\theta}$ and $R_{n,\theta+h_n}$ (cf. Lemma~\ref{lemma:contiguity}) and LeCam's first lemma \citep[cf. part~(iv) of Lemma~6.4 in][]{vanderVaart07} we also have
$\E_{n,\theta+h_n}\left[v(\Sigma_{n,\theta}) - \E[v(\Sigma_\theta)]\right]\to0$. Thus, if we can show that $\exp\left(- \Gamma_{n,\theta}\right)\to1$ in $L_1(R_{\theta+h_n})$, we obtain $\mathcal M_n(\hbar) \to \E[v(\Sigma_\theta)] = \E[u(\Sigma_\theta)\exp\left(\frac12 h^T\Sigma_{\theta}h\right)]$. But the latter expression is equal to
\begin{align*}
\E\left[u(\Sigma_\theta)\exp\left(\frac12 h^T\Sigma_{\theta}h\right)\right] &= \E\left[u(\Sigma_\theta)\E\left[\exp(h^T\Sigma_\theta^{1/2}\Delta)\Big|\Sigma_\theta\right]\right]\\
&=\E\left[u(\Sigma_\theta)\exp(h^T\Sigma_\theta^{1/2}\Delta)\right],
\end{align*}
for a random vector $\Delta\thicksim N(0,I_p)$ independent of $\Sigma_\theta$. In other words, the joint moment generating function of $(\Delta_{n,\theta},\Sigma_{n,\theta})$, which is finite, converges to the joint moment generating function of $(\Sigma_\theta^{1/2}\Delta,\Sigma_\theta)$, which implies the desired weak convergence. Thus, it remains to show $\exp\left(- \Gamma_{n,\theta}\right)\to1$ in $L_1(R_{\theta+h_n})$. By contiguity, $\Gamma_{n,\theta} = o_P(1)$ under $R_{n,\theta+h_n}$. To get the desired $L_1$-convergence, it remains to show uniform integrability of $\exp\left(- \Gamma_{n,\theta}\right)$ with respect to $R_{\theta+h_n}$. Since $\frac12 h^T\Sigma_{n,\theta}h$ is bounded by a finite constant uniformly in $n$, it actually suffices to show uniform integrability of $V_n:=\exp(-\Lambda_n(\theta+h_n,\theta) + h^T\Delta_{n,\theta})$. Sufficient conditions for this to be the case are (i) uniform boundedness of the expected values and (ii) for any sequence of events $B_n$ with probability converging to zero we have $\E_{n,\theta+h_n}[V_n\mathds 1_{B_n}] \to0$. For (i) simply observe
$$
\E_{n,\theta+h_n}[V_n] = \E_{n,\theta+h_n}\left[\exp(h^T\Delta_{n,\theta})\frac{dR_{n,\theta}}{dR_{n,\theta+h_n}}\right]
= \E_{n,\theta}\left[\exp(h^T\Delta_{n,\theta})\right],
$$
which was already shown to be bounded by $\exp(\|h\|^2e^{4\alpha}I_\theta(\P)/2)$. Similarly, for (ii) write
$$
\E_{n,\theta+h_n}[V_n\mathds 1_{B_n}] = \E_{n,\theta}\left[\exp(h^T\Delta_{n,\theta})\mathds 1_{B_n}\right] \le 
\sqrt{\E_{n,\theta}\left[\exp(2h^T\Delta_{n,\theta})\right] R_{n,\theta}(B_n) },
$$
the upper bound now converging to zero in view of contiguity $R_{n,\theta}\vartriangleleft R_{n,\theta+h_n}$.\hfill\qed

\begin{example}\label{ex:noWeakConv}\normalfont
Consider the Bernoulli model of Subsection~\ref{sec:Bernoulli}, i.e., $\P=(p_\theta)_{\theta\in\Theta}$, $p_\theta(x) = \theta^x(1-\theta)^{1-x}$, $x\in\X=\{0,1\}$, $\theta\in\Theta=(0,1)$. For a privacy parameter $\alpha>0$, let $Q_{0,\alpha}$ be the optimal privacy mechanism from that subsection generating sanitized data in $\mathcal Z=\{0,1\}$ and define the sequentially interactive mechanism from $\X^n$ to $\mathcal Z^n$ by $Q^{(n)}(dz|x) := Q_{z_1,n}(dz_n|x_n)Q_{z_1,n}(dz_{n-1}|x_{n-1})\dots Q_{z_1,n}(dz_2|x_2)Q_{\varnothing,n}(dz_1|x_1)$, where $Q_{\varnothing,n}(dz_1|x_1) := \text{Bernoulli}(\frac13)$ if $n$ is even and $Q_{\varnothing,n}(dz_1|x_1) = \text{Bernoulli}(\frac23)$ if $n$ is odd, and $Q_{z_1,n}(dz_i|x_i) := Q_{0,\alpha}(dz_i|x_i)$, if $z_1=1$ and $Q_{z_1,n}(dz_i|x_i) := Q_{0,\alpha/2}(dz_i|x_i)$, if $z_1=0$, $i=2,\dots, n$. Since the marginal channels $Q_{z_1,n}$ and $Q_{\varnothing,n}$ are elements of $\mathcal Q_\alpha$, we see that $Q^{(n)}$ is $\alpha$-differentially private. Thus
\begin{align*}
\Sigma_{n,\theta}(z) &= \frac{1}{n}\left( (n-1)I_\theta(Q_{z_1,n}\P)+ I_\theta(Q_{\varnothing,n}\P)\right) \\
&=
\frac{n-1}{n} [z_1 I_\theta(Q_{0,\alpha}\P) + (1-z_1)I_\theta(Q_{0,\alpha/2}\P)],
\end{align*}
but the distribution of $Z_1$ is either Bernoulli$(\frac13)$ or Bernoulli$(\frac23)$, depending on whether $n$ is even or odd, and hence it does not settle down at any fixed distribution for $n\to\infty$.
\end{example}

%================================================================================================

\section{Proofs of Section~4.2 on estimation of the optimal privacy mechanism}
\label{sec:proofsEstQ}

\subsection{Proof of Lemma~\ref{lemma:ConsQuantExists}}

We begin by fixing an arbitrary sequence $(\gamma_m)_{m\in\N}$ with $(0,1)\ni\gamma_m\to0$. By our measurability assumption, the following functions are measurable: 
$\theta\mapsto F_\theta(t) := P_\theta(\{x:\|x\|_2 \le t\}) = \int_\X \mathds 1_{\{x:\|x\|_2\le t\}}(y) p_\theta(y)\mu(dy)$, for all $t\in\R$, and $\theta\mapsto P_\theta(p_\theta\ge n) = \int_\X \mathds 1_{\{p_\theta\ge n\}} p_\theta d\mu $, for all $n\in\N$. Therefore, also $\theta\mapsto F_\theta^\dagger(1-\gamma_m):=\inf\{t\in\R:F_\theta(t)\ge 1-\gamma_m\}$ is measurable, for all $m\in\N$, because $\{\theta\in\Theta: F_\theta^\dagger(1-\gamma_m) \le t\} = \{\theta\in\Theta: 1-\gamma_m\le F_\theta(t)\}\subseteq\R^p$ is a Borel set, for all $t\in\R$. Moreover, $\theta\mapsto N_{m,1}(\theta) := \min\{n\in\N: P_\theta(p_\theta\ge n)\le \gamma_m\}$ is also measurable, because $\{\theta\in\Theta:N_m(\theta)\le t\}=\bigcup_{n\in\N}\{\theta\in\Theta: P_\theta(p_\theta> n)\le \gamma_m, n\le t\}$ is a countable union of Borel sets, for every $t\in\R$. By an analogous argument, also $\theta\mapsto N_{m,2}(\theta) := \min\{n\in\N: P_\theta(p_\theta< \frac1n)\le \gamma_m\}$ and $\theta\mapsto N_{m,3}(\theta) := \min\{n\in\N: P_\theta(|\dot{p}_\theta|\ge n)\le \gamma_m\}$ are measurable.

Now, define $L_m(\theta) := \{x\in\X: \|x\|_2 \le F_\theta^\dagger(1-\gamma_m)\}$ and notice that $(x,\theta)\mapsto\mathds 1_{L_m(\theta)}(x)$ is also measurable, and so is $\theta\mapsto \mu(L_m(\theta)) = \int_\X \mathds 1_{L_m(\theta)}(x) \mu(dx)$, which is also finite in view of compactness of $L_m(\theta)$. To close the preliminary considerations, set 
$$
M:=M_m(\theta):= \max\left\{\left\lceil \frac{\mu(L_m(\theta))}{\gamma_m}\right\rceil, N_{m,1}(\theta), N_{m,2}(\theta), N_{m,3}(\theta)\right\},
$$ 
which is also measurable in $\theta$.

Next, define $C_{i,m}(\theta) := \{x\in\X: \frac{i-1}{M}\le p_\theta(x) < \frac{i}{M}\}\cap L_m(\theta)$, for $i=2,\dots, M^2$, $D_{l,m}(\theta) := \{x\in\X: \frac{l-1}{M}\le \dot{p}_\theta(x)<\frac{l}{M}\}$, for $l=-M^2+1,\dots, M^2$, and $C_{i,l,m}(\theta) := C_{i,m}(\theta) \cap D_{l,m}(\theta)$, and let $B_j:=B_{j,m}(\theta)$, $j=1, 2,\dots, (M^2-1)(2M^2)=:k_m(\theta)-1$, be an enumeration of all the $C_{i,l,m}(\theta)$, $i=2,\dots, M^2$, $l=-M^2+1,\dots, M^2$. Notice that the $B_{j,m}(\theta)$ are disjoint and have finite $\mu$-measure, because $\mu(B_{j,m}(\theta)) \le \mu(p_\theta\ge\frac1M)\le M$, by Markov's inequality. In particular, with $B_{k_m(\theta)} :=\left(\bigcup_{j=1}^{k_m(\theta)-1} B_j \right)^c$, we obtain a partition of $\X$ of $k_m(\theta)$ sets such that $(x,\theta)\mapsto \mathds 1_{B_{j,m}(\theta)}(x)$ is measurable for all $j$. Now set $J_m(\theta) := \{j\in[k_m(\theta)-1]: \mu(B_{j,m}(\theta))>0\}$ and observe
\begin{align}
\bigcup_{j=1}^{k_m(\theta)-1} B_{j,m}(\theta) &= \left[\bigcup_{i=2}^{M^2} C_{i,m}(\theta)\right] \cap \left[\bigcup_{l=-M^2+1}^{M^2} D_{l,m}(\theta)\right]\notag\\
&= \left\{\frac{1}{M}\le p_\theta < M\right\}\cap \left\{ -M\le \dot{p}_\theta <M \right\} \cap L_m(\theta).\label{eq:KJbound}
\end{align}
Thus, with $K(\theta) := K_{J,m}(\theta):=\bigcup_{j\in J} B_{j,m}(\theta)$ and omitting only $\mu$-null-sets, we get
\begin{align*}
P_\theta(K_{J,m}(\theta)^c) &= P_\theta\left( \left[\bigcup_{j=1}^{k_m(\theta)-1} B_{j,m}(\theta)\right]^c\right)\\
&\le P_\theta\left(p_\theta<\frac{1}{M}\right) + P_\theta\left( p_\theta\ge M\right)+P_\theta\left( |\dot{p}_\theta|\ge M\right)+ P_\theta(L_m(\theta)^c)\\
&\le P_\theta\left(p_\theta<\frac{1}{N_{m,2}(\theta)}\right) +P_\theta\left( p_\theta \ge N_{m,1}(\theta)\right)+P_\theta\left( |\dot{p}_\theta|\ge N_{m,3}(\theta)\right)+ P_\theta(L_m(\theta)^c)\\
&\le
3\gamma_m  + P_\theta(\{x\in\X:F_\theta^\dagger(1-\gamma_m)< \|x\|_2\}) \le 4\gamma_m.
\end{align*}
Next, we compute the histogram approximation errors, noting that for $x,y\in B_j$, $j\in J$, we have $|p_\theta(x) - p_\theta(y)|\le \frac{1}{M}$. Therefore,
\begin{align*}
\Big\|(p_\theta - \bar{p}_\theta)\mathds 1_{K(\theta)}\Big\|_{L_1(\mu)} &\le 
\sum_{j\in J}\frac{\int_{B_{j}^2}  |p_\theta(x) - p_\theta(y)| \mu(dy) \mu(dx)}{\mu(B_{j})}
\le \sum_{j\in J}\frac{\mu(B_{j})}{M}= \frac{\mu(K(\theta))}{M}.
\end{align*}
But in view of \eqref{eq:KJbound} and the definition of $M$ we have 
$$
\frac{\mu(K(\theta))}{M}\le  \frac{\mu(L_m(\theta))}{M}\le \gamma_m.
$$
The same argument applies for the $L_1(\mu)$-norm involving $\dot{p}_\theta$. Thus, we see that even $\sup_{\theta\in\Theta}\Delta_m(\theta)\to0$ as $m\to\infty$.\hfill\qed

%============================================================

\subsection{Proof of Lemma~\ref{lemma:finitedimlOpti}}

Parts~\ref{lemma:finitedimlOpti:i} and \ref{lemma:finitedimlOpti:ii} do not require a proof.

For Part~\ref{lemma:finitedimlOpti:iii}, notice that $Q_{i\cdot}\in\mathcal C_\alpha(k)$ and, if $Q_{ij}=0$ for some $i,j\in[k]$ then also $Q_{ij'}=0$ for all $j'\in[k]$. For ease of notation we also write $q(i|j) := Q_{ij}$. By Lemma~\ref{lemma:pre-postProcessing} $QT_k \in \mathcal Q_\alpha(\X\to[\ell])$ and since for every $x\in\X$ and $C\subseteq[\ell]$, $\sum_{i\in C}q(i|T_k(x)) = Q(C|T_k(x))$, we see that $i\mapsto q(i|T_k(x))$ is a counting density of $[QT_k](\cdot|x)$ and that $i\mapsto q_\theta(i) := \int_\X q(i|T_k(x))p_\theta(x)\mu(dx)$ is a counting density of $QT_kP_\theta$. In particular, the three statements (a) $q_\theta(i)>0$, (b) $Q_{ij}>0$ for some $j\in[k]$ and (c) $Q_{ij}>0$ for all $j\in[k]$, are equivalent. Now by Lemma~\ref{lemma:DQM} we can express the Fisher-Information of the model $QT_k\P = [QT_k]\P$ with $Q\in\mathcal M_\alpha(\ell,k)$  as
\begin{align*}
I_\theta&(QT_k\P) = \\
&= \sum_{\substack{i\in[\ell]\\ Q_{i1}>0}} \left(\int_\X s_\theta(x)\frac{q(i|T_k(x)) p_\theta(x)}{q_\theta(i)}\mu(dx)\right)\left(\int_\X s_\theta(x)\frac{q(i|T_k(x)) p_\theta(x)}{q_\theta(i)}\mu(dx)\right)^T q_\theta(i) \\
&= \sum_{\substack{i\in[\ell]\\\|Q_{i\cdot}\|_2>0}} \frac{\left(\sum_{j\in[k]}Q_{ij}(\dot{p}_\theta)_{\cdot j}\right)\left(\sum_{j\in[k]}Q_{ij}(\dot{p}_\theta)_{\cdot j}\right)^T}{\sum_{j\in[k]} \int_{T_k^{-1}(\{j\})}q(i|T_k(x))p_\theta(x)\mu(dx)}\\
&= 
\sum_{\substack{i\in[\ell]\\\|Q_{i\cdot}\|_2>0}} \frac{\left(\dot{p}_\theta Q_{i\cdot} \right)\left( \dot{p}_\theta Q_{i\cdot}\right)^T}{Q_{i\cdot}p_\theta}=\sum_{i\in[\ell]} g_\theta(Q_{i\cdot}).
\end{align*}
The only point where continuity of $g_\theta$ might fail is $v=0$, because except for division by $v^Tp_\theta$, we are only dealing with compositions of continuous functions. For $v\in\mathcal C_\alpha(k)$ we have $v^Tp_\theta\ge \min_j v_j \ge e^{-\alpha}\|v\|_\infty\ge e^{-\alpha}\|v\|_2/\sqrt{k}$ and thus, $\|g_\theta(v)\|_2\le \sqrt{k}\|v\|_2\|\dot{p}_{\theta}\|_2^2e^\alpha\to0$ as $v\to0$. 

For Part~\ref{lemma:finitedimlOpti:iv}, pick $\lambda\in(0,1)$, $v_1,v_2\in\mathcal C_\alpha(k)$ and consider 
$$
g_\theta(\lambda v_1 + (1-\lambda)v_2) = \frac{\left( \lambda v_1^T\dot{p}_\theta + (1-\lambda) v_2^T\dot{p}_\theta\right)^2}{\lambda v_1^T{p}_\theta + (1-\lambda) v_2^T{p}_\theta}.
$$
If either one of $v_1$ or $v_2$ is equal to the null vector, we even have equality $g_\theta(\lambda v_1 + (1-\lambda)v_2) = \lambda g_\theta(v_1) + (1-\lambda) g_\theta(v_2)$. Otherwise, $v_1^Tp_\theta$ and $v_2^Tp_\theta$ are strictly positive and Lemma~\ref{lemma:convexity} yields the desired inequality $g_\theta(\lambda v_1 + (1-\lambda)v_2) \le \lambda g_\theta(v_1) + (1-\lambda) g_\theta(v_2)$. For sub-linearity, simply note that $g_\theta(cv) = \frac{c^2(v^T\dot{p}_\theta)^2}{c(v^Tp_\theta)} = c g_\theta(v)$, for every $c>0$, $v\in\mathcal C_\alpha(k)$, and use convexity to get $g_\theta(v_1+v_2) = g_\theta(2[\frac12 v_1 + \frac12 v_2]) = 2 g_\theta(\frac12 v_1 + \frac12 v_2) \le 2 [\frac12 g_\theta(v_1) + \frac12  g_\theta(v_2)] = g_\theta(v_1) + g_\theta(v_2)$. \hfill\qed

%======================================================================

\subsection{Proof of Lemma~\ref{lemma:measurability}}

We proceed by patching $Q$ together from its restrictions on the sets $B_j := k^{-1}(\{j\})$, $j\in\N$. We heavily rely on Lemma~\ref{lemma:finitedimlOpti}. Fix $j\in\N$ and define $\Psi_j(\theta,Q) := I_\theta(QT_\theta\P)\in\R_+$ on $B_j\times \mathcal M_\alpha(j,j)$. This is well defined in view of the discussion preceding Lemma~\ref{lemma:finitedimlOpti}. The same lemma shows that $\Psi_j$ is continuous in its second argument. Since $\mathcal M_\alpha(j,j)\subseteq\R^{j\times j}$ is non-empty and compact, the maximum is achieved. Clearly, $\theta\mapsto P_\theta(T_\theta^{-1}(\{l\}) = \int_\X \mathds 1_{T^{-1}(\{j\})}(x,\theta)p_\theta(x)\mu(dx)$ and $\theta\mapsto \int_{T_\theta^{-1}(\{l\})} s_\theta(x)p_\theta(x)\mu(dx) = \int_\X \mathds 1_{T^{-1}(\{l\})}(x,\theta)s_\theta(x)p_\theta(x)\mu(dx)$ are measurable, by Condition~\ref{C.measurable} and Fubini's theorem, for every $l\in[j]$. Thus, using Lemma~\ref{lemma:finitedimlOpti}.(\ref{lemma:finitedimlOpti:iii}) and referring to the definitions of the $j$-vector $p_\theta$ and the $j\times j$-matrix $\dot{p}_\theta$ of that lemma, we see that also $\theta\mapsto \Psi_j(\theta,Q)$ is measurable. In other words $\Psi_j:B_j\times \mathcal M_\alpha(j,j)\to\R$ is a Carath\'{e}odory function. Thus, by Theorem~18.19 of \citet{Aliprantis06} there exists a measurable function $M_j:B_j\to \mathcal M_\alpha(j,j)$ such that 
$$
I_\theta(M_j(\theta) T_{\theta}\P) =  \Psi_j(\theta,M_j(\theta))  = \max_{Q\in\mathcal M_\alpha(j,j)} \Psi_j(\theta,Q) = \max_{Q\in\mathcal M_\alpha(j,j)}I_\theta(QT_{\theta}\P), \quad \forall \theta\in B_j.
$$
Now, for $\theta\in B_j$, $C\subseteq[j]$ and $y\in[j]$, define $Q_\theta(C|y) := \sum_{i\in C} [M_j(\theta)]_{i,y}$ to be the corresponding Markov kernel in $\mathcal Q_\alpha([j]\to[j])$, which satisfies \eqref{eq:lemma:measurability} for all $\theta\in B_j$ and $(y,\theta)\mapsto Q_\theta(C|y)$ is measurable. Now, for $C\subseteq\N$, $y\in\N$ and $\theta\in\Theta$, define $Q(C|y,\theta) := \sum_{j\in\N} Q_\theta(C\cap[j]|y\land j)\mathds 1_{B_j}(\theta)$, which is the desired Markov kernel. By a similar argument as above, involving Lemma~\ref{lemma:finitedimlOpti} with $p_\theta := (P_{\theta_0} T_\theta^{-1}(\{j\}))_{j\in[k(\theta)]}$ and $\dot{p}_\theta := (\int_{T_\theta^{-1}(\{j\})} s_{\theta_0}(x)p_{\theta_0}(x)\mu(dx))_{j\in[k(\theta)]}$, we obtain also measurability of $\theta\to I_{\theta_0}(Q_\theta T_\theta \P)$.\hfill\qed

%======================================================================

\subsection{Proof of Lemma~\ref{lemma:Delta_k}}
For $Q\in\mathcal Q_\alpha(\X\to\mathcal Z)$, let $q:\mathcal Z\times \X \to [e^{-\alpha},e^\alpha]$ be as in Lemma~\ref{lemma:Qdensities}. In particular, $q(z|x)p_\theta(x)$ is a $\nu\otimes\mu$-density of $Q(dz|x)P_\theta(dx)$. Hence $q_\theta(z) := \int_\X q(z|x)p_\theta(x)\mu(dx)$ is a $\nu$-density of $QP_\theta$ and the conditional expectation of $s_\theta$,
$$
t_\theta(z) := \int_\X s_\theta(x)\frac{q(z|x)p_\theta(x)}{q_\theta(z)} \mu(dx) \in\R^p,
$$
is a score of the model $Q\P$ at $\theta$ (cf. Lemma~\ref{lemma:DQM}). Furthermore, for $j\in J$, notice that $p_j(x) := [\mu(B_j)]^{-1}\mathds 1_{B_j}(x)$ is a $\mu$-density of $P_j$ and that
\begin{align*}
q_k(z|j):= \begin{cases}
\int_\X q(z|x)p_j(x)\mu(dx), &j\in J,\\
1, &j\notin J,
\end{cases}
\end{align*}
is a $\nu$-density of $Q_k(\cdot|j)$. In particular, $e^{-\alpha}\le q_\theta(z),q_k(z|j)\le e^\alpha$ and $\|t_\theta(z)\|\le e^{2\alpha}\E_\theta[\|s_\theta\|]\le e^{2\alpha}\sqrt{\trace I_\theta(\P)}$. Finally, define $r_\theta(z) := \int_\X q_k(z|T_k(x))p_\theta(x)\mu(dx)\le e^\alpha$ and 
$$
u_\theta(z) := \int_{\X} s_\theta(x) \frac{q_k(z|T_k(x))p_\theta(x)}{r_\theta(z)}\mu(dx) \in\R^p.
$$
Since $\int_C\int_F q_k(z|T_k(x))p_\theta(x) \nu\otimes\mu(dz,dx) = \int_F Q_k(C|T_k(x))P_\theta(dx)$, we see that $q_k(z|T_k(x))p_\theta(x)$ is a $\nu\otimes\mu$-density of $Q_k(dz|T_k(x))P_\theta(dx)$ and thus, $r_\theta$ is a $\nu$-density of $Q_kT_kP_\theta$ and $u_\theta$ is a conditional expectation equal to the score of the model $Q_kT_k\P$ at $\theta$ (cf. Lemma~\ref{lemma:DQM}) satisfying $\|u_\theta(z)\|\le e^{2\alpha} \sqrt{\trace I_\theta(\P)}$. 

Therefore, we can compute and bound the difference of Fisher-Informations as
\begin{align*}
&\| I_\theta(Q\P) - I_\theta(Q_kT_k\P)\|_2 
=
\left\| \int_{\mathcal Z} \left[t_\theta(z)t_\theta(z)^Tq_\theta(z) - u_\theta(z)u_\theta(z)^Tr_\theta(z)\right]\nu(dz)\right\|_2\\
&\quad\le
 \int_{\mathcal Z} \|t_\theta(z)t_\theta(z)^T - u_\theta(z)u_\theta(z)^T\|_2 r_\theta(z)\nu(dz) + \int_{\mathcal Z} \|t_\theta(z)\|_2^2|q_\theta(z)-r_\theta(z)|\nu(dz).
\end{align*}
Since $B_j = T_k^{-1}(\{j\})$ and
\begin{align*}
r_\theta(z) &= \sum_{j=1}^k \int_{B_j} q_k(z|T_k(x))p_\theta(x)\mu(dx) = \sum_{j=1}^k \int_{B_j} q_k(z|j)P_\theta(dx)\\
& =
\sum_{j\in J} \int_\X q(z|x)p_j(x)\mu(dx) P_\theta(B_j) 	+ \sum_{j\notin J} P_\theta(B_j)\\
&=
\int_{K_J} q(z|x) \bar{p}_\theta(x)\mu(dx) + P_\theta(K_J^c),
\end{align*}
we have 
\begin{align*}
|q_\theta(z) - r_\theta(z)| 
&= 
\left|\int_{K_J} q(z|x) [p_\theta(x) - \bar{p}_\theta(x)] \mu(dx) + \int_{K_J^c} [q(z|x)-1]p_\theta(x)\mu(dx) \right|\\
&\le
e^\alpha \|(p_\theta - \bar{p}_\theta)\mathds 1_{K_J}\|_{L_1(\mu)} + (e^\alpha-1)P_\theta(K_J^c).
\end{align*}
Moreover, by analogous calculations we get
\begin{align*}
u_\theta(z)r_\theta(z) &= \int_\X q_k(z|T_k(x))\dot{p}_\theta(x)\mu(dx)\\
&=
\int_{K_J} q(z|x) \bar{\dot{p}}_\theta(x)\mu(dx) + \int_{K_J^c}\dot{p}_\theta(x)\mu(dx),
\end{align*}
and $t_\theta(z)r_\theta(z) = \int_\X q(z|x)\dot{p}_\theta(x)\mu(dx)\left(\frac{r_\theta(z)-q_\theta(z)}{q_\theta(z)}+1\right)$. Also, applying the Cauchy-Schwarz inequality, we have 
$$
\int_{F}\|\dot{p}_\theta(x)\|_2\mu(dx) = \int_{F}\|s_\theta(x)\sqrt{p_\theta(x)}\|\sqrt{p_\theta(x)}\mu(dx) \le \sqrt{\E_\theta[\|s_\theta\|_2^2]P_\theta(F)},
$$ 
which yields
\begin{align*}
&\|t_\theta(z)t_\theta(z)^T - u_\theta(z)u_\theta(z)^T\|_2 r_\theta(z) \\
&\quad= \|t_\theta(z)[t_\theta(z) - u_\theta(z)]^Tr_\theta(z) + [t_\theta(z)-u_\theta(z)]u_\theta(z)^Tr_\theta(z)\|_2\\
&\quad\le
2e^{2\alpha}\sqrt{\trace I_\theta(\P)} \|[t_\theta(z) - u_\theta(z)]r_\theta(z)\|_2\\
&\quad=
2e^{2\alpha}\sqrt{\trace I_\theta(\P)} 
\left\|
\int_\X  \dot{p}_\theta(x)
\left(q(z|x)\frac{r_\theta(z)-q_\theta(z)}{q_\theta(z)} + q(z|x)- q_k(z|T_k(x)) \right)\mu(dx) 
\right\|_2
\\
&\quad\le
2e^{2\alpha}\sqrt{\trace I_\theta(\P)} \left(e^\alpha\Big\|\|\dot{p}_\theta-\bar{\dot{p}}_\theta\|_2\mathds 1_{K_J}\Big\|_{L_1(\mu)} + (e^\alpha-1)\sqrt{\trace I_\theta(\P)P_\theta(K_J^c)}+\right.\\
&\quad\quad\left. + e^{2\alpha}\sqrt{\trace I_\theta(\P)}|q_\theta(z)-r_\theta(z)| \right).
\end{align*}
The claimed bound follows upon putting the pieces together, recalling that $\nu$ is a probability measure and noting that $1$, $\trace I_\theta(\P)$ and $\sqrt{\trace I_\theta(\P)}$ are all bounded by $1\lor \trace I_\theta(\P)$.\hfill\qed

%==============================================================================

\subsection{Proof of Lemma~\ref{lemma:finiteReduction}}
Theorem~2 of \citet{Kairouz16} shows that
$$
\sup_{Q\in \bigcup_{\ell\in\N}\mathcal Q_\alpha([k]\to[\ell])} I_\theta(QT_k\P) = \sup_{Q\in\mathcal Q_\alpha([k]\to[k])} I_\theta(QT_k\P),
$$
provided we can show that on the set $\bigcup_{\ell\in\N}\mathcal Q_\alpha([k]\to[\ell])$ the objective function $Q\mapsto  I_\theta(QT_k\P)$ has the sub-linear structure given in their Definition~1 and equation~(8). But this is verified by Lemma~\ref{lemma:finitedimlOpti}.

We are thus left with showing
\begin{equation}\label{eq:finiteZ}
\sup_{Q\in \mathcal Q_\alpha([k])} I_\theta(QT_k\P) = \sup_{Q\in \bigcup_{\ell\in\N}\mathcal Q_\alpha([k]\to[\ell])} I_\theta(QT_k\P).
\end{equation}
Only the inequality $\le$ requires a proof. Therefore, it suffices to show that for every $Q\in\mathcal Q_\alpha([k])$, there exists a sequence $Q_m\in\bigcup_{\ell\in\N}\mathcal Q_\alpha([k]\to[\ell])$ such that $I_\theta(Q_mT_k\P) \to I_\theta(QT_k\P)$, as $m\to\infty$. 
%For simplicity, we suppress the index $m$ whenever there is no risk of confusion.

Fix a $Q\in\mathcal Q_\alpha([k])$, that is, $Q\in\mathcal Q_\alpha([k]\to\mathcal Z)$, for some measurable space $(\mathcal Z,\mathcal G)$, and let $\nu(dz) := Q(dz|k)$. By differential privacy, for $C\in\mathcal G$, we have $Q(C|i)\le e^\alpha\nu(C)$, for each $i\in[k]$. In particular, $Q(\cdot|i)\ll\nu$ for all $i\in[k]$ and we write $q(z|i)$ for a corresponding $\nu$-density. Since 
$$
\int_C e^{-\alpha} \nu(dz) = e^{-\alpha} \nu(C) \le Q(C|i)=\int_C q(z|i) \nu(dz) \le e^\alpha \nu(C) = \int_C e^\alpha \nu(dz),
$$
we have $e^{-\alpha}\le q(z|i)\le e^\alpha$, $\nu$-almost surely and we can thus change $q(z|i)$ on a $\nu$-null set such that these bounds hold point-wise and $z\mapsto q(z|i)$ is still a density of $Q(\cdot|i)$ for every $i\in[k]$. Now, write $q_\theta(z) = \sum_{i=1}^k q(z|i)[P_\theta T_k^{-1}](\{i\}) = \int_\X q(z|T_k(x)) p_\theta(x)\mu(dx)$ for the marginal $\nu$-density of $Q[P_{\theta}T_k^{-1}] = [QT_k]P_\theta$ and note that $q(z|T_k(x)) p_\theta(x)$ is a $\nu\otimes\mu$-density of $Q(dz|T_k(x))P_\theta(dx)=[QT_k](dz|x)P_\theta(dx)$, because
$$
\int_{C\times F} q(z|T_k(x)) p_\theta(x) \nu\otimes\mu(dz,dx) = \int_F Q(C|T_k(x))P_\theta(dx)=
\int_F [QT_k](C|x)P_\theta(dx).
$$
Now, using Lemma~\ref{lemma:DQM}, write
$$
t_{\theta}(z) = \int_{\X} s_{\theta}(x)\frac{q(z|T_k(x))p_{\theta}(x)}{q_{\theta}(z)}\mu(dx),
$$
for the score in the model $QT_k\P$ at $\theta$. Next, we construct the sequence $Q_m$.

For $m\in\N$, $i\in[k]$ and $j\in[\lceil 2^{m}(e^\alpha-e^{-\alpha})\rceil]$, define 
$$
B_{i,j,m} := \left\{z\in\Z: e^{-\alpha} + 2^{-m}(j-1)\le q(z|i) < e^{-\alpha} + 2^{-m}j \right\},
$$
where $\lceil y\rceil$ denotes the smallest integer strictly larger than $y$. Notice that for fixed $i$ and $m$, $(B_{i,j,m})_{j}$ is a partition of $\Z$. Now, for $j_1,\dots, j_{k}\in[\lceil 2^{m}(e^\alpha-e^{-\alpha})\rceil]$ define $C_{j_1,\dots, j_{k}} := \bigcap_{i=1}^{k} B_{i,j_i,m}$. These are $\ell_m := \lceil 2^{m}(e^\alpha-e^{-\alpha})\rceil^{k}$ sets which, by construction, still partition all of $\Z$. We enumerate them, writing $C_j$, $j\in[\ell_m]$. Finally, note that if $z\in C_j$, then there exists $j_1,\dots, j_{k}$ such that for all $i\in[k]$, we have 
$$
e^{-\alpha} + 2^{-m}(j_i-1)\le q(z|i) < e^{-\alpha} + 2^{-m}j_i.
$$
Therefore, we conclude that
\begin{align}
&\sum_{\substack{j=1\\\nu(C_j)>0}}^{\ell_m} \int_{C_j}\max_{i\in[k]} \left|q(z|i)-\frac{Q(C_j|i)}{\nu(C_j)}\right|\nu(dz) \notag\\
&=
\sum_{\substack{j=1\\\nu(C_j)>0}}^{\ell_m} \int_{C_j}\max_{i\in[k]} \left|q(z|i)-\frac{\int_{C_j} q(z|i) \nu(dz)}{\nu(C_j)}\right|\nu(dz) \notag\\
&\le
\sum_{\substack{j=1\\\nu(C_j)>0}}^{\ell_m} \int_{C_j}\max_{i\in[k]} \left(2^{-m}j_i- 2^{-m}(j_i-1)\right)\nu(dz) = 2^{-m}. \label{eq:2-m}
\end{align}
Now let $\mathcal Z_m \subseteq [\ell_m]$ be the subset of $[\ell_m]$ such that $\nu_m(\{j\}) := \nu(C_j) > 0$, for all $j\in\mathcal Z_m$, and define the channel $Q_m$ on $({\mathcal Z_m}, 2^{{\mathcal Z_m}})$ by its ${\nu_m}$-densities ${q_m}(j|i) := Q(C_j|i)/{\nu_m}(\{j\})$, for $j\in{\mathcal Z_m}$. Clearly, $e^{-\alpha}\le {q_m}(j|i) \le e^\alpha$ and the channel ${Q}_m$ belongs to $\bigcup_{\ell\in\N}\mathcal Q_\alpha([k]\to[\ell])$, because, for any $x,x'\in\X_k=[k]$,
$$
{q_m}(j|x) =  Q(C_j|x)/{\nu_m}(\{j\}) \le e^\alpha\cdot Q(C_j|x')/{\nu_m}(\{j\}) = e^\alpha\cdot {q_m}(j|x').
$$
Now, as above, write ${t}_{m,\theta}(j) := \int_{\X} s_{\theta}(x) \frac{{q_m}(j|T_k(x))p_{\theta}(x)}{{q}_{m,\theta}(j)}\mu(dx)$, where $q_{m,\theta}(j) = \int_{\X}q_m(j|T_k(x))p_{\theta}(x)\mu(dx) \ge e^{-\alpha}$, and note that ${t}_{m,\theta}$ is the score function of the DQM model $Q_mT_k\P$, in view of Lemma~\ref{lemma:DQM}. Therefore,
\begin{align*}
| I_\theta(QT_k\P) - I_\theta({Q}_mT_k\P) | 
&=
\left|
\int_{\Z} t^2_\theta(z)q_\theta(z)\nu(dz) 
- 
\int_{{\Z_m}} {t}^2_{m,\theta}(z) {q}_{m,\theta}(z){\nu_m}(dz)
\right| \\
&=\left|
\sum_{j=1}^{\ell_m} \int_{C_j} t^2_\theta(z)q_\theta(z)\nu(dz)  
- 
\sum_{z\in{\Z_m}} {t}^2_{m,\theta}(z) {q}_{m,\theta}(z){\nu_m}(\{z\})
\right|\\
&=\left|
\sum_{j\in\mathcal Z_m}\int_{C_j} t^2_\theta(z)[QP_{k,\theta}](dz)  
- 
\sum_{j\in{\Z_m}} {t}^2_{m,\theta}(j) [QP_{k,\theta}](C_j)
\right|\\
&=\left|
\sum_{j\in{\Z_m}} \int_{C_j} \left[t^2_\theta(z) - {t}_{m,\theta}^2(j)\right]q_\theta(z)\nu(dz)  
\right|\\
&\le
\sum_{j\in{\Z_m}} \int_{C_j} \left|t_\theta(z) - {t}_{m,\theta}(j)\right| 
\left| t_\theta(z) + {t}_{m,\theta}(j) \right| q_\theta(z)\nu(dz)  \\
&\le
2e^{3\alpha} \E_{\theta}[|s_{\theta}|]\sum_{j\in{\Z_m}} \int_{C_j} \left|t_\theta(z) - {t}_{m,\theta}(j)\right| \nu(dz). 
\end{align*}
Now, for $j\in{\Z_m}$ and $z\in C_j$,
\begin{align*}
\left|t_\theta(z) - {t}_{m,\theta}(j)\right|
&\le
\int_{\X} |s_{\theta}(x)| \left| \frac{q(z|T_k(x))}{q_\theta(z)} - \frac{{q_m}(j|T_k(x))}{{q}_{m,\theta}(j)}\right|p_{\theta}(x)\mu(dx) \\
&\le
\E_{\theta}[|s_{\theta}|] \sup_{x\in\X}\left| \frac{q(z|T_k(x))}{q_\theta(z)} - \frac{{q_m}(j|T_k(x))}{q_\theta(z)} + \frac{{q_m}(j|T_k(x))}{q_\theta(z)} - \frac{{q_m}(j|T_k(x))}{{q}_{m,\theta}(j)}\right|\\
&\le 
\E_{\theta}[|s_{\theta}|]\max_{i\in[k]}
\left(e^\alpha |q(z|i)-{q_m}(j|i)| + e^{3\alpha} \left|q_\theta(z) - {q}_{m,\theta}(j)\right|\right)\\
&\le
2e^{3\alpha} \E_{\theta}[|s_{\theta}|] \max_{i\in[k]}|q(z|i)-{q_m}(j|i)|.
\end{align*}
Hence,
\begin{align*}
&| I_\theta(QT_k\P) - I_\theta({Q}_mT_k\P) | \\
&\le
4e^{6\alpha} I_\theta(\P) \sum_{j\in{\Z_m}} \int_{C_j} \max_{i\in[i]}|q(z|i)-{q_m}(j|i)|\nu(dz) \xrightarrow[m\to\infty]{} 0,
\end{align*}
in view of \eqref{eq:2-m}, which establishes \eqref{eq:finiteZ}.\hfill\qed

%==============================================================================

\subsection{Proof of Lemma~\ref{lemma:ContInTheta}}
Fix $Q\in\mathcal Q_\alpha(\X)$ and use Lemmas~\ref{lemma:DQM} and \ref{lemma:Qdensities} (yielding $\nu$-densities $e^{-\alpha}\le q(z|x)\le e^\alpha$), to write
\begin{align*}
t_\theta(z) := \E[s_\theta(X)| Z=z] &= \int_\X s_\theta(x) \frac{q(z|x) p_\theta(x)}{{q}_\theta(z)} \mu(dx)\\
&= \frac{1}{{q}_\theta(z)} \int_\X q(z|x) \dot{p}_\theta(x)\mu(dx)
\end{align*}
for the score of model $Q\P$ at $\theta$, where ${q}_\theta(z) := \int_\X q(z|x) p_\theta(x)\mu(dx) \ge e^{-\alpha}$.
Now
\begin{align*}
&\left\|I_\theta(Q \P) - I_{\theta'}(Q \P)\right\|_2
=
\left\| \int_{\mathcal Z} \left[t_\theta(z)t_\theta(z)^T q_\theta(z)
-
t_{\theta'}(z)t_{\theta'}(z)^T q_{\theta'}(z)\right] \nu(dz) \right\|_2 \\
&\quad\le
\int_{\mathcal Z} \|t_{\theta}(z)t_{\theta}(z)^T - t_{\theta'}(z)t_{\theta'}(z)\|_2{q}_\theta(z)\nu(dz)
+
\int_{\mathcal Z} \|t_{\theta'}(z)\|_2^2| {q}_\theta(z) - q_{\theta'}(z)|\nu(dz),
\end{align*}
and
$
\left|q_\theta(z) - {q}_{\theta'}(z)\right| \le e^\alpha \|{p}_\theta - {p}_{\theta'}\|_{L_1(\mu)}.
$
Hence, $\int_{\mathcal Z} \|t_{\theta'}(z)\|_2^2 | {q}_\theta(z) - q_{\theta'}(z)|\nu(dz) \le e^{2\alpha} \|{p}_\theta - {p}_{\theta'}\|_{L_1(\mu)} \trace[I_{\theta'}(Q\P)]\le e^{2\alpha} \|{p}_\theta - {p}_{\theta'}\|_{L_1(\mu)} \trace[I_{\theta'}(\P)]$. Moreover, 
\begin{align*}
&\|t_\theta(z)t_\theta(z)^T - t_{\theta'}(z)t_{\theta'}(z)^T\|_2{q}_\theta(z) = \\
&=\|(t_\theta(z) - t_{\theta'}(z))t_\theta(z)^T{q}_\theta(z) + t_{\theta'}(z)(t_\theta(z)- t_{\theta'}(z))^Tq_\theta(z)\|_2\\
&\le (\|t_\theta(z)\|_2 + \|t_{\theta'}(z)\|_2)\|t_\theta(z)q_\theta - t_{\theta'}(z)q_{\theta'}(z) + t_{\theta'}(z)q_{\theta'}(z) - t_{\theta'}(z)q_\theta\|_2
\\
&\le
(\|t_\theta(z)\|_2 + \|t_{\theta'}(z)\|_2)
\left[
e^\alpha \Big\|\|\dot{p}_\theta-\dot{p}_{\theta'}\|_2\Big\|_{L_1(\mu)} +  \|t_{\theta'}(z)\|_2e^\alpha\|{p}_\theta - {p}_{\theta'}\|_{L_1(\mu)}
\right].
\end{align*}
The integral with respect to $\nu$ of this upper bound can be further bounded by
$$
2e^{2\alpha}(\trace[I_{\theta}(\P)] \lor \trace[I_{\theta'}(\P)] \lor1) \left[ \Big\| \|\dot{p}_\theta-\dot{p}_{\theta'}\|_2\Big\|_{L_1(\mu)} + \|{p}_\theta - {p}_{\theta'}\|_{L_1(\mu)}
\right].
$$
The continuity of $\varphi$ is proved in Lemma~\ref{lemma:continuity}. \hfill\qed

\begin{lemma}\label{lemma:continuity}
If conditions \ref{C.DQM} and \ref{C.L2cont} are satisfied, the following functions are continuous:
\begin{enumerate}[i)]
\item $\theta\mapsto p_\theta : \Theta\to L_1(\mu)$ 
\item $\theta\mapsto \sqrt{p_\theta} : \Theta\to L_2(\mu)$ 
\item\label{lemma:continuity:trace} $\theta\mapsto \trace I_\theta(\P) : \Theta\to\R$
\item $\theta\mapsto \dot{p}_\theta : \Theta\to L_1(\mu,\|\cdot\|_2)$ 
\end{enumerate}
\end{lemma}
\begin{proof}
Notice that by the DQM property of Condition~\ref{C.DQM}, we have for $h:=\theta'-\theta$ and $\theta'\to\theta$,
\begin{align*}
\|\sqrt{p_{\theta'}}-\sqrt{p_{\theta}}\|_{L_2} &\le
\|\sqrt{p_{\theta+h}}-\sqrt{p_{\theta}}-\frac12 h^T s_{\theta}\sqrt{p_{\theta}}\|_{L_2}
+
\|\frac12 h^T s_{\theta}\sqrt{p_{\theta}}\|_{L_2}\\
&=
o(1) + \sqrt{h^T I_\theta(\P) h} = o(1).
\end{align*}
Using Cauchy-Schwarz, $\|p_{\theta'} - p_{\theta}\|_{L_1} = \|(\sqrt{p_{\theta'}}-\sqrt{p_{\theta}})(\sqrt{p_{\theta'}}+\sqrt{p_{\theta}})\|_{L_1}\le
\|\sqrt{p_{\theta'}}-\sqrt{p_{\theta}}\|_{L_2} \|\sqrt{p_{\theta'}}+\sqrt{p_{\theta}})\|_{L_2} \le
2\|\sqrt{p_{\theta'}}-\sqrt{p_{\theta}}\|_{L_2}$.
Moreover, Condition~\ref{C.L2cont} and the reverse triangle inequality imply
$$
|\sqrt{\trace I_\theta(\P)} - \sqrt{\trace I_{\theta'}(\P)}| = | \|s_{\theta}\sqrt{p_{\theta}}\|_{L_2} - \|s_{\theta'}\sqrt{p_{\theta'}}\|_{L_2}|
\le
\|s_{\theta}\sqrt{p_{\theta}}-s_{\theta'}\sqrt{p_{\theta'}}\|_{L_2} \to 0.
$$
Hence, $\theta\mapsto \trace I_\theta(\P) = \sqrt{\trace[I_\theta(\P)]}^2$ is continuous on $\Theta$. Finally, using Cauchy-Schwarz again, we have
\begin{align*}
\Big\| \|\dot{p}_{\theta'} - \dot{p}_{\theta}\|_2\Big\|_{L_1} &=
\Big\| \|s_{\theta'}\sqrt{p_{\theta'}}\sqrt{p_{\theta'}} - s_{\theta}\sqrt{p_{\theta}}\sqrt{p_{\theta}}\|_2\Big\|_{L_1}\\
&=
\Big\| \|(s_{\theta'}\sqrt{p_{\theta'}} - s_{\theta}\sqrt{p_{\theta}})\sqrt{p_{\theta'}} + s_{\theta}\sqrt{p_{\theta}}(\sqrt{p_{\theta'}}-\sqrt{p_{\theta}})\|_2\Big\|_{L_1}\\
&\le
\|s_{\theta'}\sqrt{p_{\theta'}} - s_{\theta}\sqrt{p_{\theta}}\|_{L_2} + \sqrt{\trace[I_\theta(\P)]}\|\sqrt{p_{\theta'}}-\sqrt{p_{\theta}}\|_{L_2} = o(1),
\end{align*}
where convergence follows by combining the facts we had already established before.
\end{proof}

%========================================================================

\section{Proofs of Section~\ref{sec:appl}}

\subsection{Proof of Lemma~\ref{lemma:location}}

Condition~\ref{C.measurable} is obviously satisfied with $s_\theta(x) := \frac{\dot{p}_\theta(x)}{p_\theta(x)} = \frac{-p'(x-\theta)}{p(x-\theta)}$. Moreover, $s_\theta(x)\sqrt{p_\theta(x)} = s_0(x-\theta)\sqrt{p(x-\theta)}$ is continuous in $\theta$ and $\|s_\theta\sqrt{p_\theta}\|_{L_2} = \|s_0\sqrt{p_0}\|_{L_2} = \|p'/\sqrt{p}\|_{L_2}< \infty$. Thus, by Proposition~2.29 of \citet{vanderVaart07}, we have $\|s_\theta\sqrt{p_\theta} - s_{\theta'}\sqrt{p_{\theta'}}\|_{L_2} \to 0$ as $\theta'\to\theta$, i.e., Condition~\ref{C.L2cont} holds. In particular, $\theta\mapsto \|s_\theta\sqrt{p_\theta}\|^2_{L_2} = I_\theta(\P)$ is continuous. Condition~\ref{C.DQM} is now a consequence of Lemma~7.6 of \citet{vanderVaart07}. That $\theta\mapsto p_\theta(x)=p(x-\theta)$ is three times continuously differentiable for every $x\in\R$ is also obvious from Assumption~\ref{A.diffable} and it is equally trivial that $\theta\mapsto \|\dot{p}_\theta\|_1 = \|p'\|_1$ is continuous and finite, and the same holds for the second and third derivatives. Finally $\theta\mapsto \dddot{p}_\theta(x) = -p^{(3)}(x-\theta)$ is continuous in $\theta$ for every $x\in\R$ and $\|\dddot{p}_\theta\|_1 = \|\dddot{p}_0\|_1 = \|p^{(3)}\|_1<\infty$. Thus, another application of Proposition~2.29 in \citet{vanderVaart07} yields the desired $L_1$-continuity. That $T_m(x-\theta)$ is a consistent quantizer easily follows from shift invariance of Lebesgue measure. Finally, for $\theta\in\R$ notice that the quantities $(P_\theta T_{m,\theta}^{-1}(\{j\}))_{j\in[k_m]} = (P_0 T_{m,0}^{-1}(\{j\}))_{j\in[k_m]}$ and $(\int_{T_{m,\theta}^{-1}(\{j\})} s_\theta(x)p_\theta(x)\lambda(dx))_{j\in[k_m]} = (\int_{T_{m,0}^{-1}(\{j\})} s_0(x)p_0(x)\lambda(dx))_{j\in[k_m]}$ used in Lemma~\ref{lemma:finitedimlOpti} to characterize the Fisher-Information $I_\theta(QT_{m,\theta}\P)$ do not depend on $\theta$. \hfill\qed

\subsection{Proof of Lemma~\ref{lemma:gauss}}

For simplicity we suppress the dependence of $k_m$ on $m$. Notice that $J:= \{l\in[k]: 0<\lambda(B_{l})<\infty\} = \{2,3,\dots, k-1\}$ and $K := \bigcup_{j\in J}B_j = (\Phi^{-1}(\frac{1}{k}), \Phi^{-1}(1-\frac{1}{k})]$. Now write $x_{j} = \Phi^{-1}\left(\frac{j}{k}\right)$, fix $a>0$ and choose $j_{a,k}\in\N$ such that $x_{j_{a,k}-1} \le -a<x_{j_{a,k}}$, for all $k\in2\N$. Note that $j_{a,k}\to \infty$ as $k\to\infty$. Now use symmetry of $p$ and $\bar{p}$ to get
$$
 \int_{-\infty}^{-a} |\bar{p}(x)|dx =\int_{a}^\infty |\bar{p}(x)|dx\le \int_{a}^\infty p(x)dx = \int_{-\infty}^{-a} p(x)dx = \Phi(-a).
$$
Thus, $\left\| p-\bar{p}\right\|_{L_1} \le \int_{-a}^a |p(x) - \bar{p}(x)|dx + 4\Phi(-a)$ and, for sufficiently large $k$ such that $j_{a,k}\ge3$ and using symmetry and unimodality of $p$, we have
\begin{align}
\int_{-a}^a &|p(x) - \bar{p}(x)|dx
=\sum_{j\in J} \int_{B_j\cap[-a,a]} |p(x) - \bar{p}(x)| dx \notag\\
&\le
\sum_{j=j_{a,k}-1}^{k-j_{a,k}} \int_{B_j} \left|p(x) - \int_{B_j} p(y) \frac{dy}{\lambda(B_j)}\right| dx \notag\\
&\le
\sum_{j=j_{a,k}-1}^{k-j_{a,k}} \frac{1}{\lambda(B_j)} \int_{B_j}\int_{B_j} \left|p(x) -  p(y)\right| dx dy  \notag\\
&\le
\sum_{j=j_{a,k}-1}^{k-j_{a,k}} \left| p(x_{j}) - p(x_{j-1})\right| \left[x_{j}-x_{j-1} \right]\notag\\
&=
2\left(\sum_{j=j_{a,k}-1}^{k/2}  p(x_{j}) \left[x_{j}-x_{j-1} \right] - \sum_{j=j_{a,k}-1}^{k/2} p(x_{j-1})\left[x_{j}-x_{j-1} \right]\right).\label{eq:Riemann}
\end{align}
Notice that 
\begin{align*}
&p(x_{j_{a,k}})\left[x_{j_{a,k}}-(-a) \right] + \sum_{j=j_{a,k}+1}^{k/2}  p(x_{j}) \left[x_{j}-x_{j-1} \right]
\quad \text{and}\quad\\
&p(-a)\left[x_{j_{a,k}}-(-a) \right] + \sum_{j=j_{a,k}+1}^{k/2} p(x_{j-1})\left[x_{j}-x_{j-1} \right]
\end{align*}
are Riemann upper and lower sums of the continuous function $p$ on the domain $[-a,0]$ where the partition satisfies $\max\{x_{j}-x_{j-1}: j=j_{a,k},\dots, k/2\} \to 0$ and $x_{j_{a,k}}-(-a) \le x_{j_{a,k}}-x_{j_{a,k}-1}\to 0$ as $k\to\infty$. In particular, both Riemann sums 
converge to $\int_{-a}^0p(x)dx$ as $k\to\infty$ and thus, using boundedness of $p$ to control the remainder, \eqref{eq:Riemann} converges to $0$. Hence, $\limsup_{k\to\infty} \left\| p-\bar{p}\right\|_{L_1} \le 4\Phi(-a)$. Since $a>0$ was arbitrary, the limit must be zero. 

An analogous argument can be used to show that $\left\| p'-\bar{p'}\right\|_{L_1}\to0$ as $k\to\infty$. Finally, $\|p\mathds 1_{K^c}\|_{L_1} = \frac{2}{k}\to0$ as $k\to\infty$.
\hfill\qed

\subsection{Proof of Lemma~\ref{lemma:scale}}

Condition~\ref{C.measurable} is obviously satisfied with 
$$
s_\theta(x) := \frac{\dot{p}_\theta(x)}{p_\theta(x)} = -\frac{1}{2\theta}\left(1+ \theta^{-1/2}x\frac{p'(\theta^{-1/2}x)}{p(\theta^{-1/2}x)}\right).
$$ 
From Assumption~\ref{A.diffableScale} we also see that $s_1\sqrt{p} \in L_2(\lambda)$. For Condition~\ref{C.L2cont} we have to show that 
$$
\|s_\theta\sqrt{p_\theta} - s_{\theta'}\sqrt{p_{\theta'}}\|^2_{L_2} 
= \frac{1}{\theta^2}\int_{\X} \left( s_1(u)\sqrt{p(u)} - \left( \frac{\theta}{\theta'}\right)^{\frac54} s_1\left( u\sqrt{\frac{\theta}{\theta'}} \right)\sqrt{p\left( u\sqrt{ \frac{\theta}{\theta'}}\right)}\right)^2\,du
$$
converges to zero as $\theta'\to\theta$. But clearly, $\eta_t(u):= t^{5/2} s_1(tu)\sqrt{p(tu)}\to s_1(u)\sqrt{p(u)}$, as $t\to1$, for every $u\in\X$ and $\|\eta_t\|_{L_2} = t^{3/2}\|s_1\sqrt{p}\|_{L_2}\to\|s_1\sqrt{p}\|_{L_2} <\infty$ as $t\to1$. Thus, Condition~\ref{C.L2cont} holds in view of Proposition~2.29 of \citet{vanderVaart07}. $L_2$-continuity also implies continuity of the norm $\theta\mapsto \|s_\theta\sqrt{p_\theta}\|_{L_2} = \sqrt{I_\theta(\P_{scale})}$. Condition~\ref{C.DQM} is now a consequence of Lemma~7.6 of \citet{vanderVaart07}. That $\theta\mapsto p_\theta(x)$ is three times continuously differentiable for every $x\in\R$ is also obvious from Assumption~\ref{A.diffableScale} and it is easy to see that these derivatives are linear combinations of scaled versions of the functions $x\mapsto x^jp^{(j)}(x)$, $j=0,1,2,3$, where the coefficients are rational functions in $\theta>0$. Thus, by analogous arguments as above, involving Proposition~2.29 of \citet{vanderVaart07}, we conclude that $\theta \mapsto \dot{p}_\theta$, $\theta \mapsto \ddot{p}_\theta$ and $\theta \mapsto \dddot{p}_\theta$ are continuous as functions from $\theta$ to $L_1(\lambda)$.
To see that $T_{m,\theta}(x) := T_m(\theta^{-1/2}x)$ is a consistent quantizer of $\P_{scale}$, note that $B_j(\theta) := T_{m,\theta}^{-1}(\{j\}) = \sqrt{\theta} B_j$, $J(\theta):= \{l\in[k_m]: 0<\lambda(B_{l}(\theta))<\infty\} = \{l\in[k_m]: 0<\lambda(B_{l})<\infty\} = J$, $K_{J}(\theta) := \bigcup_{j\in J}B_{j}(\theta) = \sqrt{\theta} K$, and thus
\begin{align*}
\bar{p}_\theta (x) &:= \sum_{j\in J(\theta)} \frac{\int_{B_j(\theta)}p_\theta(y)\lambda(dy)}{\lambda(B_j(\theta))} \mathds 1_{B_j(\theta)}(x)
=
\sum_{j\in J} \frac{\int_{\sqrt{\theta} B_j}\theta^{-1/2}p(\theta^{-1/2}y)\lambda(dy)}{\sqrt{\theta} \lambda(B_j)} \mathds 1_{B_j}(\theta^{-1/2}x) \\
&= \theta^{-1/2}\bar{p}_1(\theta^{-1/2}x).
\end{align*}
Similarly, we also have $\dot{p}_\theta(x) = \theta^{-3/2}\dot{p}_1(\theta^{-1/2}x)$ and $\bar{\dot{p}}_\theta(x) = \theta^{-3/2}\bar{\dot{p}}_1(\theta^{-1/2}x)$. Now the quantities $\|(p_\theta - \bar{p}_\theta)\mathds 1_{K(\theta)}\|_{L_1} = \|(p_1 - \bar{p}_1)\mathds 1_{K(1)}\|_{L_1}$ and $P_\theta(K(\theta)^c) = P_1(K(1)^c)$ in \eqref{eq:T_kErr} of Definition~\ref{def:quantizer} are seen to be $\theta$-invariant. Furthermore, $\|(\dot{p}_\theta - \bar{\dot{p}}_\theta)\mathds 1_{K(\theta)}\|_{L_1} = \frac{1}{\theta}\|(\dot{p}_1 - \bar{\dot{p}}_1)\mathds 1_{K(1)}\|_{L_1}$, which shows that $\Delta_m(\theta_m)\to0$ as $m\to\infty$, for every sequence $\theta_m\to\theta\in\Theta=(0,\infty)$. Finally, for $\theta\in\R$ consider the quantities $P_\theta T_{m,\theta}^{-1}(\{j\}) = \int_{\sqrt{\theta} B_j} \theta^{-1/2}p(\theta^{-1/2}x)dx = P_1 T_{m,1}^{-1}(\{j\}))$ and $\int_{T_{m,\theta}^{-1}(\{j\})} s_\theta(x)p_\theta(x)dx = \int_{\sqrt{\theta} B_j} \theta^{-3/2}\dot{p}_1(\theta^{-1/2}x)dx = \frac{1}{\theta}\int_{T_{m,1}^{-1}(\{j\})} s_1(x)p_1(x)dx$ used in Lemma~\ref{lemma:finitedimlOpti} to characterize the Fisher-Information $I_\theta(QT_{m,\theta}\P)$. Thus, we have $I_\theta(QT_{m,\theta}\P) = \frac{1}{\theta^2}I_1(QT_{m,1}\P)$. \hfill\qed

%========================================================================
%===============================================================

\section{Technical lemmas}
\label{sec:app:technical}
\begin{lemma}\label{lemma:BinomialDiff}
Let $\P=(P_\theta)_{\theta\in\Theta}$ be a statistical model with open parameter space $\Theta\subseteq\R$ and a finite sample space $\X$. Write $p_\theta(x) = P_\theta(\{x\})$ for the corresponding probability mass functions. If $\theta\mapsto p_\theta(x)$ is twice differentiable on $\Theta$ for each $x\in\X$ and $Q\in\mathcal Q_\alpha(\X)$ is an arbitrary $\alpha$-private channel on $\X$, then $\theta\mapsto I_\theta(Q\P)$ is differentiable on $\Theta$.
\end{lemma}

\begin{proof}
For given $Q\in\mathcal Q_\alpha(\X)$, let $(\Z,\mathcal G)$ denote the measurable space on which the private data are generated, that is, for each $x\in\X$, $Q(\cdot|x)$ is a probability measure on $(\Z,\mathcal G)$. For simplicity, write $I_\theta = I_\theta(Q\P)$. We will show that for every $\theta\in\Theta$ the quotient $\frac{I_{\theta+h}-I_\theta}{h}$ has a limit as $h\to0$. First, let $q(z|x)$ be as in Lemma~\ref{lemma:Qdensities}. In particular, $e^{-\alpha}\le q(z|x) \le e^\alpha$, for all $z\in\Z$ and all $x\in\X$, $(z,x)\mapsto q(z|x)p_\theta(x)$ is a $\nu\otimes\mu$-density of $Q(dz|x)P_\theta(dx)$ and $\nu$ is a probability measure on $(\Z,\mathcal G)$. The model $Q\P$ can thus be described by its $\nu$ densities $q_\theta(z) := \sum_{x\in\X} q(z|x)p_\theta(x)\ge e^{-\alpha}$. Clearly, these are twice differentiable with $\dot{q}_\theta(z) = \sum_{x\in\X} q(z|x)\dot{p}_\theta(x)$ and $\ddot{q}_\theta(z) = \sum_{x\in\X} q(z|x)\ddot{p}_\theta(x)$ and we have $I_\theta = \int_\Z \left( \frac{\dot{q}_\theta(z)}{q_\theta(z)}\right)^2 q_\theta(z)\nu(dz)$. Now consider the quotient 
\begin{align}
\frac{I_{\theta+h} - I_\theta}{h} &= \int_\Z \frac{1}{h}\left[\frac{\dot{q}_{\theta+h}^2(z)}{q_{\theta+h}(z)} - \frac{\dot{q}_{\theta}^2(z)}{q_{\theta}(z)}  \right]\nu(dz)\notag\\
&=
\int_\Z \frac{\dot{q}_{\theta+h}^2(z)-\dot{q}_\theta^2(z)}{h}\frac{1}{q_{\theta}(z)}\frac{q_{\theta}(z)}{q_{\theta+h}(z)} + \dot{q}_\theta^2(z) \frac{q_{\theta}(z)-q_{\theta+h}(z)}{hq_{\theta+h}(z)q_{\theta}(z)}  \nu(dz).\label{eq:DCT1}
\end{align} 
For every $z\in\Z$, the integrand converges to
$$
\frac{2\dot{q}_\theta(z)\ddot{q}_\theta(z)}{q_\theta(z)} - \frac{\dot{q}_\theta^2(z)\dot{q}_\theta(z)}{q_\theta^2(z)},
$$
as $h\to0$, which is clearly $\nu$-integrable in view of boundedness of $q_\theta$ and its derivatives. To see boundedness of the integrand in \eqref{eq:DCT1} by a $\nu$-integrable function independent of $h$, first note that $e^{-\alpha}\le q_\theta(z)\le e^\alpha$. Next, notice that with $D(0):= \sum_{x\in\X} |\dot{p}_\theta(x)|$, the function $h\mapsto D(h) := \sum_{x\in\X} \frac{1}{h}\left| p_{\theta+h}(x) - p_\theta(x)\right|$ is continuous on a compact neighborhood of zero, say $U$, and thus attains its maximum. Therefore, 
$$
\sup_{h\in U}\left|\frac{q_{\theta+h}(z)-q_\theta(z)}{h}\right| \le e^\alpha \sup_{h\in U} D(h) < \infty.
$$
Similarly, with $\dot{D}(h) := \sum_{x\in\X} \frac{1}{h}\left| \dot{p}_{\theta+h}(x) - \dot{p}_\theta(x)\right|$,
\begin{align*}
&\sup_{h\in U}\left|\frac{\dot{q}_{\theta+h}^2(z)-\dot{q}_\theta^2(z)}{h}\right|
= \sup_{h\in U}\left|\frac{\left[\dot{q}_{\theta+h}(z)-\dot{q}_\theta(z)+\dot{q}_\theta(z)\right]^2-\dot{q}_\theta^2(z)}{h}\right|
\\
&\quad=
\sup_{h\in U}\left|
\frac{\dot{q}_{\theta+h}(z) - \dot{q}_{\theta}(z)}{h} \left[ \dot{q}_{\theta+h}(z) - \dot{q}_{\theta}(z) \right] + 2 \dot{q}_{\theta}(z) \frac{\dot{q}_{\theta+h}(z) - \dot{q}_{\theta}(z)}{h}
\right| \\
&\quad\le
\sup_{h\in U}\left( e^{2\alpha} |h| \dot{D}(h)^2 + 2e^\alpha\dot{D}(h)\right) < \infty.
\end{align*}
The proof is now finished by the dominated convergence theorem, recalling the fact that $\nu$ is a probability measure.
\end{proof}

\begin{lemma}\label{lemma:Qdensities}
Let $\alpha\in(0,\infty)$ and $Q\in\mathcal Q_\alpha(\X)$. If $x_0\in\X$ and $\nu(dz) := Q(dz|x_0)$, then $Q(\cdot|x)\ll \nu$ for all $x\in\X$. Moreover, if $\mathcal P\subseteq\mathfrak P(X)$ is a set of probability measures on sample space $(\X, \mathcal F)$ dominated by the $\sigma$-finite measure $\mu$, then there exists a measurable function $q:\mathcal Z\times\X\to[e^{-\alpha},e^{\alpha}]$, such that for every $P\in\P$, $(z,x)\mapsto q(z,x)\frac{dP}{d\mu}(x)$ is a $\nu\otimes \mu$-density of the probability measure $Q(dz|x)P(dx)$. Thus, we use the notation $q(z|x) := q(z,x)$ suggestive of conditioning.
\end{lemma}

\begin{proof}
Let $(\mathcal Z,\mathcal G)$ denote the measurable space in which $Q$ generates its outputs. For $x\in\X$ and $A\in\mathcal G$, we have $Q(A|x) \le e^{\alpha} Q(A|x_0)$ and hence $Q(\cdot|x)\ll\nu$. For the second claim, define the measures $R(dz,dx):= Q(dz|x)\mu(dx)$ and $R_0 := \nu\otimes\mu$. For $C\in\mathcal G\otimes\mathcal F$, we have $R(C) = \int_{\X}Q(C_x|x)\mu(dx)\le e^{\alpha}\int_{\X}Q(C_x|x_0)\mu(dx) = e^{\alpha}\int_{\X}\nu(C_x)\mu(dx) = e^{\alpha}R_0(C)$ and thus, $R\ll R_0$. Furthermore, since $\mu$ is $\sigma$-finite, there exists a measurable partition $(B_j)_{j\in\N}$ of $\mathcal X$ such that $\mu(B_j)<\infty$ for all $j$. Hence, taking the partition $C_j = \mathcal Z\times B_j$, we see that $R$ and $R_0$ are $\sigma$-finite measures as well. Now take $\tilde{q} := \frac{dR}{dR_0}:\mathcal Z\times\X\to\R_+$, which is clearly measurable, and notice that for arbitrary $C\in\mathcal G\otimes\mathcal F$,
\begin{equation}\label{eq:lemma:Qdensities}
\int_{C}\tilde{q}(z,x)\frac{dP}{d\mu}(x) R_0(dz,dx) =
\int_{C} \frac{dR}{dR_0} \frac{dP}{d\mu} dR_0 
= \int_{C} \frac{dP}{d\mu} dR 
= R_P(C),
\end{equation}
where $R_P(dz,dx) := Q(dz|x)P(dx)$.
Hence, $\tilde{q}(z,x)\frac{dP}{d\mu}(x)$ is an $R_0$-density of $R_P$. Furthermore, for arbitrary $C\in\mathcal G\otimes\mathcal F$, we have
\begin{align*}
&\int_{C} e^{-\alpha} dR_0 = \int_\X e^{-\alpha}Q(C_x|x_0)\mu(dx) 
\le 
\int_\X Q(C_x|x)\mu(dx) = R(C) \\
&\quad= \int_{C} \tilde{q} dR_0 = \int_\X Q(C_x|x)\mu(dx)\le \int_\X e^{\alpha}Q(C_x|x_0)\mu(dx)
=
\int_{C} e^{\alpha} dR_0.
\end{align*}
Thus we have $e^{-\alpha}\le \tilde{q}(z,x)\le e^{\alpha}$, for $R_0$-almost all $(z,x)\in\mathcal Z\times\X$. Now set $\tilde{q}$ equal to $1$ on the corresponding $R_0$-null set to obtain $q$, which does not have an effect on the validity of \eqref{eq:lemma:Qdensities} and also the desired boundedness property is satisfied.
\end{proof}

The following result is an extension of Lemma~\ref{lemma:Qdensities} to the sequentially interactive case.

\begin{lemma}\label{lemma:QdensitiesSI}
Let $\alpha\in(0,\infty)$, $n\in\N$ and $Q$ be $\alpha$-sequentially interactive from $(\X^n, \mathcal F^n)$ to $(\mathcal Z^{(n)}, \mathcal G^{(n)})$ as in Section~\ref{sec:DefPriv} and suppose that the $\sigma$-fields $\mathcal F, \mathcal G_1,\dots, \mathcal G_n$ are countably generated. Fix $x^*\in\X^n$, $i\in[n]$ and define 
\begin{align*}
Q^{(i)}(dz_{1:i}|x_{1:i}) &:= Q_{z_{1:i-1}}(dz_i|x_i) 
Q_{z_{1:i-2}}(dz_{i-1} |x_{i-1})\cdots Q_\varnothing(dz_1|x_1) \\
\nu^{(i)}(dz_{1:i}) &:= Q_{z_{1:i-1}}(dz_i|x_i^*) 
Q_{z_{1:i-2}}(dz_{i-1} |x_{i-1}^*)\cdots Q_\varnothing(dz_1|x_1^*), \quad\text{and}\\
\nu_{z_{1:i-1}}(dz) &:= Q_{z_{1:i-1}}(dz_i|x_i^*), 
\end{align*}
where we adopt the convention that whenever a sub- or superscript is equal to $0$ the corresponding quantity is omitted.
If the collection $\mathcal P\subseteq\mathfrak P(\X,\mathcal F)$ of probability measures is dominated by the $\sigma$-finite measure $\mu$ with densities $p_P = \frac{dP}{d\mu}$, then the following hold true for each $i\in[n]$:
\begin{enumerate}[i)]
\item\label{lemma:QdensitiesSI:a} For all $P\in\mathcal P$ and all $A\in\mathcal G^{(i)}$ we have $e^{-i\alpha}Q^{(i)}P^i(A)\le \nu^{(i)}(A) \le e^{i\alpha}Q^{(i)}P^i(A)$ and we write $r_{i,P} := \frac{dQ^{(i)}P^i}{d\nu^{(i)}}$ for a corresponding density.
\item\label{lemma:QdensitiesSI:b}  For every $i\in[n]$ there exists a measurable function $q_i:\mathcal Z_i\times\X\times \mathcal Z^{(i-1)}\to[0,\infty)$, such that for every $P\in\P$ and for every $z_{1:i-1}\in\mathcal Z^{(i-1)}$,
$$
(z_{i},x_i)\mapsto q_i(z_i| x_i, z_{1:i-1}) p_P(x_i)
$$ 
is a $\nu_{z_{1:i-1}}\otimes \mu$-density of the joint distribution $Q_{z_{1:i-1}}(dz_i|x_i)P(dx_i)$ of $Z_i, X_i$ given $Z_{1:i-1}=z_{1:i-1}$ and $q_i(z|x,z_{1:i-1}) \in[e^{-\alpha},e^\alpha]$ for $\nu_{z_{1:i-1}}\otimes\mu$-almost all $(z,x)$.
\item\label{lemma:QdensitiesSI:c} For every $P\in\P$ and for $q_{i,P}(z_i|z_1,\dots,z_{i-1}) := \int_\X q_i(z_i|x_i,z_{1:i-1})P(dx_i)$ we have $q_{i,P}(z_i|z_1,\dots,z_{i-1})r_{i-1,P}(z_{1:i-1}) = r_{i,P}(z_{1:i})$, $\nu^{(i)}$-almost surely. 
\end{enumerate}
\end{lemma}

\begin{proof}
For $x\in\X^i$ and $A\in\mathcal G^{(i)}$, by $\alpha$-differential privacy \eqref{eq:alphaPriv} we have $Q^{(i)}(A|x) \le e^{i\alpha} Q^{(i)}(A|x^*) = e^{i\alpha}\nu^{(i)}(A)$ and hence 
$$
[Q^{(i)}P^i](A) = \int_{\X^i} Q^{(i)}(A|x) P^i(dx) \le e^{i\alpha} \nu^{(i)}(A).
$$
Exchanging the roles of $x$ and $x^*$ proves the claim in \ref{lemma:QdensitiesSI:a}).

For Part~\ref{lemma:QdensitiesSI:b}), first note that $\mu$ is dominated by a probability measure $\bar{\mu}$, say, and write $m = \frac{d\mu}{d\bar{\mu}}$. Now fix $z_{1:i-1}\in\mathcal Z^{(i-1)}$ and define the Markov kernels $\bar{R}(dz_i,dx_i|z_{1:i-1}):= Q_{z_{1:i-1}}(dz_i|x_i)\bar{\mu}(dx_i)$, $\bar{S}(dz_i,dx_i|z_{1:i-1}):= \nu_{z_{1:i-1}}(dz_i)\otimes\bar{\mu}(dx_i)$, $R(dz_i,dx_i|z_{1:i-1}):= Q_{z_{1:i-1}}(dz_i|x_i)\mu(dx_i)$ and $S(dz_i,dx_i|z_{1:i-1}):= \nu_{z_{1:i-1}}(dz_i)\otimes\mu(dx_i)$. For $C\in\mathcal G_i\otimes\mathcal F$, we have 
\begin{align*}
\bar{R}(C|z_{1:i-1}) &= \int_{\X} Q_{z_{1:i-1}}(C_{x}|x)\bar{\mu}(dx)
\le e^{\alpha} \int_\X Q_{z_{1:i-1}}(C_{x}|x^*)\bar{\mu}(dx) = e^{\alpha}\bar{S}(C|z_{1:i-1})
\end{align*} 
and thus, $\bar{R}_{z_{1:i-1}}\ll \bar{S}_{z_{1:i-1}}$. By Theorem~58 in \citet[][page~52]{Dellacherie82} there exists measurable $\bar{q}_i :\mathcal Z_i\times\X\times\mathcal Z^{(i-1)}\to[0,\infty)$ such that for all $z_{1:i-1}$, $(z,x)\mapsto\bar{q}_i(z|x,z_{1:i-1})$ is a $\bar{S}_{z_{1:i-1}}$-density of $\bar{R}_{z_{1:i-1}}$, which is clearly independent of $P$. Thus, $(z,x)\mapsto q_i(z|x,z_{1:i-1}) := \frac{\bar{q}_i(z|x,z_{1:i-1})}{m(x)}$ is a ${S}_{z_{1:i-1}}$-density of ${R}_{z_{1:i-1}}$.
Notice that for arbitrary $C\in\mathcal G_i\otimes\mathcal F$,
\begin{align*}%\label{eq:lemma:QdensitiesSI}
&\int_{C}q_i(z|x,z_{1:i-1}) \cdot p_P(x)\; S_{z_{1:i-1}}(dz,dx) =
\int_{C} \frac{d{R}_{z_{1:i-1}}}{d{S}_{z_{1:i-1}}} \cdot p_P\; d{S}_{z_{1:i-1}} \notag 
= \int_{C} p_P\; d{R}_{z_{1:i-1}} \notag \\
&\quad= \int_{\X} Q_{z_{1:i-1}}(C_{x}|x) p_P(x)\mu(dx)
= \int_{\X} Q_{z_{1:i-1}}(C_{x}|x) P(dx).
\end{align*}
Hence, $q_i \cdot p_P$ is an $S_{z_{1:i-1}}$-density of $Q_{z_{1:i-1}}(dz|x) P(dx)$. 
Furthermore, still for fixed $z_{1:i-1}$ and for arbitrary $C\in\mathcal G_i\otimes\mathcal F$, we have
\begin{align*}
&\int_{C} e^{-\alpha} dS_{z_{1:i-1}} = \int_\X e^{-\alpha}\nu_{z_{1:i-1}}(C_x)\mu(dx) \\
&\le
\int_\X Q_{z_{1:i-1}}(C_x|x) \mu(dx) = R_{z_{1:i-1}}(C) 
= \int_{C} q_i dS_{z_{1:i-1}} \\
&\le \int_\X e^{\alpha} \nu_{z_{1:i-1}}(C_x) \mu(dx) 
= \int_{C} e^{\alpha} dS_{z_{1:i-1}}.
\end{align*}
We conclude that $e^{-\alpha}\le {q}_i\le e^{\alpha}$, for $S_{z_{1:i-1}}$-almost all $(z,x)\in\mathcal Z_i\times\X$. 

For \ref{lemma:QdensitiesSI:c}) simply integrate the density in part~\ref{lemma:QdensitiesSI:b}) with respect to $\mu$.
\end{proof}

\begin{lemma}\label{lemma:Q_kPriv}
Fix $k\in\N$ and $\alpha\in(0,\infty)$. For $Q\in\mathcal Q_\alpha(\X\to\mathcal Z)$ and $B=(B_j)_{j\in[k]}$ a measurable partition of $\X$, the `projection' $Q_k$ defined in \eqref{eq:ProjQ} is an $\alpha$-differentially private Markov kernel, that is, $Q_k\in\mathcal Q_\alpha([k]\to\mathcal Z)$.
\end{lemma}
\begin{proof}
The Markov property is obvious from the definition and since $[k]$ is a finite set. For $\alpha$-differential privacy, fix $C\in \mathcal G$ and $j,j'\in[k]$. 

If $j,j'\in J$, choose $x'\in B_{j'}$ such that $Q(C|x')\le Q_k(C|j')$, which is possible since the upper bound is equal to the expectation $\E_{P_{j'}}[Q(C|\cdot)]$. Therefore, we have $Q_k(C|j) = \int_\X Q(C|x) P_j(dx) \le e^\alpha Q(C|x') \le e^\alpha Q_k(C|j')$. 

If $j\in J$ and $j'\notin J$, then $Q_k(C|j) = \int_\X Q(C|x) P_j(dx) \le e^\alpha Q(C|x_0) = e^\alpha Q_k(C|j')$. The case $j'\in J$ and $j\notin J$ is analogous.

If both $j,j'\notin J$, the desired inequality is trivially satisfied.
\end{proof}

\begin{lemma}\label{lemma:c-d}
For $\gamma>0$ and $c,d\in\R$, we have
$$
|c^2-d^2| \le (1+\gamma)(c-d)^2 + \frac{c^2\land d^2}{\gamma}.
$$
\end{lemma}
\begin{proof}

Without loss of generality, let $|c|>|d|>0$ and define
$$
f(\gamma) := (1+\gamma)(c-d)^2 + d^2(1+\gamma^{-1}) - c^2.
$$
We need to show that $f(\gamma)\ge0$. For both $\gamma\to0$ and $\gamma\to\infty$ we have $f(\gamma)\to\infty$. Thus, we only need to show that $f$ is non-negative at its minimum. The first order condition for a minimum is
$$
f'(\gamma) = (c-d)^2 -\gamma^{-2}d^2 = 0,
$$
which is achieved at $\gamma^* = \frac{d}{c-d}$. But 
$$
f(\gamma^*) = \left( \frac{c-d}{c-d}+ \frac{d}{c-d}\right)(c-d)^2 + d^2\left(\frac{d}{d} + \frac{c-d}{d}\right) - c^2
= c(c-d) + cd - c^2 = 0,
$$
which finishes the proof.
\end{proof}

\begin{lemma}\label{lemma:pre-postProcessing}
Let $(\mathcal W, \mathcal E)$, ($\X,\mathcal F)$, $(\mathcal Y, \mathcal G)$ and $(\Z,\mathcal H)$ be measurable spaces and consider $R\in\mathfrak P(\mathcal W \to\X)$, $Q\in\mathfrak P(\X\to\mathcal Y)$ and $T\in\mathfrak P(\mathcal Y\to\Z)$. If $Q$ is $\alpha$-private, then so are $QR$ and $TQ$.
\end{lemma}
\begin{proof}
Clearly $QR\in\mathfrak P(\mathcal W\to\mathcal Y)$. Fix $F\in\mathcal F$ and $w,w'\in\mathcal W$ and pick $x'\in\X$ such that $Q(F|x')\le \int_{\X} Q(F|x)R(dx|w')$. To see that this is possible, first consider the case where $\inf_{x\in\X}Q(F|x)$ is attained. In that case we can simply take $x'$ to be the minimizer. Otherwise, we have $\inf_{\bar{x}\in\X}Q(F|\bar{x}) < Q(F|x)$, for all $x\in\X$. By strict monotonicity of the integral, this implies $\inf_{\bar{x}\in\X}Q(F|\bar{x}) < QR(F|w')$. Since the infimum can be approximated to arbitrary precision, we can find the desired $x'$. Therefore, using $\alpha$-privacy of $Q$, we have $QR(F|w) = \int_{\X} Q(F|x)R(dx|w) \le e^\alpha Q(F|x') \le e^\alpha QR(F|w')$.

To establish $\alpha$-privacy of $TQ\in\mathfrak P(\X\to\Z)$, fix $Z\in\Z$, $x,x'\in\X$. Notice that for any $Y\in\mathcal G$, we have $\int_{\mathcal Y} \mathds 1_{Y}(y)Q(dy|x) = Q(Y|x) \le e^\alpha Q(Y|x') = e^\alpha \int_{\mathcal Y} \mathds 1_{Y}(y)Q(dy|x')$ and approximate the integrand $y\mapsto T(Z|y)$ in $TQ(Z|x) = \int_{\mathcal Y} T(Z|y)Q(dy|x)$ by simple functions.
\end{proof}

\begin{lemma}\label{lemma:convexity}
Let $\lambda\in(0,1)$, $x,y>0$ and $a,b\in\R$. Then 
$$
\frac{(\lambda a + (1-\lambda)b)^2}{\lambda x + (1-\lambda)y} \le \lambda \frac{a^2}{x} + (1-\lambda)\frac{b^2}{y}.
$$
\end{lemma}
\begin{proof}
We expand the square and multiply both sides by $\lambda x + (1-\lambda)y$ to obtain, after simplifying, the equivalent inequality
$$
2ab\le a^2\frac{y}{x} + b^2\frac{x}{y}.
$$
Now define $h(a) := a^2\frac{y}{x} + b^2\frac{x}{y} - 2ab$. We want to show that $h(a)\ge 0$ for all $a\in\R$. Since $h'(a) = 2a\frac{y}{x} -2b$ and $h''(a) = 2\frac{y}{x}>0$, we easily see that $h$ is strictly convex and attains its unique minimum at $a^*= b\frac{x}{y}$. Since $h(a)\ge h(a^*)= 0$ the claim follows.
\end{proof}

\begin{lemma}\label{lemma:Qdiff}
For an open subset $\Theta\subseteq\R$ and a measure space $(\X,\mathcal F, \mu)$, let $f:\Theta\times \X\to\R$ be continuously differentiable in its first argument and such that $x\mapsto f(\theta,x)$ and $x\mapsto \dot{f}(\theta,x) := \frac{\partial}{\partial\theta} f(\theta,x)$ are $\mu$-integrable for every $\theta\in\Theta$ and $\theta\mapsto \|\dot{f}(\theta,\cdot)\|_{L_1(\mu)}$ is continuous and finite. Let $q:\X\to\R$ be measurable and bounded. Then $\theta\mapsto \hat{q}(\theta) := \int_\X q(x) f(\theta,x)\mu(dx)$ is differentiable with derivative $\frac{d}{d\theta}\hat{q}(\theta) =  \int_\X q(x) \dot{f}(\theta,x)\mu(dx)$. If, in addition, $\theta\mapsto \dot{f}(\theta,\cdot)$ is continuous as a function from $\Theta$ to $L_1(\mu)$, then the derivative $\theta\mapsto\frac{d}{d\theta}\hat{q}(\theta)$ is also continuous.
\end{lemma}
\begin{proof}
For $h\in\R\setminus\{0\}$, $x\in\X$ and $\theta\in\Theta$, write $g_h(x) := \frac{1}{h}[f(\theta+h,x)-f(\theta,x)] \to \dot{f}(\theta,x)$ as $h\to0$, and note that
$$
\frac{1}{h}[\hat{q}(\theta+h)-\hat{q}(\theta)] = \int_\X q(x)g_h(x)\mu(dx).
$$
By the fundamental theorem of calculus, $g_h(x) = \int_0^1 \dot{f}(\theta+uh,x)du$ and thus $\|g_h\|_{L_1(\mu)} \le \int_0^1 \|\dot{f}(\theta+uh,\cdot)\|_{L_1(\mu)}du$. By assumption, the integral in the upper bound is finite and continuous as a function in $h$ on a closed neighborhood $U$ of $0$ and thus also uniformly bounded on $U$. By the dominated convergence theorem the integral converges to $\|\dot{f}(\theta,\cdot)\|_{L_1(\mu)}$ as $h\to0$. We conclude that $\limsup_{h\to0} \|g_h\|_{L_1(\mu)} \le \|\dot{f}(\theta,\cdot)\|_{L_1(\mu)}$. Thus, Proposition~2.29 in \citet{vanderVaart07} yields $\|g_h- \dot{f}(\theta,\cdot)\|_{L_1(\mu)} \to0$, as $h\to0$, and we obtain
\begin{align*}
\left| \frac{1}{h}[\hat{q}(\theta+h)-\hat{q}(\theta)] - \int_\X q(x) \dot{f}(\theta,x)\mu(dx)\right| &\le 
\int_\X |q(x)| |g_h(x) - \dot{f}(\theta,x)|\mu(dx)\\
&\le \|q\|_\infty \|g_h- \dot{f}(\theta,\cdot)\|_{L_1(\mu)} \xrightarrow[h\to0]{}0.
\end{align*}
For the continuity claim, simply note that
$$
\left| \int_\X q(x) \dot{f}(\theta',x)\mu(dx) - \int_\X q(x) \dot{f}(\theta,x)\mu(dx)\right| \le \|q\|_\infty  \|\dot{f}(\theta',\cdot) - \dot{f}(\theta,\cdot)\|_{L_1(\mu)} \xrightarrow[\theta'\to\theta]{} 0.
$$
\end{proof}

\bibliographystyle{imsart-nameyear}
\bibliography{../../../bibtex/lit}{}
%\bibliography{article}{}

\end{document}